%% file: R3_260107_final_.tex
\documentclass{article}
\usepackage{graphicx} % Required for inserting images

\usepackage{framed}
\usepackage{algorithm}
\usepackage{enumitem}
\usepackage[noend]{algpseudocode}
\usepackage{setspace}
\usepackage{float}
\usepackage{booktabs}
\usepackage{multirow}
\usepackage[dvipsnames]{xcolor}
\usepackage{amsmath,bm}
\usepackage[margin=.94in]{geometry}
\usepackage{soul}
\usepackage{cite}
\usepackage{enumitem}
\usepackage[square,numbers]{natbib}
\usepackage{float}

\usepackage[normalem]{ulem}
\usepackage{pdflscape}
\usepackage{longtable}
\usepackage{adjustbox}
\definecolor{armygreen}{rgb}{0.29, 0.33, 0.13}

\usepackage[colorlinks = true,
            linkcolor = blue,
            urlcolor  = blue,
            citecolor = blue,
            anchorcolor = blue]{hyperref}
\algrenewcommand\algorithmicrequire{\textbf{Input:}}
\algrenewcommand\algorithmicensure{\textbf{Output:}}

\newcommand{\qedflush}{\hfill \scalebox{0.7}{$\blacksquare$}}

\DeclareMathOperator{\dist}{dist}

\definecolor{algemph}{RGB}{0,76,153}  

\newcommand{\psis}{\psi_\mathrm{s}}
\newcommand{\psin}{\psi_\mathrm{n}}

\newcommand{\SA}{\textproc{FP-ADMM}}
\newcommand{\DA}{\textproc{{VP-ADMM}}}
\newcommand{\DAA}{\textproc{{Adapt-ADMM}}}

\newcommand{\SBIPP}{\textproc{{BIPP}}}
\newcommand{\ABIPP}{\textproc{{A-BIPP}}}

\newcommand{\VA}{\textproc{{vADMM}}}
\newcommand{\ADAPTVA}{\textproc{Adapt-vADMM}}
\newcommand{\CONSTVA}{\textproc{Const-vADMM}}

\newcommand{\PE}{\textproc{{Penalty}}}
\newcommand{\ADAPTPE}{\textproc{Adapt-Penalty}}
\newcommand{\CONSTPE}{\textproc{Const-Penalty}}

\newcommand{\ADAP}{\textproc{{ADAP-FISTA}}}

\newcommand{\rev}[1]{{\textcolor{blue}{#1}}}

\newcommand{\mK}{\mathrm{E}}
\newcommand{\Lc}{{\cal L}_c}
\newcommand{\Il}{{\cal I}_\ell}

\input{def}

\title{An Adaptive Proximal ADMM for Nonconvex \\ Linearly Constrained Composite Programs 
\thanks{\textbf{Funding}: 
The first and second authors are supported by the NSF grant \#2410328. The third author has been partially supported by AFOSR Grants FA9550-22-1-0088 and FA9550-25-1-0131. The last author has been supported by CNPq Grants 401864/2022-7 and 306593/2022-0
}}
\author{rm88 }
\date{July 12, 2024 \\ (1st revision: June 30, 2025; 2nd revision: Jan 16, 2026; 3rd revision: April 10, 2026)} 

\begin{document}
\author{Leandro Farias Maia\thanks{School of Mechanical, Industrial, and Manufacturing Engineering, Oregon State University,
Corvallis, OR, 97331. \protect\protect\href{mailto:fariasmaia@gmail.com}{fariasmaia@gmail.com}}\hspace*{0.5em} \and David H. Gutman\thanks{Department of Industrial and Systems Engineering, Texas A\&M University,
College Station, TX, 77843.  
\protect\protect\href{mailto:dhgutman@gmail.com}{dhgutman@gmail.com}} \and Renato D.C. Monteiro\thanks{School of Industrial and Systems Engineering, Georgia Institute of
Technology, Atlanta, GA, 30332-0205. 
\protect\protect\href{mailto:monteiro@isye.gatech.edu}{monteiro@isye.gatech.edu}}\and Gilson N. 
Silva\thanks{Department of Mathematics, Federal University of Piauí, Teresina, PI, 64049-550. \protect\protect\href{mailto:monteiro@isye.gatech.edu}{gilson.silva@ufpi.edu.br}}}

\maketitle

\begin{abstract}
   This paper develops an adaptive proximal alternating direction method of multipliers (ADMM) for solving linearly constrained, composite optimization problems under the assumption that the smooth component of the objective is weakly convex, while the non-smooth component is a convex block-separable function with compact domain. 
   The proposed method is adaptive to all problem parameters, including smoothness and weak convexity constants, and allows
   each of its block proximal subproblems to be inexactly solved.
   Each iteration of our adaptive proximal ADMM consists of two steps: 
   the sequential solution of each block proximal subproblem; and adaptive tests to decide whether to perform a full Lagrange multiplier and/or penalty parameter update(s). Without any rank assumptions on the constraint matrices, it is shown that  the adaptive proximal ADMM obtains an approximate first-order stationary point of the constrained problem in 
   a number of iterations that matches the state-of-the-art complexity for the class of proximal ADMM's. 
   The three proof-of-concept numerical experiments that conclude the paper suggest our adaptive proximal ADMM enjoys significant computational benefits.

\vgap

\end{abstract}

\section{Introduction}
This paper develops an adaptive proximal alternating direction method of multipliers, called {\DAA}, for solving the linearly constrained, smooth, weakly convex, composite optimization problem  
\begin{equation}\label{initial.problem}
    \phi^* = \min_{y\in\R^n} \left\{ \phi(y) := f(y)+h(y)  :   Ay=b \right\},  
\end{equation}
where $A:\R^{n} \to \R^l$ is a linear operator,
$b\in \R^l$ is a vector in the image of $A$, $h$ is a proper closed convex function which is Lipschitz continuous on its  compact domain and, for some positive integer $B$ (the number of blocks) and positive integer vector $(n_1,\ldots,n_t)$ such that $n=\sum_{t=1}^{B} n_{t}$, has the blockwise
representation $h(y)=\sum_{t=1}^{B}h_{t}(y_{t})$
for every $y=(y_{1},\ldots,y_{B})\in\mathbb{R}^{n_{1}}\times\cdots\times\mathbb{R}^{n_{B}}$,
and 
$f$ is a real-valued weakly convex differentiable function on the domain  of $h$ (assumed compact), whose gradient satisfies a blockwise Lipschitz condition. 
Thus, in terms of this blockwise representation,
$f(y)$, $h(y)$, and $Ay$, can be written as
\begin{gather}
\begin{gathered}
% n=\sum_{t=1}^{B} n_{t}, \quad y=(y_{1},\ldots,y_{B})\in\mathbb{R}^{n_{1}}\times\cdots\times\mathbb{R}^{n_{B}},\\ 
f(y)=f(y_{1},\ldots,y_{B}), \quad Ay=\sum_{t=1}^{B}A_{t}y_{t}, \quad h(y)=\sum_{t=1}^{B}h_{t}(y_{t}),
\end{gathered}
\label{eq:block_structure}
\end{gather}
where, for each $t\in\{1,\ldots,B\},$ $h_{t}$
is a proper closed convex function with compact domain and $A_{t}:\R^{n_t} \to \R^l$ is a linear operator. 

The goal in this paper is to find a $(\rho,\eta)$-stationary solution of \eqref{initial.problem}-\eqref{eq:block_structure}, i.e., a quadruple  $(\hat x, \hat p, \hat u,\hat\varepsilon) \in (\dom h) \times A(\mathbb{R}^n) \times \mathbb{R}^l
\times \R_{+}$ satisfying 
\begin{equation}\label{eq:stationarysol}
\hat u \in \nabla f(\hat x) + \partial_{\hat\varepsilon} h(\hat x) + A^*\hat p, \quad \sqrt{\|\hat u\|^2 + \hat\varepsilon}  \leq \rho , \quad \|A\hat x-b\|\leq \eta,
\end{equation}
where $(\rho,\eta)\in \mathbb{R}^2_{++}$ is 
a given tolerance pair and the $\hat \varepsilon$-subdifferential of $h$ is defined in Subsection~\ref{subsec:notation} below.

A popular primal-dual algorithmic framework for solving problem \eqref{initial.problem} that takes advantage of its block structure \eqref{eq:block_structure} is the proximal ADMM, which is based on the augmented Lagrangian (AL) function,
\begin{equation}\label{DP:AL:F}
{\cal L}_{c}(y;p)  :=\phi(y)+\left\langle p,Ay-b\right\rangle +\frac{c}{2}\left\Vert Ay-b\right\Vert ^{2},
\end{equation}
where $c>0$ is a penalty parameter. Given $(\ty^{k-1}, \tilde q^{k-1},c_{k-1})$, the proximal ADMM finds the next triple $(\ty^k,\tilde q^k,c_k)$ as follows. 
Starting from $\ty^{k-1}$,
it first performs $\ell_k$  block inexact proximal
point (\SBIPP) iterations applied to ${\cal L}_{c_{k-1}}(\cdot\,;\,\tilde q^{k-1})$ to obtain $\ty_k$ where $\ell_k$ is a positive integer. 
Next, it performs a Lagrange multiplier update according to 
\begin{equation}
    \tilde q^{k} = (1-\theta)\Big[\tilde q^{k-1}+\chi c_k\left(A\ty^k-b\right)\Big],\label{eq:p_update}
\end{equation}
where $\theta\in [0,1)$  is a dampening parameter and $\chi$ is a positive relaxation parameter, and chooses a scalar $c_k \ge c_{k-1}$ as the next penalty parameter.

We now formally describe how
a proximal ADMM obtains $\tilde y^k$  from $\tilde y^{k-1}$.
It sets $z^0=\tilde y^{k-1}$, and for some positive integer $\ell_k$,
% for every $j=1,\ldots \ell_k$, 
it performs a
{\SBIPP} iteration from $z^{j-1}$ to obtain $z^j$ for every $j=1,\ldots \ell_k$, and finally sets $\tilde  y^k=z^{\ell_k}$.
The {\SBIPP} iteration to obtain $z^j$ from $z^{j-1}$ 
consists of
inexactly solving,
sequentially from $t=1$ to $t=B$,
the $t$-th block  proximal AL subproblem with prox stepsize $\lam_t$ 
\begin{align}\label{eq:x_update}
z_t^{j} \approx \argmin_{u_t\in\R^{n_{t}}}\left\{ \lam_t{\cal L}_{c_{k-1}}(z_{<t}^{j},u_t,z^{j-1}_{>t};\tilde q^{k-1})+\frac{1}{2}\|u_t-z^{j-1}_t\|^{2}\right\} ,
\end{align} 
and finally setting $\ty^k=z^{\ell_k}$.

The recent publication  \cite{KongMonteiro2024} 
proposes a version of 
a proximal ADMM for solving \eqref{initial.problem}-\eqref{eq:block_structure} which assumes that $\ell_k=1$,
$\lam_1=\cdots=\lam_B$, and
$(\chi, \theta) \in (0,1]^2$ satisfies
\begin{equation}\label{eq:assumptio:B}
2\chi B(2-\theta)(1-\theta)\le \theta^2,
\end{equation} 
and hence that $\theta=0$ is not allowed in \cite{KongMonteiro2024}. 

%say explicitly that $\theta=0$ is not allowed

One of the main contributions of \cite{KongMonteiro2024} is that its convergence guarantees do not require \textit{the last block condition}, 
${\rm{Im}}(A_B)\supseteq \{b\}\cup{\rm{Im}}(A_1)\cup\ldots\cup{\rm{Im}}(A_{B-1})$ and $h_B\equiv 0$, 
that hinders many instances of proximal ADMM, see \cite{chao2020inertial,goncalves2017convergence,themelis2020douglas,zhang2020proximal}.
However, the main drawbacks of the proximal ADMM of \cite{KongMonteiro2024} include: 
(i) the strong assumption \eqref{eq:assumptio:B} 
on $(\chi, \theta)$; (ii) subproblem \eqref{eq:x_update}
must be solved exactly;
(iii) the stepsize parameter $\lam$ is conservative and requires the knowledge of $f$'s weak convexity parameter;
(iv) it (conservatively) updates the Lagrange multiplier after each primal update cycle (i.e., $\ell_k=1$);
(v) its iteration-complexity has a high dependence on the number of blocks $B$, namely,
${\cal O}(B^8)$; (vi) its iteration-complexity bound depends linearly on $\theta^{-1}$, and hence grows to infinity as $\theta$ approaches zero.
Paper \cite{KongMonteiro2024} also presents computational results comparing its proximal ADMM with a more practical variant where $(\theta,\chi)$, instead of satisfying  \eqref{eq:assumptio:B}, is set to $(0,1)$. Intriguingly, this $(\theta,\chi)=(0,1)$ regime substantially outperforms the theoretical regime of \eqref{eq:assumptio:B} in the provided computational experiments.  
No convergence analysis for the $(\theta,\chi)=(0,1)$ regime is forwarded in \cite{KongMonteiro2024}. 
Thus, \cite{KongMonteiro2024} leaves open the tantalizing question of whether the convergence of proximal ADMM with $(\theta,\chi)=(0,1)$ can be theoretically secured.\\

\noindent{\bf Contributions:}
This work partially addresses the convergence analysis issue raised above by studying
a {\it completely parameter-free} proximal ADMM, with $(\theta,\chi)=(0,1)$
and $\ell_k$ adaptively chosen, called {\DAA}.  
Rather than making the conservative determination that $\ell_k=1$, the studied adaptive method ensures the dual updates are committed as frequently as possible.
It is shown that {\DAA} finds a $(\rho,\eta)$-stationary solution
in 
${\cal O}(B\max\{\rho^{-3}, \eta^{-3}\})$ iterations. {\DAA} also exhibits the following additional features:

\begin{itemize}%[itemsep=0pt]
\setlength{\itemsep}{0pt}

    \item 
    Similar to the proximal ADMM of \cite{KongMonteiro2024}, its complexity is established without assuming that the {\it last block condition} holds.

    \item Compared to the 
   ${\cal O}(B^8\max\{\rho^{-3},\eta^{-3}\})$ iteration-complexity of the proximal ADMM of \cite{KongMonteiro2024},
    the one for  {\DAA} vastly  {\it improves the dependence on $B$}.

    \item
    {\DAA}  uses a scheme that adaptively computes {\it variable block prox stepsizes}, instead of constant ones
    whose expressions depend on the weakly convex parameters of $f$ as in the proximal ADMM of \cite{KongMonteiro2024}. Specifically, while the method of \cite{KongMonteiro2024} chooses $\lam_1=\ldots=\lam_B \in (0, 1/(2 \bar{m})]$ where $\bar{m}$ is a weakly convex parameter for $f(y)$ relative to the whole $y$, {\DA} adaptively generates  possibly distinct $\lam_t $'s
    that are larger than  $1/(2m_t)$ (and hence $1/(2\bar m)$) where $m_t$ is the weakly convex parameter of $f$ relative to its $t$-th block $y_t$.
    Thus, {\DAA} allows some of (or all) the subproblems \eqref{eq:x_update} to be non-convex.
    
    \item
    {\DAA} is also adaptive to Lipschitz parameters.

    \item 
    In contrast to the proximal ADMM in \cite{KongMonteiro2024},
    {\DAA} allows the block proximal subproblems \eqref{eq:x_update} to be either exactly or {\it inexactly} solved.
\end{itemize}

\noindent \textbf{Related Works}:  
ADMM methods with $B=1$ are well-known to be equivalent to
augmented Lagrangian methods. Several references have studied augmented Lagrangian and proximal augmented Lagrangian methods in the convex (see e.g., \cite{Aybatpenalty,AybatAugLag,LanRen2013PenMet,LanMonteiroAugLag,ShiqiaMaAugLag16, zhaosongAugLag18,necoara2019complexity,Patrascu2017, YangyangAugLag17}) and nonconvex 
(see e.g. \cite{bertsekas2016nonlinear, birgin2,solodovGlobConv,HongPertAugLag, AIDAL, RJWIPAAL2020, NL-IPAL, kong2023iteration, YinMoreau, ErrorBoundJzhang-ZQLuo2020, zhang2020proximal,sun2021dual}) settings.
Moreover, ADMMs and proximal ADMMs in the convex setting have also been broadly studied in the literature (see e.g. \cite{bertsekas2016nonlinear,boyd2011distributed,eckstein1992douglas,eckstein1998operator,eckstein1994some,eckstein2008family,eckstein2009general,gabay1983chapter,gabay1976dual,glowinski1975approximation,monteiro2013iteration,rockafellar1976augmented,ruszczynski1989augmented}).
So from now on, we
just discuss proximal ADMM variants where $f$ is nonconvex and $B > 1$.
% , as the convex case has been extensively explored in the literature (see \cite{bertsekas2016nonlinear,boyd2011distributed,eckstein1992douglas,eckstein1998operator,eckstein1994some,eckstein2008family,eckstein2009general,gabay1983chapter,gabay1976dual,glowinski1975approximation,monteiro2013iteration,rockafellar1976augmented,ruszczynski1989augmented}). 

A discussion of the existent literature on nonconvex proximal ADMM is best framed by dividing it into two different corpora: those papers that assume the last block condition and those that do not. Under the \textit{last block condition}, the iteration-complexity established is ${\cal O}(\varepsilon^{-2})$, where $\varepsilon := \min\{\rho,\eta\}$.  Specifically, \cite{chao2020inertial,goncalves2017convergence,themelis2020douglas, wang2019global} 
introduce proximal ADMM approaches assuming $B=2$, while \cite{jia2021incremental, jiang2019structured,melo2017iterationJacobi,melo2017iterationGauss} present (possibly linearized) proximal ADMMs assuming $B\geq 2$. 
A first step towards removing the last block condition was made by \cite{jiang2019structured} which
proposes an ADMM-type method applied to a penalty reformulation of \eqref{initial.problem}-\eqref{eq:block_structure} that artificially satisfies the last block condition. This method possesses an ${\cal O}(\varepsilon^{-6})$ iteration-complexity bound.

On the other hand, development of ADMM-type methods directly applicable to \eqref{initial.problem}-\eqref{eq:block_structure} is considerably more challenging and only a few works addressing this topic have surfaced. In addition to \cite{jiang2019structured}, earlier contributions to this topic were obtained in   \cite{hong2016convergence,sun2021dual,zhang2020proximal}. More specifically,
\cite{hong2016convergence,zhang2020proximal} develop a novel small stepsize ADMM-type method without establishing its complexity. Finally, \cite{sun2021dual} considers an interesting but unorthodox negative stepsize for its Lagrange multiplier update, that sets it outside the ADMM paradigm, and thus justifies its qualified moniker, ``scaled dual descent ADMM''.

\subsection{Structure of the Paper}\label{sub:1:1}

This subsection outlines this article's structure. 
% This section's lone remaining subsection,
Subsection~\ref{subsec:notation} briefly lays out the basic definitions and notation used throughout. 
Section \ref{subsec:subpro:e} introduces a notion of an inexact solution of {\DAA}'s foundational block proximal subproblem \eqref{eq:x_update} and discusses efficient ways to find said solutions. 

The ADMM variants considered in Sections \ref{adp:ADMM} to \ref{section:VP.ADMM} assume that the weak convexity parameters $m_t$'s are known
and use them to compute their %constant 
prox stepsize,
which is
kept {\it constant} throughout its execution. Specifically, Section~\ref{adp:ADMM} presents a static (i.e., with fixed penalty parameter) 
ADMM variant, {\SA}, and states its main result,
Theorem~\ref{thm:static.complexity},
governing its iteration-complexity. Section~\ref{subsubsec:subroutine} provides the detailed proof of
Theorem~\ref{thm:static.complexity}
and presents all supporting technical lemmas. 
Section~\ref{section:VP.ADMM} presents a non-static (i.e., with variable penalty parameter) ADMM variant, namely {\DA}, and establishes its iteration-complexity in Theorem~\ref{the:dynamic}.

Section~\ref{section:Adaptive.ADMM} presents the centerpiece algorithm of this work, {\DAA}, an {\it adaptive} prox stepsize version of {\DA} that requires no knowledge of the weak convexity parameters $m_t$'s.
 
Section~\ref{sec:numerical} presents proof-of-concept numerical experiments that display the efficiency of {\DAA} for three different problem classes. Section~\ref{sec:concluding} gives some concluding remarks that suggest further research directions. Appendix~\ref{Appendix:tech.lagra.mult} presents some technical results on convexity and linear algebra, while Appendix \ref{sec:acg} describes an adaptive accelerated gradient method and its main properties. Finally, Appendix~\ref{Appendix:inexact.sol} discusses some technical results on the inexact solution notion adopted in this work and its connection to directional derivatives.

\subsection{Notation, Definitions, and Basic Facts }\label{subsec:notation}

This subsection lists the elementary notation deployed throughout the paper. Let $\R$ denote the set of real numbers, and $\R_{+}$  and $\R_{++}$ denote the set of non-negative and positive real numbers, respectively. 
We assume that the $n$-dimensional Euclidean space, $\R^n$, is equipped with an inner product, $\Inner{\cdot}{\cdot}$. 

The norm induced by $\Inner{\cdot}{\cdot}$ is denoted by  $\|\cdot\|$. 
Let $\R^n_{++}$ and $\R^n_{+}$ denote the set of vectors in $\R^n$ with positive and non-negative entries, respectively.
The smallest positive singular value of a  nonzero linear operator $Q:\R^n\to \R^l$ is denoted $\nu^+_Q$ and its operator norm is $\|Q\|:=\sup \{\|Q(w)\|:\|w\|=1\}.$ 
If $S$ is a symmetric and positive definite matrix, the norm induced by $S$ on $\R^n$, denoted by $\|\cdot\|_S$, is defined as $\|\cdot\|_S=\Inner{\cdot}{S(\cdot)}^{1/2}$.
For $x=(x_{1},\ldots,x_{ B}) \in\mathbb{R}^{n_{1}}\times\cdots\times\mathbb{R}^{n_{B}}$, we define the aggregated quantities 
\begin{gather}
\begin{gathered}
x_{<t}:=(x_{1},\ldots,x_{t-1}), \quad  x_{>t}:=(x_{t+1},\ldots,x_{ B}),\quad x_{\leq t}:=(x_{<t},x_{t}),\quad x_{\geq t}:=(x_{t},x_{>t}).
\end{gathered}
\label{eq:intro_agg}
\end{gather}
Moreover, for $a\in \R^B$, we define
\begin{equation}\label{def.lambda.min.max}
\min(a)  = \min_{1\leq t\leq B} a_t  \quad \text{and} \quad \max(a)  = \max_{1\leq t\leq B} a_t .
\end{equation}
For a given closed, convex set $Z \subset \R^n$, we let $\partial Z$ designate its boundary.
The distance of a point $z \in \R^n$ to $Z$, measured in terms of $\|\cdot\|$, is denoted ${\rm dist}(z,Z)$. 
The indicator function of $Z$, denoted by $\delta_Z$, is defined by $\delta_Z(z)=0$ if $z\in Z$, and $\delta_Z(z)=+\infty$ otherwise.
%Let $\log(\cdot)$ denote the logarithm base $2$,  and 
%\rev{For any $s>0$ and $b\ge 0$, let 
% $\log(s):=(\log 2)^{-1}\log s$  
%$\log s/\log 2$  
%and 
%$\log^+_b (s):=\max\{\log_2 s, b\}$. }

For a given function $g :\R^n\to (-\infty,\infty]$, let $\dom g := \{x\in \R^n : g(x) < +\infty\}$ denotes the effective domain of $g$. 
We say that $g$ is proper if $\dom g \ne \emptyset$. The set of all lower semi-continuous proper convex
functions defined in $\R^n$ is denoted by $\bConv{n}$. 
For $\varepsilon\geq 0$, the $\varepsilon$-subdifferential of $g\in \bConv{n}$ at $z\in \dom g$ is 
\begin{equation}\label{e-subd}
\partial_\varepsilon g(z):=\{w\in \R^n ~:~ g(\tilde z)\geq g(z)+\Inner{w}{\tilde z-z}-\varepsilon, \forall \tilde z\in \R^n\}.
\end{equation}
When $\varepsilon=0$, the $\varepsilon$-subdifferential recovers the classical subdifferential, $\partial g(\cdot):=\partial_0 g(\cdot)$. 
 It is well-known (see \cite[Prop.\ 1.3.1 of Ch.\ XI]{lemarechal1993})
 that for any $\beta>0$ and  $g \in \bConv{n}$,
\begin{equation}\label{prop1:subd}
     \partial_{\varepsilon} (\beta g)(\cdot ) = \beta \partial_{(\varepsilon/\beta)} g(\cdot ).
 \end{equation}
Moreover, 
  if
 $ h_i \in \bConv{n_i}$ for $i=1,\ldots,B$
 and $h(y):=\sum_{t=1}^Bh_t(y_t)$ for any $y=(y_1,\ldots,y_B)$, then
 we have
 (see \cite[Remark 3.1.5 of Ch.\ XI]{lemarechal1993})
\begin{align}
\label{prop2:subd}
\partial_{\varepsilon}h(y)  =\cup \{  \partial_{\varepsilon_1}h_1(y_1) &\times \ldots \times
\partial_{\varepsilon_B}h_B(y_B) : \varepsilon_t\ge 0, \;\; \varepsilon_1+\cdots+\varepsilon_B\le \varepsilon\}.
\end{align}

\section{Assumptions and an Inexact Solution Concept}
\label{subsec:subpro:e}

This section contains two subsections.
The first one details a few mild technical assumptions imposed on the main problem \eqref{initial.problem}-\eqref{eq:block_structure}.
The second one introduces  a notion of an inexact stationary point for the block proximal subproblem \eqref{eq:x_update} along with an efficient method for finding such points.

\subsection{Assumptions for Problem~\eqref{initial.problem}-\eqref{eq:block_structure}}\label{subsec:assump}
The main problem of interest in this paper is problem \eqref{initial.problem} with the block structure as in \eqref{eq:block_structure}.
It is assumed that vector $b\in \mathbb{R}^l$,
linear operator
$A:\mathbb{R}^n\rightarrow \mathbb{R}^l$, and functions $f:\mathbb{R}^{n}\rightarrow (-\infty, \infty]$ and  $h_t:\mathbb{R}^{n_t}\rightarrow (-\infty, \infty]$ for $t=1,\ldots, B$, satisfy the following conditions:

\begin{enumerate}

\item[(A1)] $h(\cdot)$ as in \eqref{eq:block_structure} satisfies the following properties:
\begin{itemize}
    \item for every $t=1,\ldots,B$, $h_t(\cdot)\in \bConv{n_t}$ is prox friendly (i.e., its proximal operator is easily computable) and its domain ${\cal H}_t$ is compact;

    \item there exists $M_h\ge 0$ such that $h(\cdot)$ restricted to ${\cal H}:= {\cal H}_1\times\cdots\times {\cal H}_B$  is $M_h$-Lipschitz continuous;
\end{itemize}

% \begin{framed}
% CHECK AGAIN IN THE TEXT THE REFERENCES FOR THE THESE ASSUMPTIONS
% \item[(A1)] 
% % for every $t\in \{1,\ldots,B\}$,
% % function
% $h_t\in \bConv{n_t}$ is prox friendly (i.e., its proximal operator is easily computable) and its domain ${\cal H}_t$ is compact;

% \item[(A2)] 

%  there exists $M_h\ge 0$ such that $h(\cdot)$ as in \eqref{eq:block_structure} restricted to ${\cal H}:= {\cal H}_1\times\cdots\times {\cal H}_B$  is $M_h$-Lipschitz continuous;
% \end{framed}

\item[(A2)]
for some $m=(m_1,\ldots,m_B)\in\R^B_{+}$, function $f$ is block $m$-weakly convex, i.e.,  for every $t\in \{1,\ldots,B\}$,
%\begin{gather} \label{eq:weak_cvx}
\[
f(x_{<t},\cdot,x_{>t}) + \delta_{{\cal H}_t}(\cdot) + \frac{m_{t}}{2}\|\cdot\|^{2}\text{ is convex for all } x\in{\cal H};
\]
%\end{gather}
    \item[(A3)] $f$ is differentiable on ${\cal H}$ and, for every $t\in \{1,\ldots,B-1\}$, there exists $L_{>t}\ge 0$ such that
\begin{gather}
\begin{aligned}
& \|\nabla_{t}f(x_{\le t},\tilde{x}_{>t})-\nabla_{t}f(x_{\le t},{x}_{>t})\|
\leq (L_{>t})\|\tilde{x}_{>t}-x_{>t}\| \qquad  \forall \, x,\tilde{x}\in{\cal H},
\end{aligned}
\label{eq:lipschitz_x}
\end{gather}
where $\nabla_{t} f(\cdot) $ denotes the $t$-th block component of the whole gradient $\nabla f(\cdot) $;

    \item[(A4)] $A$ is nonzero and there exists $\bar{\mathrm{x}}\in \{x\in {\cal H}:Ax=b\}\neq \emptyset$ such that $\bar{d}:=\text{dist}(\bar{\mathrm{x}}, \partial {\cal H})>0$.
    
\end{enumerate}

We now make some remarks about the above assumptions.
First, since ${\cal H}$ is compact by (A1), the scalars
\begin{equation}\label{def:damH}
       D_h:=\sup_{z \in {\cal H}} \|z-\bar{x}\|, \quad \widehat \nabla_f:=\sup_{u \in {\cal H}} \| \nabla f(u)\|, \quad \underline \phi := \inf_{u \in {\cal H}} \phi(u), \quad \overline \phi := \sup_{u \in {\cal H}}  \phi(u),  
\end{equation}
are bounded. 
Second, (A2) allows $m_t$ to be zero for some or all $t \in \{1,\ldots,B\}$.
Finally, it follows from (A4) that
$\mathrm{\bar x}$ is a Slater point for \eqref{initial.problem}
in the strong sense that $\mathrm{\bar x} \in \interior( {\cal H})$ and $A\mathrm{\bar x}=b$.
%has a Slater point and  ${\cal H}$  has finite volume. ??????

% ??????? {\bf Move???} Throughout this paper, we let

\subsection{An Inexact Solution Concept for \eqref{eq:x_update}}\label{subsec:inexact-solution}

This subsection introduces our notion (Definition \ref{def:BB}) of an inexact solution of the block proximal AL subproblem \eqref{eq:x_update}. To cleanly frame this solution concept, observe that \eqref{eq:x_update} can be cast in the form 
\beq \label{ISO:problem}
\min \{\psi(z):=\psis(z)+\psin(z) : z\in\R^n\},
\eeq
where
\begin{equation}\label{def:ps}
\psis(\cdot) =\lam_t \hat {\cal L}_{c}(y_{<t}^{i},\cdot,y_{>t}^{i-1};\tilde q^{k-1})+\frac{1}{2}\|\cdot-y_{t}^{i-1}\|^{2}, \quad \psin(\cdot) = \lam_t h_t(\cdot),
\end{equation}
and, for every $p \in \mathbb{R}^l$,  ${\cal \hat L}_{c}(\cdot\,;p)$ denotes the smooth part of \eqref{DP:AL:F}, i.e.,
\begin{align}\label{def:smooth:ALM}
{\cal \hat L}_{c}(y;p):= f(y)+\left\langle p,Ay-b\right\rangle +\frac{c}{2}\left\Vert Ay-b\right\Vert ^{2}.
\end{align}
Hence, to describe a notion of an inexact solution for \eqref{eq:x_update}, it \rev{suffices} to do so in the  context of \eqref{ISO:problem}. Assume that:
\begin{itemize}
    \item[(B1)] $\psis:\R^n \to \R$ is a differentiable function;
    \item[(B2)] $\psin\in \bConv{n}$.
\end{itemize}

\begin{definition}\label{def:BB}
For a given $z^0 \in \dom \psin$ and parameter
$\sigma \in \R_+$,
a triple $(\bar z, \bar r, \bar \varepsilon)\in \R^n\times\R^n\times\R_{+}$ satisfying
\begin{align}
\bar r \in \nabla \psis(\bar z) + \partial_{\bar\varepsilon} \psin(\bar z) \quad \text{and}\quad
\|\bar r\|^2 + 2 \bar\varepsilon \le 
\sigma \|z^0-\bar z\|^2 \label{ISO:Cond:1&2}  
\end{align}
is called a
$(\sigma;z^0)$-relative stationary solution of \eqref{ISO:problem} with composite term $\psin$.
\end{definition}

We now make some remarks about  Definition \ref{def:BB}. First,  Appendix~\ref{Appendix:inexact.sol} (see Proposition~\ref{Appendix.Prop.C.main} and the remark after it)
discusses how a triple $(\bar z, \bar r, \bar \varepsilon)$ satisfying Definition~\ref{def:BB} yields a point close to $\bar z$ with nearly nonnegative 
directional derivatives along all unit directions.
Second, if $\sigma=0$, then the inequality in \eqref{ISO:Cond:1&2} implies that $(\bar r,\bar\varepsilon)=(0,0)$, and hence the inclusion in \eqref{ISO:Cond:1&2} implies that $\bar z$ is an exact stationary point of \eqref{ISO:problem}, i.e., it satisfies $0\in \nabla \psis(\bar z) + \partial \psin(\bar z)$, or equivalently, the condition that
the directional derivative of 
$\psis+\psin$ at $\bar z$ satisfies $(\psis+\psin)'(\bar z;d) \ge  0$ for every $d \in \R^n$ (see Lemma~\ref{lemma:C:2} in the Appendix). 
Thus, if the triple $(\bar z, \bar r, \bar\varepsilon)$ is a $(\sigma; z^0)$-relative  stationary solution of \eqref{ISO:problem}, then $\bar z$ can be viewed as an approximate stationary solution of \eqref{ISO:problem} where the residual pair $(\bar r, \bar\varepsilon)$
is bounded according to \eqref{ISO:Cond:1&2} (instead of being zero as in the exact case). 
Third, if $\bar z$ is an exact stationary point of \eqref{ISO:problem}
(e.g., $\bar z$ is an
exact solution of
\eqref{ISO:problem}), then the triple $(\bar z,0,0)$ is a $(\sigma; z^0)$-relative  stationary point of \eqref{ISO:problem} for any $\sigma \in \R_+$. 

In general, an exact solution or relative stationary point of \eqref{ISO:problem} is not easy to compute. In such a case,
Proposition 3.5 of \cite{HeMonteiro2015} (see also Subsection 2.3 in \cite{fistaReport2021} and Appendix A in \cite{sujanani2023adaptive})
establishes the iteration-complexity of a variant of the original Nesterov's accelerated gradient method  \cite{Nes83} (see also \cite{nesterov2013gradient} and \cite[Chapter 2]{nesterov1983method})
%\cite{fistaReport2021,YHe2,RJWIPAAL2020,Monteiro2016})
to find such an approximate solution under the assumptions that either $\psis$ is convex or is strongly convex.

Appendix \ref{sec:acg}
describes the {\ADAP} method of \cite{sujanani2023adaptive} and its main properties (see Proposition~\ref{pro.inexact.sol} in Appendix~\ref{sec:acg}).
The main reason for focusing on this method instead of the ones above are:
i) in addition to being able to find an approximate solution as in Definition~\ref{def:BB} when $\psis$ is strongly convex
and its gradient is Lipschitz continuous,
ADAP-FISTA
is also applicable to instances where $\psis$ is weakly convex (and hence possibly nonconvex),
and;
ii)
ADAP-FISTA provides a key
and easy to check inequality whose validity at every iteration guarantees its successful termination.
These two properties of ADAP-FISTA play an important role in the development of adaptive ADMMs with variable prox stepsizes for solving nonconvex instances of \eqref{initial.problem}
(see Section \ref{section:Adaptive.ADMM}).
Finally,
ADAP-FISTA
shares similar features with other accelerated gradient methods (e.g., see \cite{kong2024complexity,fistaReport2021,florea2018accelerated,nesterov1983method,Monteiro2016,nesterov2013gradient,HeMonteiro2015,YHe2}) in that:
it has similar complexity guarantees
regardless of whether it succeeds or fails (e.g., see \cite{kong2024complexity,kong2020efficient,paquette2018catalyst});
it successfully
%\cite{nesterov2015universal}
terminates when
$\psis$ is $\mu$-strongly convex;
and it
performs a line search for estimating a local Lipschitz constant for
the gradient of $\psis$
(e.g., see \cite{florea2018accelerated}).

As mentioned in Subsection \ref{sub:1:1},  Sections \ref{adp:ADMM} to \ref{section:VP.ADMM} assume that the weak convexity parameters $m_t$'s are known
and present ADMM variants  which compute their constant prox stepsize
in terms of the $m_t$'s.

\section{A Fixed Penalty ADMM}\label{adp:ADMM}

This section contains two subsections. The first one describes an important component of an ADMM method, namely, a  subroutine  for performing a block inexact proximal point
(\SBIPP) iteration within it, as mentioned in the paragraph containing \eqref{DP:AL:F} and \eqref{eq:p_update}. The second one presents {\SA}, an  ADMM variant which, in addition to keeping its prox stepsize constant, also keeps its penalty parameter fixed.

\subsection{The {\SBIPP} Subroutine} \label{sub:B-IPP}

This subsection is to state the {\SBIPP} subroutine, its main properties, and relevant remarks about it. 

We start by describing the subroutine.

{\floatname{algorithm}{Subroutine}
\renewcommand{\thealgorithm}{}
\begin{algorithm}[H]
\setstretch{1}
\caption{{\SBIPP}}\label{algo:BlockIPP}
\begin{algorithmic}[1]
\Statex \hskip-1.8em
\Require $(z,p,\lambda,c)
\in {\cal H} \times A(\mathbb{R}^n) 
\times \R^B_{++}\times \R_{++} $
\Ensure $(z^+,v^+,\delta_+)\in {\cal H} \times \mathbb{R}^l 
\times \R_{++}$ 
\vspace{1em}
\\ 
\textcolor{algemph}{STEP 1: Block-IPP Iteration}

  \For{$t=1,\ldots, B$}\label{subroutine:loop.inexact.sol}
    
    \State compute a $\left(1/8;z_t\right)$-relative  stationary solution $(z_t^+,r_t^+,\varepsilon_t^{+})$
         of \label{algo:BIPP.inexact.subp}
    \begin{equation}\label{algo:BIPP.inexact.subproblem}
    \min_{u \in \R^{n_t}} \left  \{ \lam_t \hat {\cal L}_{c}(z_{<t}^+,u,z_{>t};p)+\frac{1}{2}\|u-z_{t}\|^{2}+\lam_t h_t(u) \right\} 
    \end{equation}
    \hspace{1.3em} with composite term $\lam_t h_t(\cdot)$ (see Definition \ref{def:BB})

  \EndFor

  \State $z^+\gets (z_1^+,\ldots ,z_B^+)$\label{subroutine.primal}

\vspace{1em}     
    \\
    \textcolor{algemph}{STEP 2: Computation of the residual pair $(v^+,\delta_+)$ for $(z^+,p)$}

    \For{$t=1,\ldots ,B$}
    \State $ v_t^+\gets \nabla_t f(z_{< t}^+,z_t^+, z_{>t}^+)- \nabla_t f(z_{< t}^+,z_t^+, z_{>t}) + \dfrac{r_t^+}{\lambda_t}+ cA_{t}^{*}\sum_{s=t+1}^{ B}A_{s} (z^+_{s}-z_s) -\dfrac{1}{\lambda_t}( z^+_{t}-z_t)$ \label{algo.stat.vit}\;
    \EndFor
    \State $v^+\gets (v_1^+,\ldots,v_B^+)$ and $\delta_+\gets (\varepsilon_1^+/\lam_1) + \ldots +  (\varepsilon_B^+/\lam_B)$ \label{algo:IPP.v.delta}
    %\vspace{1em}
    \vspace{1em}
    \State \Return $(z^+,v^+,\delta_+)$
   
\end{algorithmic}
\end{algorithm}
\renewcommand{\thealgorithm}{\arabic{algorithm}}
}

We now clarify some aspects of {\SBIPP}.
First, line~\ref{algo:BIPP.inexact.subp} requires a subroutine to find an approximate solution of \eqref{algo:BIPP.inexact.subproblem} as in Definition \ref{def:BB}. A detailed discussion giving examples of such subroutine will be  given  in the second paragraph after Proposition~\ref{main:subroutine:1}. Second, Proposition~\ref{main:subroutine:1} shows that the iterate $z^+$ and the residual pair $(v^+,\delta_+)$ computed in
lines \ref{subroutine.primal} and \ref{algo:IPP.v.delta}, respectively, satisfy the approximate stationary inclusion $v^+ \in \nabla f(z^+) + \partial_{\delta_+} h(z^+) + \textrm{Im}(A^*)$. Hence, upon termination of {\SBIPP}, its output satisfies the first condition in \eqref{eq:stationarysol} (though it may not necessarily fulfill the remaining two) and
establishes an important bound on
the residual pair $(v^+,\delta_+)$ in terms of a Lagrangian function variation
that will be used later to determine a suitable potential function.

We now make some remarks about the prox stepsizes. First, the prox stepsizes $\lam$ \textcolor{blue}{are kept} constant during the whole execution of {\SBIPP}. Hence, ADMMs that repeatedly \rev{invoke} {\SBIPP} yield constant stepsize ADMM variants. Section~\ref{section:Adaptive.ADMM} describes adaptive prox stepsize ADMM variants that repeatedly invoke a version of {\SBIPP} that adaptively chooses the initial stepsizes $\lam$, and outputs $\lam^+$ possibly different from $\lam$.

The quantities 
\begin{gather}
\begin{gathered}
\label{eq:sigma1&2}
\chione := 100 \max
\left\{1 \, , \, \max_{1\leq t\leq B}m_t \right\} +24L^2+1, \qquad \chitwo:=24(B-1)\|A\|^2_{\dagger},
\end{gathered}
\end{gather}
where
\begin{equation}\label{eq:block_norm}
L :=\sqrt{\sum_{t=1}^{B-1} (L_{>t})^2} , \qquad \|A\|_\dagger := \sqrt{\sum_{t=1}^B \|A_t\|^2},
\end{equation}
with scalars $L_{>t}$'s
as in \eqref{eq:lipschitz_x} and
% $t\in \{1,\ldots, B\}$, 
submatrices $A_t$'s as in \eqref{eq:block_structure},
are used in the following statement of the main result of this subsection.

\begin{proposition}\label{main:subroutine:1}
Assume that $(z^+,v^+,\delta_+)={\SBIPP}(z,p,\lam, c)$ for some $(z,p,\lam, c)
\in {\cal H} \times A(\mathbb{R}^n) 
\times \R_{++}^B\times \R_{++}$
%Then,
% \begin{gather}\label{prop.B-IPP.inclusion}
% \begin{gathered}
% v^+  \in \nabla f (z^+)+ \partial_{\delta_+} h(z^+) +
% A^*[p+c(Az^+-b)].
% \end{gathered}
% \end{gather}
%Moreover, 
and the prox stepsize $\lam \in \R^B_{++}$ input to {\SBIPP} is chosen as
\begin{equation}\label{choo:lamt}
\lam_t = \frac1{2\max\{m_t,1\}}\quad \forall t \in \{1,\ldots,B\}.
\end{equation}
Then, the following statements hold:
\begin{itemize}
    \item[(a)] for any $t\in\{1,\ldots,B\}$, the smooth part of the objective function of the $t$-th block subproblem~\eqref{algo:BIPP.inexact.subproblem} is $(1/2)$-strongly convex;
    \item[(b)] it holds that
%     \begin{gather}\label{prop.B-IPP.inclusion}
%  \begin{gathered}
%  v^+  \in \nabla f (z^+)+ \partial_{\delta_+} h(z^+) +
%  A^*[p+c(Az^+-b)],
%  \end{gathered}
%  \end{gather}
% \begin{gather}\label{lemma:norm:residual:Ine}
% \begin{gathered}
% \|v^+\|^2 +\delta_+  
% \le (\zeta_1 + c \zeta_2)
%   \Big[ {\cal L}_c(z;p) - {\cal L}_c(z^+;p)\Big],
% \end{gathered}
% \end{gather}
 \begin{align}
 v^+  \in \nabla f (z^+)+ \partial_{\delta_+} h(z^+) +
 A^*[p+c(Az^+-b)], \label{prop.B-IPP.inclusion} \\
 \|v^+\|^2 +\delta_+  
\le (\zeta_1 + c \zeta_2)
  \Big[ {\cal L}_c(z;p) - {\cal L}_c(z^+;p)\Big], \label{lemma:norm:residual:Ine}
\end{align}
where $\chione$ and  $\chitwo$ are as in  \eqref{eq:sigma1&2}.
\end{itemize}
\end{proposition}

We postpone the proof of Proposition \ref{main:subroutine:1} to the end  of Subsection \ref{subsection:ABIPP} and, for now,
only make some remarks about it.
First, the inclusion~\eqref{prop.B-IPP.inclusion} shows that $(v^+, \delta_+)$ is a residual pair for the point $z^+$. Second, the inequality in \eqref{lemma:norm:residual:Ine} provides a bound on the magnitude of the residual pair $(v^+, \delta_+)$ in terms of a variation of the Lagrangian function
which, in the analysis of the next section, will play the role of a potential function.
% Third, ?????? the prox stepsize selection in \eqref{choo:lamt} is so as to guarantee that \eqref{lemma:norm:residual:Ine} holds and will play no further role in the analyses of the ADMM variants presented in the subsequent sections.
% Fourth, ?????Subsection \ref{Adaptive prox stepsize ADMM}
% discusses an adaptive way of choosing the prox stepsizes, which
% yields a fully adaptive  ADMM variant
% that do not require knowledge of the parameters $m_t$'s. The purpose of the adaptive stepsize selection is to guarantee that a slightly modified version of the inequality in \eqref{lemma:norm:residual:Ine} holds
% (i.e, with a different choice of constant $\zeta_1$), and hence enable the arguments and proofs of the subsequent sections to follow through similarly.

% for ADMM's based on an adaptive stepsize version of {\SBIPP}.

For any given $t \in \{1,\ldots,B\}$, we now comment on the possible ways
to obtain a $(1/8;z_t)$-relative stationary solution of \eqref{algo:BIPP.inexact.subproblem}
as required in
line~\ref{algo:BIPP.inexact.subp} of {\SBIPP}.
As already observed in the paragraph following Definition \ref{def:BB},
if an exact solution $z^+_t$ of \eqref{algo:BIPP.inexact.subproblem}
can be computed in closed form, then
$(z_t^+,v_t^+,\varepsilon_t^+)=(z^+_t,0,0)$ is a $(1/8;z_t)$-relative  stationary solution of \eqref{algo:BIPP.inexact.subproblem}.
Another approach is to use {\ADAP} described in Appendix \ref{sec:acg}.
Specifically, assume that $\nabla_{t} f (x_1,\ldots,x_B)$ exists for every $x=(x_1,\ldots,x_B) \in {\cal H}$ and is
$\tilde L_{t}$-Lipschitz continuous with respect to the $t$-th block $x_t$. Using this assumption and Proposition~\ref{main:subroutine:1}(a),
it follows from statement (c) of Proposition~\ref{pro.inexact.sol} in Appendix \ref{sec:acg} with 
% $(\tau,z^0,L_0,\mu_0)=(1/8,z_t,\lam_tc\|A_t\|^2,1/2)$ and 
$\tilde M =1+\lam_t(\tilde L_t + c \|A_t\|^2 )
$ that {\ADAP} with input  $(\sigma,z^0,M_0,\mu_0)=(1/\sqrt{8},z_t,\lam_tc\|A_t\|^2,1/2)$ obtains a $(1/8;z_t)$-relative  stationary solution of \eqref{algo:BIPP.inexact.subproblem}. Moreover, 
since $M_0 \le \tilde M$,  $\mu_0=1/2$, and $\lam_t \le 1/2$ for every $t \in \{1,\ldots,B\}$,
Proposition~\ref{pro.inexact.sol}(a) ensures that the number of iterations performed by {\ADAP} to obtain such a relative stationary solution is bounded (up to logarithmic terms) by
${\cal O} ( [\tilde L_t + c \|A_t\|^2 ]^{1/2})$.
Even though  this  iteration-complexity bound  is expressed in terms of $\tilde L_t$,  {ADAP-FISTA} itself does not require
$\tilde L_t$.

% It is worth noting that ADAP-FISTA does not require knowledge of the aforementioned Lipschitz constant $\tilde L_t$, which is used only in the expression for the  iteration-complexity bound for ADAP-FISTA.

Finally, as already observed before, {\SBIPP} is a key component that is invoked once in every iteration of the non-adaptive ADMM presented in subsequent subsection. The complexity bounds 
for this ADMM will be given in terms of ADMM iterations (and hence {\SBIPP} calls)
and will not take into account the complexities of implementing line~\ref{algo:BIPP.inexact.subp}.
The main reason for doing so is the possible different ways of solving the block subproblems (e.g., in closed form, or using an ACG variant, or some other convex optimization solver).
Nevertheless, the discussion in the previous paragraph provides ways of estimating the contribution of each block to the overall algorithmic effort.
% Also, even though BIPP
% requires the parameters $m_t$does not require the Lipschitz constant

\subsection{ Description of {\SA} }\label{subsection:SA-ADMM}

This subsection presents {\SA}, an  ADMM variant which, in addition to keeping its penalty parameter constant, also works with a prox stepsize parameter
 determined by the weak convexity parameters $m_t$'s.
 
{\SA} is shown to obtain a quadruple $(\hat{y}, \hat{q}, \hat{v}, \hat{\delta})$ satisfying \eqref{eq:stationarysol}
 only under the condition that its (constant) penalty parameter is larger than a threshold value, which is usually unknown or difficult to estimate.
 Due to this difficulty and other ones discussed in the first two remarks after Theorem 3.2,
 a single call to  {\SA} as a means to find a quadruple as in \eqref{eq:stationarysol}
 is not recommended.
 Instead, the variable penalty ADMM of Section \ref{section:VP.ADMM} is shown to obtain such a quadruple by repeatedly doubling $c$ and warm-starting   
 {\SA}. It is also shown that the iteration-complexity of VP-ADMM is better than that of FP-ADMM under the assumption that its constant penalty parameter is (somehow) suitably chosen.

{\SA} is formally stated next.

\setcounter{algorithm}{0}
\begin{algorithm}[H]
\caption{{\SA} ($m=(m_1,\ldots,m_B)$ is required)}\label{alg:static}
\begin{algorithmic}[1]
\Statex \hskip-1.8em
\textbf{Universal Input:}  $\rho>0$, \rev{$\alpha> 0$}, \rev{$C \geq \rho$}, and $m=(m_1,\ldots,m_B)\in\R^B_+$
\Require $(y^0,q^0, c)
\in {\cal H} \times A(\mathbb{R}^n) 
\times \R_{++}$
\Ensure $(\hat y, \hat q, \hat v, \hat \delta)$ that satisfies the conclusions of (a) and (b) of Theorem~\ref{thm:static.complexity}.
\vspace{1em}

\State $T_0\gets 0$, $k\gets 0$ \label{def:T0}  

\State $\lam_t \gets \dfrac{1}{2\max\{m_t,1\}}$ for $t=1,\ldots ,B$\label{choose:lamt}

\For{$i\gets 1,2,\ldots$} \label{star:cycle}
        \State $(y^i,v^i,\delta_i)= \SBIPP(y^{i-1},q^{i-1},\lam,c)$\label{call:SBIPP}
        \If{$\|v^i\|^2 +\delta_i  \le \rho^2$}\label{cond:vi:stop} \Comment{termination \textcolor{blue}{criterion}}
            \State $k\gets k+1$, \ $i_k\gets i$ \Comment{end of the final epoch}\label{algo:j_k&k.update:1}
            \State  $q^i=  q^{i-1} + c (A y^i - b) $ \label{algo.sa.lastLagrangeMultiplier} 
            \State \Return $(\hat y, \hat q, \hat v, \hat \delta)=(y^i, q^i, v^i, \delta_i)$\label{algo.sa.return}
        \EndIf

\vspace{1em}

    \State $T_i= {\cal L}_c( y^{i-1};q^{i-1}) - {\cal L}_c(y^i;q^{i-1}) + T_{i-1}$\label{def:T:i}
    
    \If{$\|v^i\|^2 +\delta_i\le C^2$ and $\dfrac{\rho^2}{\alpha(k+1)} \ge \dfrac{T_i}{i}  $}\label{stat:begin.test.cond}
        \State $k\gets k+1$, \ $i_k\gets i$\Comment{end of epoch $\Ik$} \label{algo:j_k&k.update}
        \State  $q^{i} =  q^{i-1} + c (A y^i - b) $  \label{def:qi:chi}
 \label{stat:test.i.qi.ti.updates} \Comment{Lagrange multiplier update}

    \Else
        \State $q^i= q^{i-1}$ \label{sa.qi.update}
   
    \EndIf
\EndFor
\end{algorithmic}
\end{algorithm}

We now explain some details about {\SA}. The iteration index $i$ counts the number of iterations
of {\SA}.
Index $k$ counts the number of
Lagrange multiplier updates performed by  {\SA}.
The index $i_k$ computed either in lines \ref{algo:j_k&k.update:1} or \ref{algo:j_k&k.update} of {\SA} is the iteration index where the $k$-th Lagrange multiplier occurs. It is shown in Theorem~\ref{thm:static.complexity}(a) that the total number of iterations performed by {\SA} is finite, and hence that the index $i_k$ is well-defined.
If the inequality in line~\ref{cond:vi:stop} is satisfied, {\SA} performs the last Lagrange multiplier update and stops in line~\ref{algo.sa.return}. Otherwise, depending on
the test in line~\ref{stat:begin.test.cond}, {\SA}
either performs a Lagrange multiplier in line \ref{def:qi:chi}
or leaves it unchanged in line~\ref{sa.qi.update},
and in both cases moves on to the next iteration.

We now comment on the instance parameters required by {\SA}.
First, it requires the weakly convex parameters $\{m_t\}_{t=1}^B$ (see its universal input), as they are used to compute the %(constant) 
proximal stepsizes $\{\lambda_t\}_{t=1}^B$ in line \ref{choose:lamt}, which \textcolor{blue}{remain} constant throughout its execution. However, {\SA}  requires none of the
Lipschitz constant $M_h$ as in assumption (A1), the Lipschitz constants $\{L_{>t}\}_{t=1}^{B-1}$ as in assumption (A3), or the Lipschitz constants $\{\tilde L_t\}_{t=1}^B$ mentioned in the second last paragraph of Subsection~\ref{sub:B-IPP} (see the discussion therein). Subsection \ref{Adaptive prox stepsize ADMM}
presents a fully adaptive  ADMM variant
(i.e., one that  requires none of the above parameters)
which, instead of choosing
the prox stepsizes in a constant manner as in line \ref{choose:lamt} of {\SA}, selects them adaptively.

We next define a few concepts that will be used in the discussion and analysis of {\SA}.
For every $k \ge 1$, define the $k$-th epoch $\Ik$
as the index set 
\begin{equation}
\label{eq:cy}
{\cal I}_k := \{i_{k-1}+1, \ldots, i_k\},
\end{equation}
with the convention that $i_0=0$.
Moreover, let
\begin{equation}\label{def.variablesTilde}
   (\ty^k,\tq^k,\tilde\lam^k, \tT_k):=(y^{i_k},q^{i_k},\lam^{i_k}, T_{i_k}) \quad \forall k\geq 0,
\end{equation}
and
\begin{equation}\label{def.variablesTilde2}
  (\tv^k,\tilde \delta_k):=(v^{i_k},\delta_{i_k}) \quad \forall k\ge 1.
\end{equation}

We now make two additional remarks  about the logic of {\SA} regarding the prox stepsize and the Lagrange multiplier. First, due to the definition of $i_k$, it follows that $q^i=q^{i-1}$ for every $i \in \{i_{k-1}+1,\ldots, i_k-1\}=\Ik\setminus \{i_k\}$, which implies that 
\begin{equation}\label{eq:q.constant}
    q^{i-1}=q^{i_{k-1}}=\tq^{k-1} \ \ \forall i\in \Ik.
\end{equation}
Moreover, \eqref{def.variablesTilde} and \eqref{eq:q.constant} with $i=i_k$
imply that 
\begin{equation}\label{eq:q.update}
    \tq^{k} = q^{i_k}= q^{i_k-1} + c(Ay^{i_k}-b)=\tq^{k-1} +  c(A\ty^k-b) \quad \forall k\geq 1.
\end{equation}
Second, $\tilde q^0 \in \text{Im} (A)$ due to  the facts that  $\tilde q^0=q^0$
by \eqref{def.variablesTilde} with $k=0$,
and the input $q^0$ of  {\SA} is in $\text{Im}(A)$. 
This observation, 
the fact that $b \in \text{Im}(A)$ by (A4), identity \eqref{eq:q.update}, and
a simple induction argument, then imply that
\begin{equation}\label{eq:q.in.Img.A}
   \tilde q^k \in A(\R^n) \quad \forall  k \ge 0, 
\end{equation}
and hence that
\begin{equation}\label{eq:q.in.Img.A.2}
   q^i \in A(\R^n) \quad \forall  i \ge 0,
\end{equation}
in view of~\eqref{eq:q.constant}.

Before stating the main result of this subsection,  we define the quantities
\begin{equation}
\label{eq:def.Upsilon}
{\Upsilon(C):=\frac{2D_hM_h+
(2D_h+1)(C + C^2+ \widehat \nabla_f)}{\bar d \nu^{+}_A}},  %\quad \forall a \in \R,\\ 
\end{equation}
\begin{equation}\label{eq:def.Gamma.Static}
\Gamma(y^0,q^0,c\;;C,\alpha):=\overline{\phi}-\underline{\phi} + c\|Ay^0-b\|^2 + \left[\frac{4 (\chisum)}{\alpha}+ 1\right]\frac{\|q^0\|^2 + \Upsilon^2(C)}{c},
\end{equation}
% -----------------------
% if $\alpha = \Omega(1)$ then
% \[
% \Gamma(???) = {\cal O}(1+ c\|Ay^0-b\|^2)
% \]
% --------------
where $(y^0,q^0, c)$ is the input of {\SA}, $(\chione, \chitwo)$ is as in \eqref{eq:sigma1&2}, $M_h$ and $\bar{d}$ are as in (A1) and (A4), respectively, $(D_h, \widehat \nabla_f)$ is as in \eqref{def:damH},  and  $\nu^{+}_A$ is the smallest positive singular value of the nonzero linear operator $A$.

The main iteration-complexity result for {\SA}, whose
proof is given in Section \ref{subsubsec:subroutine}, is stated next.

\begin{thm}[{\SA} Complexity]\label{thm:static.complexity}
Assume that $(\hat y, \hat q, \hat v, \hat \delta) = {\SA}(y^0,q^0, c)$ for some triple
$(y^0,q^0,c) \in {\cal H} \times A(\mathbb{R}^n)
%\times \R_{++}^B 
\times \R_{++}$. Then, for any
tolerance pair $(\rho,\eta)\in \R^2_{++}$ and $\rev{C\ge \rho}$,
the following statements hold for {\SA}:
\begin{enumerate}
\item[(a)] its total  number of iterations (and hence {\SBIPP} calls) is bounded by  
\begin{align}\label{thm:iter.complexity:SADMM}
 \left(\frac{\chisum}{\rho^2}\right)\Gamma(y^0,q^0, c\,; C,\alpha) + 1,
\end{align}
where  $(\chione, \chitwo)$ and $\Gamma(y^0,q^0, c\,; C,\alpha)$ are as in \eqref{eq:sigma1&2} and \eqref{eq:def.Gamma.Static}, respectively;

\item[(b)] its output  $(\hat y,\hat q, \hat v, \hat \delta)$ satisfies
\begin{equation}\label{residual:theorem:1}
\hat v \in \nabla f(\hat y) + \partial_{\hat \delta} h(\hat y)+A^*\hat q   \quad \text{and}\quad \|\hat v\|^2+\hat \delta \le \rho^2,
\end{equation}
and the following bounds
\begin{equation}\label{feasi:theorem:1}
c\|A\hat y-b\| \le 2\max\{\|q^0\|,\Upsilon(C)\}  \quad \text{and} \quad
\quad \|\hat q\| \le \max\{\|q^0\|,\Upsilon(C)\},
\end{equation}
where $\Upsilon(C)$ is as in \eqref{eq:def.Upsilon};
\item[(c)] if $
   c\geq 2 \max\{\|q^0\|,\Upsilon(C)\}/\eta$,
then the output $(\hat y, \hat q, \hat v, \hat \delta)$ of {\SA}
  is a $(\rho,\eta)$-stationary solution of problem \eqref{initial.problem}-\eqref{eq:block_structure} according to \eqref{eq:stationarysol}. 
\end{enumerate}
\end{thm}

We now make some remarks about Theorem \ref{thm:static.complexity}. First, Theorem \ref{thm:static.complexity}(b) implies that {\SA} returns a quadruple $(\hat y, \hat q, \hat v, \hat \delta)$ satisfying both conditions in \eqref{residual:theorem:1}, but not necessarily the feasibility condition $\|A\hat y-b\|\le \eta$. 
However, Theorem \ref{thm:static.complexity}(c) guarantees that, if $c$ is chosen large enough, i.e., $c = \Omega(\eta^{-1})$,
then the feasibility also holds,
and hence that $(\hat y,\hat q, \hat v,\hat \delta)$ is a $(\rho,\eta)$-stationary solution of \eqref{initial.problem}-\eqref{eq:block_structure}. 

 \rev{Second, if $\alpha=\Omega(1)$ and $C={\cal O}(1)$ then it follows from \eqref{eq:def.Gamma.Static} that $\Gamma(y^0,q^0,c \ ;C, \alpha) = {\cal O}(1+ c\|Ay^0-b\|^2)$, and hence  \eqref{thm:iter.complexity:SADMM} implies  that the overall complexity of {\SA} is 
\[
{\cal O}\left(\frac{1+c}{\rho^2} \left( 1 + c \|Ay^0-b\|^2 \right) \right).
\]}
If the initial point $y^0$ satisfies $c\|Ay^0-b\|^2={\cal{O}}(1)$, then the above bound reduces to  ${\cal{O}}((1+c) \rho^{-2})$. 
Moreover, under the assumption made in  Theorem \ref{thm:static.complexity}(c), i.e., that
$c  = {\Theta}(\eta^{-1})$, then
the above two complexity estimates reduce to 
$ {\cal O}(\eta^{-2}\rho^{-2})$ if $y^0$ satisfies $\|Ay^0-b\|={\cal O}(1)$ and to
${\cal O}(\eta^{-1}\rho^{-2})$ if
$y^0$ satisfies $c\|Ay^0-b\|^2 = {\cal O}(1)$. 

\rev{Third}, it is worth discussing the dependence of the complexity bound~\eqref{thm:iter.complexity:SADMM} in terms of number of blocks $B$ only.
Observe that the only quantity in \eqref{eq:def.Gamma.Static} that depends on $B$ is
 the quantity $\zeta_2$ defined in \eqref{eq:sigma1&2},
 and hence satisfies 
 $\zeta_2={\Theta}(B)$. Thus, $\Gamma(y^0,q^0,
c \ ; C,\alpha) = {\cal O}
(1+ B/\alpha)$ due to \eqref{eq:def.Gamma.Static}, and hence 
the complexity bound~\eqref{thm:iter.complexity:SADMM} is
${\cal O}(B(1+B\alpha^{-1}))$. So, if $\alpha$ is chosen to be
$\alpha=\Omega(B)$, then the dependence of
\eqref{thm:iter.complexity:SADMM}, in terms of $B$ only, is ${\cal O}(B)$.

Section \ref{section:VP.ADMM} presents an ADMM variant, namely {\DA}, which gradually increases the penalty parameter and achieves the complexity
bound $ {\cal O}(\eta^{-1}\rho^{-2})$ of the previous paragraph for any $y^0$ such that $\|Ay^0-b\|={\cal O}(1)$.
Specifically,
{\DA}
 repeatedly doubles $c$ and warm-starts
{\SA}, i.e.,
if $c$ is the  penalty parameter used in the previous {\SA} call and $(\hat x,\hat p, \hat v , \hat\varepsilon)$ is its output, then the current {\SA} call uses 
$(\hat x, \hat p,2c)$
as input.

\section{The Proof of {\SA}'s Complexity (Theorem \ref{thm:static.complexity})}
\label{subsubsec:subroutine}

% This section gives the proof 
% of Theorem~\ref{thm:static.complexity}. 

The  first result of this section shows that every iterate $(y^i,v^i,\delta_i,\lam^i)$ of {\SA} satisfies the stationary inclusion $v^i  \in \nabla f (y^i)+ \partial_{\delta_i} h(y^i) +
\text{Im}(A^*)$
and derives a bound on the residual error
$(v^i,\delta_i)$.

\begin{lemma}\label{lemma:SADMM.BIPP}
Consider the sequence $\{(y^i,q^i,v^i,\delta_i,T_i)\}$ generated by {\SA}.
Then, for every iteration index $i \ge 1$, we have
\begin{align}\label{eq.SA.stat.inclusion}
    v^i  &\in \nabla f (y^i)+ \partial_{\delta_i} h(y^i) +
A^*[q^{i-1} + c(Ay^i-b)]\\
    \frac{1}{\chisum}(&\|v^i\|^2+\delta_i) \le {\cal L}_c( y^{i-1};q^{i-1}) - {\cal L}_c(y^i;q^{i-1}) =  T_i-T_{i-1}\label{claim:bounded:iterations},
\end{align}
where  $(\chione, \chitwo)$ is as in \eqref{eq:sigma1&2}.
\end{lemma}

\begin{proof}
The proofs of both \eqref{eq.SA.stat.inclusion} and \eqref{claim:bounded:iterations} are based on
Proposition~\ref{main:subroutine:1} with $(z,p,\lam,c)= (y^{i-1},q^{i-1},\lam,c)$ and the fact that $(y^i,v^i,\delta_i)= \SBIPP(y^{i-1},q^{i-1},\lam,c)$ due to line \ref{call:SBIPP} of {\SA}. Specifically, \eqref{eq.SA.stat.inclusion} follows from the conclusion \eqref{prop.B-IPP.inclusion}, and \eqref{claim:bounded:iterations} follows from \eqref{lemma:norm:residual:Ine} and the fact that $\lam$ is chosen as in line \ref{choose:lamt} of {\SA}.
\end{proof}

We now make some remarks about Lemma \ref{lemma:SADMM.BIPP}.
First, \eqref{claim:bounded:iterations} implies that:
$\{T_i\}$ is nondecreasing;
and,
if $T_i=T_{i-1}$, then $(v^i,\delta_i)=(\mathbf{0},0)$,
which together with ~\eqref{eq.SA.stat.inclusion}, implies that the algorithm stops in line~\ref{algo.sa.return} with an exact stationary point for problem~\eqref{initial.problem}-\eqref{eq:block_structure}. 
In view of this remark, it is natural to view 
$\{T_i\}$ as a potential sequence.
Second, if
$\{T_i\}$ is bounded, \eqref{claim:bounded:iterations} immediately implies that 
the quantity $\|v^i\|^2+\delta_i$ converges to zero, and hence that $y^i$ eventually becomes a near stationary point for problem~\eqref{initial.problem}-\eqref{eq:block_structure}, again
in view of ~\eqref{eq.SA.stat.inclusion}.
A major effort of our analysis below will be to show that $\{T_i\}$ is bounded.

Among the results of this section, as well as the ones in Section \ref{section:VP.ADMM},
only Lemma~\ref{lemma:SADMM.BIPP}
uses the
fact that $\lam$ is chosen according to \eqref{choo:lamt}. All other results only require that the following weaker condition hold:

\begin{framed}
\noindent
{\bf Condition C1:} there exists 
$(\chione, \chitwo) \in  \R^2_{++}$ such that \eqref{eq.SA.stat.inclusion} and \eqref{claim:bounded:iterations} hold for every $i \ge 1$.
\end{framed}

\noindent
Moreover, the fact that
the pair $(\chione, \chitwo)$ is chosen as in \eqref{eq:sigma1&2}
plays no role in the proofs of these results.
Thus, if condition {\bf C1} holds for some pair $(\chione, \chitwo) \in \R^2_{++}$ other than the one given by \eqref{eq:sigma1&2}, then the conclusions of Theorems \ref{thm:static.complexity} and \ref{the:dynamic} still hold for this pair.
These observations
will be used in Section \ref{section:Adaptive.ADMM} to establish the complexity of an ADMM variant that chooses the prox stepsizes adaptively,
and hence not according to line~\ref{choose:lamt} of {\SA}.

The first result gives an expression for $T_i$ that plays an important role in our analysis.
\begin{lemma}\label{lemma:Ti.expression}
If $i$ is an iteration index generated by {\SA} such that $i\in \Ik$, then
\begin{equation}\label{eq:Ti.value}
T_i = \left [ {\cal L}_c(\ty^{0};\tq^{0}) - {\cal L}_c(y^i;\tq^{k-1}) \right ]+\frac{1}{c}\sum_{\ell=1}^{k-1}\|\tq^\ell-\tq^{\ell-1}\|^2.
\end{equation}   
\end{lemma}

\begin{proof}
We first note that
\begin{align}\label{eq.Ti.sum.1}
T_i -T_1 &= \sum_{j=2}^i (T_j-T_{j-1}) =
\sum_{j=1}^{i-1} (T_{j+1}-T_{j})=
\sum_{j=1}^{i-1}
 \left[ \Lc (y^{j};q^{j}) - \Lc(y^{j+1};q^{j})\right] 
 \nonumber \\
 &= \sum_{j=1}^{i-1} \left[ \Lc (y^{j};q^{j}) - \Lc (y^{j};q^{j-1})\right] +  \sum_{j=1}^{i-1}\left[\Lc(y^{j};q^{j-1}) - \Lc(y^{j+1};q^{j})\right]. 
\end{align}
Moreover, using the definition of $T_i$ with $i=1$ (see line \ref{def:T:i} of {\SA}), the fact that $q^{i-1}=\tq^{k-1}$ due to \eqref{eq:q.constant} and simple algebra, we have
\begin{align}\label{eq.Ti.sum.2}
T_1 &+ \sum_{j=1}^{i-1} \left[\Lc(y^{j};q^{j-1}) - \Lc(y^{j+1};q^{j})\right] = T_1+ \Lc (y^1;q^0) - \Lc(y^i;q^{i-1})\nonumber \\
&= [\Lc(y^0;q^0) - \Lc(y^1;q^0)] + 
[\Lc (y^1;q^0) - \Lc(y^i;\tq^{k-1})] = \Lc(y^0;q^0) - \Lc(y^i;\tq^{k-1}).
\end{align}
% Assume that $j\in \Il$, for some $\ell< k$.
Using the definition of the Lagrangian function (see definition in~\eqref{DP:AL:F}), relations \eqref{eq:q.update} and \eqref{eq:q.constant}, we conclude that for any $\ell \le k$,
\[
\Lc (y^{j};q^{j}) - \Lc (y^{j};q^{j-1}) \overset{\eqref{DP:AL:F}}=
\Inner{Ay^j-b}{q^j-q^{j-1}} \overset{\eqref{eq:q.update}, \eqref{eq:q.constant}}= \left\{
\begin{array}{ll}
0 & \mbox{, if $j\in \Il\setminus \{i_\ell\}$;} \\
\dfrac{\|\tq^{\ell}-\tq^{\ell-1}\|^2}{c} & \mbox{, if $j=i_{\ell}$,}
\end{array}
\right.
\]
and hence that
\begin{equation*}
\sum_{j=1}^{i-1} \left[ \Lc (y^{j};q^{j}) - \Lc (y^{j};q^{j-1})\right] = \frac{1}{ c}\sum_{\ell=1}^{k-1}\|\tq^{\ell}-\tq^{\ell-1}\|^2.
\end{equation*}
Identity~\eqref{eq:Ti.value} now follows by combining the above identity with the ones in \eqref{eq.Ti.sum.1} and \eqref{eq.Ti.sum.2}.
\end{proof}

The next technical  result will be used to establish an upper bound on the first term of the right-hand side of \eqref{eq:Ti.value}.

\begin{lemma}\label{lem:lowerbound.L}
For any given 
$c >0$ and pairs
$(u,p) \in  {\cal H} \times \mathbb{R}^l$ and
$(\tilde u,\tilde p) \in  {\cal H} \times \mathbb{R}^l$, we have
\begin{equation}\label{def:diff:lagrangian}
{\cal L}_{c}(u; p)
- {\cal L}_{c}(\tu ; \tp )
\le \overline{\phi}-\underline{\phi}
+ c\|Au-b\|^2 +\frac{1}{2c}\max\{\|p\|,\|\tilde p\|\}^2
\end{equation}
where  $(\overline \phi,\underline{\phi} )$ is as in \eqref{def:damH}.
\end{lemma}

\begin{proof}
Using the definitions of ${\cal L}_c(\cdot\,;\,\cdot)$ and $\underline \phi$ as in \eqref{DP:AL:F} and \eqref{def:damH}, respectively, we have
\begin{align*}
{\cal L}_{c}(\tu ; \tp)-\underline \phi &\overset{\eqref{def:damH}}\geq  {\cal L}_{c}(\tu ; \tp)-(f+h)(\tu) \\
&\overset{\eqref{DP:AL:F}}=\langle \tp, A \tu-b\rangle +\frac{c}{2}\|A \tu-b\|^2=\frac{1}{2}\left\|\frac{\tp}{\sqrt{c}}+\sqrt{c}(A \tu-b)\right\|^2-\frac{\|\tp\|^2}{2c} \ge -\frac{\|\tp \|^2}{2c}.
\end{align*}
On the other hand, using the definitions of ${\cal L}_{c}(\, \cdot \,;\, \cdot \, )$  and $\bar \phi$ as in \eqref{DP:AL:F} and \eqref{def:damH},
respectively, and the Cauchy-Schwarz inequality, we have
\begin{align*}
{\cal L}_{c}(u; p) - \overline \phi 
&\overset{\eqref{def:damH}}\le {\cal L}_{c}(u; p)-(f+h)(u)  \overset{\eqref{DP:AL:F}}= \langle p, A u -b\rangle +\frac{c\|A u-b\|^2}{2}\\
% ?????????? &\le \|p\|\|A u-b\|+\frac{c}{2}\|A u-b\|^2 
&\le \left( \frac{\|p\|^2}{2c}+\frac{c\|A u-b\|^2}{2} \right) +\frac{c\|A u-b\|^2}{2}=\frac{\|p\|^2}{2c}+c\|A u-b\|^2.%\le \bar \zeta\|A u-b\|+\frac{c}{2}\|A u-b\|^2,
\end{align*}
 Combining the above two relations, we then conclude that \eqref{def:diff:lagrangian} holds.
\end{proof}

The following result shows that the number of epochs, the number of iterations, and the sequence $\{T_i\}$, generated by {\SA} are all suitably bounded.

\begin{prop}\label{lem:number.of.cycles}

The following statements about {\SA} hold:
\begin{itemize}
    \item[(a)] 
    its total
    number ${\mK}$ of epochs is bounded by $\lceil  (\chione+c\chitwo)/\alpha\rceil$
where $\chione$ and $\chitwo$ are as in \eqref{eq:sigma1&2};
\item[(b)]
for every iteration index $i$, 
we have
$T_i \le \Lambda_{\mK}(y^0;c)$;

\item[(c)]
 the number of iterations  performed by {\SA} is bounded by
\begin{equation}\label{total:complexity:E}
1+ \left( \frac{\chisum}{\rho^2} \right) \Lambda_{\mK}(y^0;c),
\end{equation}
\end{itemize} 
where
\begin{equation}\label{pre:Ti}
 \Lambda_{\mK}(y^0;c) := \overline{\phi}-\underline{\phi}
+ c\|Ay^0-b\|^2 +\frac{Q_{\mK}^2}{2c}+\frac{(\chisum)F_{\mK}^2}{c\alpha},
\end{equation}
and
\begin{gather}
\begin{gathered}
\label{eq:def.Qk.Fk}
Q_{\mK} := \max \left\{ \|\tilde q^k\| : k \in \{0,\ldots,{\mK}-1\} \right\},  \\
 F_{\mK} := \left\{\begin{array}{ll}
 0 & \mbox{, if ${\mK} =1$}\\
 \max \left\{ \| \tq^k-\tq^{k-1} \| : k \in \{1,\ldots,{\mK}-1\} \right\} 
  & \mbox{, if ${\mK}\neq 1$}.
 \end{array}
\right.
\end{gathered}
\end{gather}
\end{prop}

\begin{proof}
(a) Assume for the sake of contradiction that {\SA} generates an epoch ${\mathcal{I}}_{K}$ such that $K > \lceil  (\chione+c\chitwo)/\alpha\rceil$, and hence $K \ge 2$. Using the
fact that $i_{K-1}$ is the last index of ${\cal I}_{K-1}$ and
noting the
epoch termination criteria
in line \ref{stat:begin.test.cond} of {\SA}, we then conclude that
$\tilde T_{K-1}/i_{K-1}\leq \rho^2/K$.
Also, since {\SA} did not terminate during epochs ${\cal I}_1,\ldots,{\cal I}_{K-1}$, it follows from its termination criterion in line~\ref{cond:vi:stop} that 
$\|v^i\|^2+\delta_i > \rho^2$ for every
iteration $i \le i_{K-1}$.
These two previous observations, \eqref{claim:bounded:iterations} with $i\in \{1,\ldots ,i_{K-1}\}$, the facts that $T_0=0$ by definition and
$T_{i_{K-1}}=\tilde T_{K-1}$ due to \eqref{def.variablesTilde}, imply that
\begin{align*}
\rho^2 < \frac{1}{i_{K-1}}\sum_{i=1}^{i_{K-1}}(\|v^i\|^2+\delta_i) \overset{\eqref{claim:bounded:iterations}}\leq \frac{\chisum}{i_{K-1}}\sum_{i=1}^{i_{K-1}}(T_{i}-T_{i-1}) =\frac{(\zeta_1+ c \zeta_2)\tilde T_{K-1}}{i_{K-1}}\leq \frac{(\zeta_1+ c \zeta_2)}{\alpha K} \rho^2.
\end{align*}
% where the last inequality above is due to the second inequality in  \eqref{eq:lema1:test} with $k=K$.
Since this inequality and the assumption (for the contradiction)  that  $K >\lceil (\zeta_1+c\zeta_2)/\alpha\rceil$
yield an immediate contradiction, the conclusion of the statement follows.

(b)  Since $\{T_i\} $ is nondecreasing, it suffices to show that $T_i \le \Lambda_{\mK}(y^0;c)$ holds for any $i\in {\cal I}_{\mK}$, where $\mK$ is the total number of epochs of {\SA} (see statement (a)). It follows from the definition of $Q_{\mK}$, and Lemma~\ref{lem:lowerbound.L} with $(u,p)=(\ty^{0};\tq^{0})=(y^0,q^0)$ (due to \eqref{def.variablesTilde}) and $(\tilde u,\tilde p)=(y^i;\tq^{\mK-1})$, that
\[
{\cal L}_c(\ty^{0};\tq^{0}) - {\cal L}_c(y^i;\tq^{\mK-1})\le \overline{\phi}-\underline{\phi}
+ c\|Ay^0-b\|^2 +\frac{1}{2c}Q_{\mK}^2.
\]
Now, using the definition of $F_{\mK}$ as in~\eqref{eq:def.Qk.Fk}, we have that
\[
 \frac{1}{ c}\sum_{k=1}^{\mK-1}\|\tq^k-\tq^{k-1}\|^2\leq  \frac{(\mK-1)}{c}F_{\mK}^2\leq       \frac{ (\chisum)}{c\alpha}F_{\mK}^2,
\]
where the last inequality follows from the fact that $\mK -1 \le (\chisum)/\alpha $  due to statement (a). 
The inequality $T_i \le \Lambda_{\mK}(y^0;c)$ now follows from the two inequalities above and identity \eqref{eq:Ti.value} with $k=\mK$. 

(c) Assume by contradiction that there exists an iteration index $i$ generated by {\SA} such that
\begin{equation}\label{prel.contr:assum:i}
i > \left(\frac{\chisum}{\rho^2}\right)\Lambda_{\mK}(y^0;c) + 1.
\end{equation}
Since {\SA} does not stop at any iteration index smaller than $i$,  the stopping criterion in line \ref{cond:vi:stop} is violated at these iterations, i.e.,
$
\|v^j\|^2 +\delta_j > \rho^2
$ for every $j \le i-1$.
Hence, it follows from \eqref{claim:bounded:iterations}, the previous inequality, the fact that $T_0=0$ due to line \ref{def:T0} of {\SA}, and statement (b) that
\[
\frac{(i-1) \rho^2}{\chisum} < \frac{1}{\chisum}\sum_{j=1}^{i-1}(\|v^j\|^2+\delta_j) \overset{\eqref{claim:bounded:iterations}}\le \sum_{j=1}^{i-1}(T_j-T_{j-1}) = T_{i-1}-T_0 \le T_i \leq  \Lambda_{\mK}(y^0,c),
\]
which contradicts \eqref{prel.contr:assum:i}.
Thus, statement (c)  holds.
\end{proof}

We now make some remarks about Proposition~\ref{lem:number.of.cycles}.
First, Proposition~\ref{lem:number.of.cycles}(a) shows that the number of epochs depends linearly on $c$.
Second, Proposition~\ref{lem:number.of.cycles}(c) shows that the total number of iterations of
{\SA} is bounded
but
the derived bound is given in terms of the quantities $Q_{\mK}$ and $F_{\mK}$ in
\eqref{eq:def.Qk.Fk}, both of which depend on the magnitude of the sequence of generated Lagrange multipliers $\{\tilde q_k : k=1,\ldots,E\}$.
Hence, the bound in~\eqref{total:complexity:E} is algorithm-dependent in that it depends on the sequence $\{\tilde q^k\}$ generated by {\SA}.

In what follows, our goal is to derive a bound on the total number of iterations performed by {\SA} that depends only on the instance of \eqref{initial.problem}–\eqref{eq:block_structure}.
With this goal in mind, we first establish two technical results related to the boundedness of the sequence of Lagrange multipliers generated by {\SA}.
\begin{lemma}\label{lem:previous:q:f}
    Consider the following subset of epoch ${\cal I}_k$, for some $k \ge 1$, defined as 
\begin{equation}\label{def:I:k:C1}
    {\cal I}_k(C) := \{i\in {\cal I}_k:\|v^{i}\|^2+\delta_i\leq C^2\},
\end{equation}
where $C>0$ is part of the input for {\SA} and $(v^i, \delta_i)$ is as in line \ref{algo:IPP.v.delta} of {\SBIPP}. Then, the pair $(\tq^{k-1}, y^i)$ satisfies
\begin{equation}\label{eq:FeasBound.proof.1}
\|\tilde q^{k-1} + c A(y^i-b)\| \le 
\max \{ \|\tilde q^{k-1}\|,\Upsilon(C)\}, \quad \forall i \in \Ik(C),
\end{equation}
where $\Upsilon(C)$ is as in \eqref{eq:def.Upsilon}. As a consequence,
\begin{equation}\label{step:induc:next:result}
  \|\tq^k\|\le  \max \{ \|\tilde q^{k-1}\|,\Upsilon(C)\}.
\end{equation}
\end{lemma}
\begin{proof}
We first argue that \eqref{step:induc:next:result} follows as an immediate consequence of \eqref{eq:FeasBound.proof.1}.
Indeed, noting that the assumption that $C \ge \rho$ (see the input of {\SA}) and the logic of {\SA} imply that $i_k$ always satisfies $\|v^{i_k}\|^2+\delta_{i_k}\leq C^2$, i.e., $i_k \in \Ik(C)$, and using \eqref{eq:FeasBound.proof.1} with $i=i_k$, and identities \eqref{def.variablesTilde} and \eqref{eq:q.update}, we immediately see that \eqref{step:induc:next:result} holds.

To show \eqref{eq:FeasBound.proof.1}, let $i \in \Ik(C)$ be given. The proof of \eqref{eq:FeasBound.proof.1} relies on Lemma~\ref{lem:qbounds-2} with $(q^-,\varrho) = (\tilde q^{k-1},c)$ and $(z,q,r,\delta)=(y^i, \underline{q}^i, r^i, \delta_i)$, where
\begin{align}\label{eq:IkC.bounded.q.r}
\underline{q}^i:=\tilde q^{k-1}+c(Ay^i-b) \quad \text{and} \quad r^i := v^i-\nabla f(y^i).
\end{align}
We first claim that  $(q^-,\varrho) = (\tilde q^{k-1},c)$ and $(z,q,r,\delta)=(y^i, \underline{q}^i, r^i, \delta_i)$ satisfy the assumptions of Lemma~\ref{lem:qbounds-2}, i.e., the inclusion $\tilde q^{k-1} \in A(\R^n)$ and the conditions in \eqref{eq:cond:Lem:A1}. Indeed, the inclusion is due to \eqref{eq:q.in.Img.A}. Moreover, the identity in \eqref{eq:cond:Lem:A1} is due to the first identity in \eqref{eq:IkC.bounded.q.r}. Also,  Lemma \ref{lemma:SADMM.BIPP}(a), and
relations \eqref{eq:q.constant} and \eqref{eq:IkC.bounded.q.r}, imply that
$
r^i \in  \partial_{\delta_i} h(y^i) + A^*\underline{q}^i
$. Thus, the claim holds.

Now, 
noting that the triangle inequality for norms, the second identity in \eqref{eq:IkC.bounded.q.r}, the definitions of $\Ik(C)$ in \eqref{def:I:k:C1} and $\widehat \nabla_f$ in \eqref{def:damH}, imply that
\begin{equation}\label{eq:ri.deltai.C}
   \delta_i + \|r^i\| \overset{\eqref{eq:IkC.bounded.q.r}}{\leq} \delta_i + \|v^i\|+\|\nabla f(y^i)\| \overset{\eqref{def:I:k:C1}}{\le} C^2 + (C + \|\nabla f(y^i)\|) \overset{\eqref{def:damH}}{\le} C^2 + (C + \widehat \nabla_f),
\end{equation}
and using the conclusion of Lemma~\ref{lem:qbounds-2}, we conclude that
\begin{align*}
 \|\underline{q}^i\| &\overset{\eqref{q-bound-2}}{\leq} \max\left\{\| \tilde q^{k-1}\|\, , \,\Xi(\|y^i-\mathrm{\bar x}\|\, ,\, \|r^i\|+\delta_i) \right\}\\
 &\overset{\eqref{eq:ri.deltai.C}}{\leq} \max\left\{\| \tilde q^{k-1}\|\, , \, \Xi(D_h \, , \, C+C^2 + \widehat \nabla_f) \right\} \leq \max\{\|\tq^{k-1}\|\, , \,\Upsilon(C)\},
\end{align*}
where the second inequality is due to \eqref{eq:ri.deltai.C}, the definition of $D_h$ in \eqref{def:damH}, and the fact that the function $\Xi(\cdot,\cdot)$ defined in \eqref{eq:Technical.varphi.def} is non-decreasing, and the last inequality is due to the definition of $\Xi(\cdot \,;\cdot)$ and the definition of $\Upsilon(\cdot )$  in \eqref{eq:def.Upsilon}.
\end{proof}

The proof of Lemma \ref{lem:previous:q:f} clearly uses the assumption that $\dom h$ is bounded (see A1). However, the above proof shows that the conclusion of Lemma \ref{lem:previous:q:f} also holds under the weaker assumption that the sequence $\{y^i\}$ generated by {\SA} is bounded,
and $D_h$ and $\widehat \nabla_f$ are instead defined as $D_h:=\sup_i \|y^i-\bar{\mathrm{x}}\|$ and $\widehat \nabla_f:=\sup_{i} \| \nabla f(y^i)\|$, respectively.

\begin{lemma}\label{bounds:q:f}
For every epoch $k$  generated by {\SA}, we have
\begin{flalign}
\label{eq:Lag.Feas.Bound.2}
&\|\tq^k\| \le \max\{\|q^0\|, \Upsilon(C)\}, \\
&c\|A y^i - b\| \le 2 \max\{\|q^0\|, \Upsilon(C)\}, \quad \forall i \in \Ik(C), \label{lemma:4.6(b)} \\
&\|\tq^k - \tq^{k-1}\| \le 2 \max\{\|q^0\|, \Upsilon(C)\}, \label{bound:f:final}
\end{flalign}
where $\Upsilon(C)$ and  $\Ik(C)$ are as in \eqref{eq:def.Upsilon} and \eqref{def:I:k:C1}, respectively.
\end{lemma}

\begin{proof} 
Inequality \eqref{eq:Lag.Feas.Bound.2} follows by recursively using \eqref{step:induc:next:result} and the fact that $\tq^0=q^0$ due to  \eqref{def.variablesTilde}.  Inequality  \eqref{lemma:4.6(b)} follows from \eqref{eq:FeasBound.proof.1}, the triangle inequality, \eqref{eq:Lag.Feas.Bound.2} with $k=k-1$, and the fact that $\tq^0=q^0$ due to  \eqref{def.variablesTilde}, i.e., 
\begin{align*}
c\|A y^i-b\|
\overset{\eqref{eq:FeasBound.proof.1}}\le \max\{\|\tilde q^{k-1}\|,\Upsilon(C)\}+\|\tq^{k-1}\|
\overset{\eqref{eq:Lag.Feas.Bound.2}}\leq 2\max\{\|\tilde q^{0}\|,\Upsilon(C)\}.
\end{align*}
Inequality \eqref{bound:f:final} follows from identity \eqref{eq:q.update}, the triangle inequality for norms, \eqref{lemma:4.6(b)} with $i=i_{k}$, and the fact that $\tilde y^k = y^{i_{k}}$ due to \eqref{def.variablesTilde}.
\end{proof}

\vspace{1em}

\noindent
{\bf Proof of Theorem~\ref{thm:static.complexity}:} (a) It follows from Proposition \ref{lem:number.of.cycles}(c) that the total number of iterations generated by {\SA} is bounded by the expression in \eqref{total:complexity:E}.
Now, recalling that $\mK$ is the last epoch generated by {\SA}, and using \eqref{eq:def.Qk.Fk}, \eqref{eq:Lag.Feas.Bound.2} and \eqref{bound:f:final}, we conclude that $Q_{\mK} \leq \max\{\|q^0\|,\Upsilon (C)\}$ and $ F_{\mK} \leq 2\max\{\|q^0\|,\Upsilon (C)\}$, and hence that $\Lambda_{\mK}(y^0;c)\leq \Gamma(y^0, q^0, c \ ; C,\alpha)$, 
where $\Lambda_{\mK}(y^0;c)$ and   $\Gamma(y^0, q^0, c \ ; C,\alpha)$ are as in \eqref{pre:Ti} and \eqref{eq:def.Gamma.Static}, respectively.
The conclusion now follows from the two previous observations.

(b) We first prove that 
the inclusion in \eqref{residual:theorem:1} holds. It follows from 
%Lemma~\ref{lemma:SADMM.BIPP}(a) 
\eqref{eq.SA.stat.inclusion} with $i=i_{\mK}$, \eqref{def.variablesTilde} and \eqref{def.variablesTilde2} with $k=\mK$ that
\[
    \tv^{\mK}  \in \nabla f (\ty^{\mK})+ \partial_{\tilde \delta_{\mK}} h(\ty^{\mK}) +
A^*[q^{i_{\mK}-1} + c(A\ty^{\mK}-b)].
\]
Using~\eqref{eq:q.constant} with $i=i_{\mK}$, \eqref{eq:q.update} with $k=\mK$, and the fact that $(\hat y,\hat q, \hat v, \hat \delta)=(\ty^{\mK},\tq^{\mK},\tv^{\mK},\tilde \delta_{\mK})$, we conclude that the inclusion in \eqref{residual:theorem:1} holds.
The inequality in \eqref{residual:theorem:1} follows from the fact that {\SA} terminates in  line~\ref{cond:vi:stop} with the condition $\|\hat v\|^2+\hat \delta = \|v^{i_{\mK}}\|^2+\delta_{i_{\mK}}\le \rho^2$ satisfied.
The first inequality in~\eqref{feasi:theorem:1} follows from \eqref{lemma:4.6(b)} with $i=i_{{\mK}}$ and the fact that $\tilde y^{\mK}= y^{i_{\mK}}$ due to \eqref{def.variablesTilde}.
Finally, the second inequality in \eqref{feasi:theorem:1} follows from  \eqref{eq:Lag.Feas.Bound.2} and the fact that $(y^{i_{\mK}}, q^{i_{\mK}})=(\ty^{\mK}, \tq^{\mK})$ due to \eqref{def.variablesTilde}.

(c) Using the assumption that $
   c\geq 2 \max\{\|q^0\|,\Upsilon(C)\}/\eta$,  statement (b)  guarantees that {\SA} outputs $\hat y= y^{i_k}$ satisfying $\|A\hat y-b\|\leq [2\max\{\|q^0\|,\Upsilon(C)\}]/c\le \eta.$ Hence, the conclusion that $(\hat y,\hat q,\hat v, \hat \delta)=(\tilde y^k, \tilde q^k,\tilde v^k,\tilde\delta_k)$ satisfies \eqref{eq:stationarysol} follows from the previous inequality, the inclusion in \eqref{residual:theorem:1}, and the last inequality in \eqref{feasi:theorem:1}.
\qedflush

\section{A Variable Penalty ADMM}
\label{section:VP.ADMM}

This section describes a variable penalty ADMM, or {\DA} for short, and establishes its iteration-complexity. 
In contrast to {\SA}, which keeps the penalty parameter constant, {\DA} adaptively changes the penalty parameter. The version of {\DA} presented in this section keeps the prox stepsizes constant throughout since it performs multiple calls to the {\SA} 
which, as already observed, also has this same attribute.
An adaptive variant of {\DA}
with variable prox stepsizes is presented in Section~\ref{section:Adaptive.ADMM}. 

{\DA} is formally stated next.

\begin{algorithm}[H]
\setstretch{1}
\caption{{\DA} ($m=(m_1,\ldots,m_B)$ is required)}\label{alg:dynamic}
\begin{algorithmic}[1]

\Statex \hskip-1.8em
\textbf{Universal Input:} $(\rho,\eta) \in \R^2_{++}$, $\rev{\alpha >0}$, $\rev{C\ge \rho}$, and $m=(m_1,\ldots,m_B)\in\R^B_+$
\Require $x^0\in {\cal H}$ 
%and $\gamma^0=(\gamma_1^0,\ldots,\gamma_B^0)\in \R^B_{++}$
\Ensure $(\hat x,\hat p, \hat u, \hat{\varepsilon},\hat c)$ that satisfies the conclusion of Theorem~\ref{the:dynamic}.
\vspace{1em}

\State $p^0=(p^0_1,\ldots,p^0_B) \gets (0,\ldots,0)$ and $c_0\gets 1/[1+\|Ax^0-b\|]$ \label{def:c0:p0} 

\For{$\ell \gets 1,2,\ldots$}
    \State $(x^\ell ,p^\ell,u^\ell,\varepsilon_\ell )=\SA(x^{\ell-1},p^{\ell-1}, c_{\ell-1})$ \label{dyn:stat.call}
    \State $c_{\ell}=2c_{\ell-1}$ \label{dyn:penalty.update}
    \If{$\|Ax^\ell-b\|\leq \eta$}\label{test:DA:ADMM}

        \State \Return $(\hat x,\hat p, \hat u , \hat\varepsilon,\hat c)=(x^\ell,p^\ell, u^\ell, \varepsilon_\ell, c_\ell)$

    \EndIf
\EndFor
\end{algorithmic}
\end{algorithm}

We now make some remarks about {\DA}. 
First,  similar to Algorithm~\ref{alg:static}, {\DA} requires the weakly convex parameters $\{m_t\}_{t=1}^B$ to be known (see its universal input). 
Even though these parameters do not appear in the main body of {\DA}, they are used within every call to {\SA} (see line \ref{dyn:stat.call}) to compute the  prox stepsize $\lambda$ 
(see line \ref{choose:lamt} of {\SA}). Second, similar to Algorithm~\ref{alg:static}, {\DA} requires none of the
Lipschitz constant $M_h$ as in assumption (A1), the Lipschitz constants $\{L_{>t}\}_{t=1}^{B-1}$ as in assumption (A3), or the Lipschitz constants $\{\tilde L_t\}_{t=1}^B$ mentioned in the second last paragraph of Subsection~\ref{sub:B-IPP}. Third,  even though an
initial penalty parameter $c_0$
is prescribed in line~\ref{def:c0:p0} for the sake of analysis simplification, {\DA} can be equally shown to converge for other choices of $c_0$. Fourth, it uses a ``warm-start" strategy for calling {\SA}, i.e.,  after the first call to {\SA},  the input of  any %$\ell$-th 
 {\SA} call is the output  
of the previous {\SA} call.
 
 Lemma~\ref{lemma:translate} below and Theorem~\ref{thm:static.complexity}(b) imply that each  {\SA} call in line \ref{dyn:stat.call} of {\DA} generates a 
quadruple $(x^\ell ,p^\ell,u^\ell,\varepsilon_\ell)$ satisfying the first two conditions in \eqref{eq:stationarysol}, but not necessarily the last one, i.e., the feasibility condition which is tested in line~\ref{test:DA:ADMM}. 
To
ensure that this condition is also  attained,
% will also satisfy the other coensure the feasibility in \eqref{eq:stationarysol}, 
{\DA} doubles the penalty parameter $c_\ell$ (see its line~\ref{dyn:penalty.update}) every iteration.
Since the first inequality in \eqref{feasi:theorem:1} ensures that $\|Ax^\ell-b\| = {\cal O}(1/c_\ell)$,
this penalty update scheme guarantees that the test in line~\ref{test:DA:ADMM} will eventually be satisfied, and
 {\DA} will terminate with a $(\rho,\eta)$-stationary solution of \eqref{initial.problem}-\eqref{eq:block_structure}. 

Before describing the main result, we define the following constant, which appears in the total iteration-complexity,
\begin{equation}\label{eq:def.Gamma.dynamic}
\bar \Gamma(x^0;C,\alpha):=\overline{\phi}-\underline{\phi} +  \frac{8\zeta_2\Upsilon^2(C)}{\alpha} +2\Upsilon^2(C)\left(\frac{4\zeta_1}{\alpha}+9\right)(1+\|Ax^0-b\|),
\end{equation}
where $(\zeta_1,\zeta_2)$ is as in~\eqref{eq:sigma1&2} and $\Upsilon(C)$ is as in \eqref{eq:def.Upsilon}.

Recalling that every {\DA} iteration makes an {\SA} call, the following result 
translates the properties of {\SA} established  in Theorem~\ref{thm:static.complexity} to the context of {\DA}.

\begin{lemma}\label{lemma:translate}
Let $\ell$ be an iteration index of ${\DA}$. Then,  the following statements hold:
\begin{itemize}
    \item[(a)] the sequences $\{(x^{k},p^{k},u^k,\varepsilon_k)\}_{k=1}^\ell$ 
    and $\{c_k\}_{k=1}^{\ell}$ satisfy
\begin{equation}\label{eq:prop:input.station}
u^k \in \nabla f(x^k) + \partial_{\varepsilon_k} h(x^k)+A^*p^k   \quad \text{and}\quad \max_{1 \le k \le \ell} \|u^k\|^2+\varepsilon_k \le \rho^2,
\end{equation}
%the identity $\gamma^{k}=\gamma^0$,
and the following bounds
\begin{equation}
\label{eq:prop:input.bounds}
\max_{1 \le k \le \ell} \|p^{k}\| \le \Upsilon(C) \quad \text{and} \quad \max_{1 \le k \le \ell}c_{k} \|Ax^{k} -b \|
\le 4\Upsilon(C);
\end{equation}

\item[(b)] the number of iterations performed
by the  {\SA} call  within the $\ell$-th iteration of {\DA} (see line~\ref{dyn:stat.call} of {\DA}) is bounded by 
\begin{equation}\label{eq:DA.ell-th.SA.call}
\left(\frac{\zeta_1+c_{\ell-1}\zeta_2}{\rho^2}\right)\bar \Gamma(x^0;C,\alpha) + 1,
\end{equation}
where $\bar \Gamma(x^0;C,\alpha)$ is as in \eqref{eq:def.Gamma.dynamic};

\item[(c)] if $c_{\ell}\geq 4\Upsilon(C)/\eta$
then $(x^\ell, p^\ell, u^\ell, \varepsilon_\ell)$ is a $(\rho,\eta)$-stationary solution of problem \eqref{initial.problem}-\eqref{eq:block_structure}. 
\end{itemize}
\end{lemma}

\begin{proof}
(a)  Using Theorem \ref{thm:static.complexity}(b) with $(y^0,q^0, c)= (x^{k-1},p^{k-1}, c_{k-1})$ and noting line~\ref{dyn:stat.call} of {\DA}, we conclude that for any $k\in\{1,\ldots,\ell\}$, the quadruple $(x^{k},p^{k},u^{k},\varepsilon_{k})$ satisfies \eqref{eq:prop:input.station} and the conditions
\begin{equation}\label{step:induction}
\|p^{k}\| \le \max\{\|p^{k-1}\|, \Upsilon(C)\}, \quad  c_{k-1} \|Ax^{k} -b \|
\le 2\max\{\|p^{k-1}\|, \Upsilon(C)\}.
\end{equation}
A simple induction argument applied to the first inequality in \eqref{step:induction}, with the fact that $p^0 = 0$, show that  the first inequality in \eqref{eq:prop:input.bounds}  holds.
The second inequality in \eqref{step:induction}, the assumption that $p^0=0$, the fact that $c_k=2c_{k-1}$ for every $k \in \{1,\ldots, \ell\}$,
and the first inequality in~\eqref{eq:prop:input.bounds}, imply that the second inequality in \eqref{eq:prop:input.bounds} also holds. 

(b) Theorem \ref{thm:static.complexity}(a) with $(y^0,q^0, c) = (x^{\ell-1},p^{\ell-1},c_{\ell-1})$ implies that the total number of iterations performed
by the  {\SA} call  within the $\ell$-th iteration of {\DA} is bounded by
\[
\left(\frac{\zeta_1+c_{\ell-1}\zeta_2}{\rho^2}\right)\Gamma(x^{\ell-1},p^{\ell-1},c_{\ell-1}; C,\alpha)+1,
\]
where $\Gamma(\cdot,\cdot,\cdot; C,\alpha)$ is as in~\eqref{eq:def.Gamma.Static}. 
Thus, to show 
\eqref{eq:DA.ell-th.SA.call}, it suffices to show that $\Gamma(x^{\ell-1},p^{\ell-1};c_{\ell-1};C,\alpha)\leq \bar \Gamma(x_0;C,\alpha)$.

Before showing the above inequality, we first show that
$c_{\ell-1}\|Ax^{\ell-1}-b\|^2\leq 16\Upsilon^2(C)/c_0$ for every index $\ell$. 
Indeed, this observation trivially holds for $\ell=1$ due to the fact that $c_0=1/(1+\|Ax^0-b\|) \le 1$ (see line~\ref{def:c0:p0} of {\DA}) and the assumption that  $\Upsilon(C)\geq 1$. Moreover, the second inequality in~\eqref{eq:prop:input.bounds} and the fact that $c_{\ell-1}\geq c_0$ show that the inequality also holds for
$\ell>1$, and thus it holds for any $\ell \ge 1$.

Using the last conclusion, the definition of $\Gamma(x^{\ell-1},p^{\ell-1},c_{\ell-1};C,\alpha)$, the fact that $c_{\ell-1}\geq c_0$, and the first inequality in \eqref{eq:prop:input.bounds}, we have
\begin{align*}
\Gamma(x^{\ell-1},p^{\ell-1},c_{\ell-1};C,\alpha)&\leq\overline{\phi}-\underline{\phi} + \frac{16\Upsilon^2(C)}{c_0} + \left[\frac{4\zeta_1}{\alpha c_{0}}+\frac{1}{c_{0}}+\frac{4\zeta_2}{\alpha}\right]\left(\|p^{\ell-1}\|^2 + \Upsilon^2(C)\right)\\
&\overset{\eqref{eq:prop:input.bounds}}\leq\overline{\phi}-\underline{\phi} + \frac{16\Upsilon^2(C)}{c_0} + \left[\frac{4\zeta_1}{\alpha c_{0}}+\frac{1}{c_{0}}+\frac{4\zeta_2}{\alpha}\right]\left(2\Upsilon^2(C)\right)\\
&=\overline{\phi}-\underline{\phi} +  \frac{8\zeta_2\Upsilon^2(C)}{\alpha} +\frac{2\Upsilon^2(C)}{c_0}\left(\frac{4\zeta_1}{\alpha}+9\right)=\bar \Gamma(x_0;C,\alpha),
\end{align*}
where the last identity follows from $c_0=1/(1+\|Ax^0-b\|)$ and the definition of $\bar \Gamma(x_0;C,\alpha)$ in \eqref{eq:def.Gamma.dynamic}.

(c) Assume that  $c_{\ell}\geq 4\Upsilon(C)/\eta$.
This assumption, the first inequality in \eqref{eq:prop:input.bounds}, and the fact that $c_{\ell}=2c_{\ell-1}$, immediately imply that $c_{\ell-1}\geq 2\max\{\|p^{\ell-1}\|,\Upsilon(C)\}/\eta$. The statement now follows from the previous observation, line~\ref{dyn:stat.call} of {\DA}, and Theorem~\ref{thm:static.complexity}(c) with $(y^0,q^0,c) = (x^{\ell-1},p^{\ell-1},c_{\ell-1})$.
% and $(\hat y,\hat q,\hat v,\hat \varepsilon, \hat\lam)=(x^{\ell},p^{\ell},v^{\ell},\varepsilon_{\ell},\lam^\ell)$.
\end{proof}
% \begin{framed}
% \[
% c_0 f_0^2 = \frac{c_0^2 f_0^2}{c_0} \le \frac{\tau^2}{c_0} = \tau(1+f_0)
% \]
% \[
% c_0 f_0 \le \tau
% \]
% For example,
% \[
% c_0 = \frac{\tau}{1+f_0}
% \]
% \[
% c_{k} f_k^2 = \frac{c_k^2 f_k^2}{c_k} \le \frac{(c_kf_k)^2}{c_0} \le \frac{16 \kappa^2(C)}{c_0} = \frac{16(1+f_0) \kappa^2(C)}{\tau}
% \]

% --------------------------------------

% For authors:
% \begin{align*}
% \Gamma(x^{\ell-1},p^{\ell-1};c_{\ell-1})&=\overline{\phi}-\underline{\phi} + c_{\ell-1}\|Ax^{\ell-1}-b\|^2 + \left[\frac{4\chi (\sigma_1+\sigma_2c_{\ell-1})}{\alpha c_{\ell-1}}+ \frac{1}{c_{\ell-1}}\right]\left(\|p^{\ell-1}\|^2 + \kappa^2(C)\right)\\
% &\leq \overline{\phi}-\underline{\phi} + \frac{16\kappa^2(C)}{c_0} + \left[\frac{4\chi (\sigma_1+\sigma_2c_{0})}{\alpha c_{0}}+ \frac{1}{c_{0}}\right] \left(2\kappa^2(C)\right)\\
% &\leq \overline{\phi}-\underline{\phi} + \frac{4 \chi \sigma_2\kappa^2(C) }{\alpha} +
% \frac{2\kappa^2(C)}{c_0} \left[\frac{4\chi \sigma_1}{\alpha}+ 9 \right] = \Gamma \\
% &\leq \overline{\phi}-\underline{\phi} + \frac{4 \chi \sigma_2\kappa^2(C) }{\alpha} +
% 2\kappa^2(C)\left[\frac{4\chi \sigma_1}{\alpha}+ 9 \right](1+f_0) = \Gamma 
% \end{align*}
% --------------------------------------
% \end{framed}

The next result describes the iteration-complexity of {\DA} in terms of total ADMM iterations (and hence {\SBIPP} calls within {\SA}). 

\begin{theorem}[{\DA} Complexity]\label{the:dynamic}
The following statements about {\DA} hold:
   \begin{itemize}
       \item[(a)]
       it obtains a $(\rho,\eta)$-stationary solution of \eqref{initial.problem}-\eqref{eq:block_structure} in no more than $
   \log_2 \left[1+4\Upsilon(C)/(c_0\eta) \right]+1$ calls to {\SA};
   \item[(b)]
   its total number of {\SA} iterations (and hence {\SBIPP} calls within {\SA}) is bounded by
\begin{equation*}%\label{com.AADMM}
\frac{\zeta_2\bar \Gamma(x^0;C,\alpha)}{\rho^2}\left[ \frac{8\Upsilon(C)}{\eta} + c_0\right] +
\left[ 1+\frac{\zeta_1\bar \Gamma(x^0;C,\alpha)}{\rho^2}\right]\log_2\left( 2+ \frac{8\Upsilon(C)}{c_0\eta}\right), 
\end{equation*}
\end{itemize} 
where $(\zeta_1,\zeta_2)$, $\Upsilon(C)$   and $\bar \Gamma(x^0;C,\alpha)$ are as in~\eqref{eq:sigma1&2}, \eqref{eq:def.Gamma.Static} and \eqref{eq:def.Gamma.dynamic}, respectively, and $c_0$ is as in line \ref{def:c0:p0} of {\DA}.
\end{theorem}
\begin{proof}
(a)
Assume for the sake of contradiction that {\DA} generates an iteration index $\hat \ell$ such that $\hat \ell  > 1 + \log_2 \left[1+4\Upsilon(C)/(c_0\eta) \right] > 1$, and hence 
\[
    c_{\hat \ell-1} = c_0 2^{\hat \ell-1} > c_0\left(1+\frac{4 \Upsilon(C)}{c_0\eta} \right) > \frac{4 \Upsilon(C)}{\eta}.
\]
Using Lemma~\ref{lemma:translate}(c) with $\ell=\hat\ell-1 \ge 1$, we conclude that the quadruple $(x^{\hat\ell-1},p^{\hat\ell-1},u^{\hat\ell-1},\varepsilon_{\hat\ell-1})$ is a $(\rho,\eta)$ stationary solution of problem \eqref{initial.problem}-\eqref{eq:block_structure}, and hence satisfies $\|Ax^{\hat\ell-1}-b\| \le \eta$. In view of line~\ref{test:DA:ADMM} of {\DA}, this implies that {\DA} stops at the $(\hat \ell-1)$-th iteration, which hence contradicts the fact  that $\hat \ell$ is an iteration index. We have thus proved that (a) holds.

(b) Let $\tilde \ell$ denote the total number of {\SA} calls and observe that $\tilde{\ell}\le1+\log_2[1+4\Upsilon(C)/(c_0\eta)]$ due to (a). 
Now, using Lemma~\ref{lemma:translate}(b) and the previous observation, we have that the overall number of iterations performed by {\SA} is bounded by
\begin{align*}
\sum_{\ell=1}^{\tilde \ell}\left[ \left(\frac{\zeta_1+c_{\ell-1}\zeta_2}{\rho^2}\right)\bar \Gamma(x^0;C,\alpha) + 1 \right] &=\left[ 1+\frac{\zeta_1\bar \Gamma(x^0;C,\alpha)}{\rho^2}\right] \tilde \ell + \frac{\zeta_2 \bar \Gamma(x^0;C,\alpha)}{\rho^2}\sum_{\ell=1}^{\tilde \ell}c_{\ell-1} \\
&\leq \left[ 1+\frac{\zeta_1\bar \Gamma(x^0;C,\alpha)}{\rho^2}\right]\tilde \ell + \frac{c_0\zeta_2\bar \Gamma(x^0;C,\alpha)}{\rho^2}\left( 2^{\tilde \ell} -1\right)\\
&\leq \left[ 1+\frac{\zeta_1\bar \Gamma(x^0;C,\alpha)}{\rho^2}\right]\tilde \ell + \frac{c_0\zeta_2\bar \Gamma(x^0;C,\alpha)}{\rho^2}\left( 1+\frac{8\Upsilon(C)}{c_0\eta} \right).
\end{align*}
The result now follows by using that $\tilde{\ell}\le1+\log_2[1+4\Upsilon(C)/(c_0\eta)]$.
\end{proof}

We now make some comments about Theorem~\ref{the:dynamic}.
First, it follows from
Theorem~\ref{the:dynamic}(a) that the final penalty parameter generated by {\DA} is ${\cal O}(\eta^{-1})$.
Second, it follows from Theorem~\ref{the:dynamic}(a)  that {\DA} ends with a $(\rho,\eta)$-stationary solution of \eqref{initial.problem}-\eqref{eq:block_structure} by calling
{\SA} (and hence doubling the penalty parameter) no more than ${\cal O} (\log_2(\eta^{-1}))$ times.
Third, under the mild assumptions that $\|Ax^0-b\| = {\cal O}(1)$, \rev{$\alpha=\Omega(1)$, and $C={\cal O}(1)$,} Theorem~\ref{the:dynamic}(b) and the fact that $\zeta_2=0$ when $B=1$ (see \eqref{eq:sigma1&2}), imply that the complexity of {\DA}, in terms of the tolerances only, is:
\begin{itemize}
    \item 
    $\tilde {\cal O} (\rho^{-2}\eta^{-1})$ if $B>1$, and thus
$\tilde {\cal O}(\epsilon^{-3})$;
\item 
$\tilde {\cal O} (\rho^{-2})$ if $B=1$, and thus
$\tilde {\cal O}(\epsilon^{-2})$,
\end{itemize}
where $\epsilon:=\min\{\rho, \eta\}$.
On the other hand,
{\SA} only achieves the above complexities
with (a generally non-computable)
$c = \Theta(\eta^{-1})$ and with the condition that
$c\|Ax^0-b\|^2 =
{\cal O}(1)$ (see the first paragraph following Theorem~\ref{thm:static.complexity}), or equivalently,
$\|Ax^0-b\|= {\cal O}(\eta^{1/2})$, and hence the initial point $x^0$ being nearly feasible.
Finally, the above complexity for $B=1$ is similar to those derived for some AL methods in terms of tolerance dependencies (e.g., see \cite[Theorem 2.3(b)]{kong2023iteration}, \cite[Proposition 3.7(a)]{sujanani2023adaptive}, and  \cite{zhang2020proximal,ErrorBoundJzhang-ZQLuo2020}).
 
% \begin{framed}
%     \[
%     {\cal O}\left(
%     \left[1+\frac{m_f\kappa_d}{\rho^2}\right]\sqrt{1+\lam_0(L_f+c\|A\|^2)}\left[\log_1^+\left(1+\lam_0(L_f+c\|A\|^2)+\frac{c}{c_0}\right)\right]^2
%     \right)
%     \]
%     If $d_l=\Theta(1)$ then the above complexity becomes
%     \[
%     {\cal O}(\rho^{-3}+\rho^{-2}\eta^{-1/2})
%     \]
%     If $d_l=\Theta(c_l)$ then the complexity becomes
%     \[
%     {\cal O}(\rho^{-5/2}+\rho^{-2}\eta^{-1/2})
%     \]
% \end{framed}

% \begin{framed}
%  Important Comments:
%  \begin{itemize}
%      \item Even if you start with a bad start point
%      \item Complexity depends on $f_0$ and {\DA} is 3. {\SA} is 4.
%      \item about initialization $c_0$ is big $1/\eta$
%  \end{itemize}
% \end{framed}

\section{An adaptive prox stepsize {\DA}}\label{section:Adaptive.ADMM}

This section describes an adaptive prox stepsize 
 version of {\DA}, referred to as {\DAA}, which requires no knowledge of the \textcolor{blue}{weak convexity} parameters $m_t$'s. 
It contains four subsections. Subsection~\ref{subsection:ABIPP} describes an adaptive 
prox stepsize version of {\SBIPP},
referred to as {\ABIPP}. 
Subsection~\ref{Adaptive prox stepsize ADMM} 
presents {\DAA},
whose description invokes {\ABIPP}
instead of {\SBIPP}, and shows that its iteration-complexity is similar to its
constant stepsize {\DA} analog.
Subsection \ref{sub:6.4} provides further details about the implementation of {\ABIPP}. 
 Finally, Subsection \ref{sub:section:6.3} proves two main technical results stated in Subsection~\ref{subsection:ABIPP}.

\subsection{{\ABIPP}: A variable prox stepsize version of {\SBIPP}}\label{subsection:ABIPP}

This subsection describes {\ABIPP}
and states two main results.
The first one
shows that {\SBIPP} is a special case of {\ABIPP}  when the prox stepsize input to the former is sufficiently small. The second one, which is a generalization of Proposition \ref{main:subroutine:1} to the {\ABIPP} context, states the main properties
of {\ABIPP}. The subsection ends with a proof of Proposition~\ref{main:subroutine:1}, which follows as a consequence of these two results.

% This subsection describes {\ABIPP}
% and states its main properties
% in a result (see Proposition~\ref{prop:ABIPP.output} below) that can be viewed as a 
% generalization of
% Proposition \ref{main:subroutine:1} to the {\ABIPP} context.
% It also shows that {\SBIPP} is a special case of {\ABIPP} when the prox stepsize input to it is sufficiently small (see Proposition~\ref{prop:ABIPP.BIPP.same} below).

{\floatname{algorithm}{Subroutine}
\renewcommand{\thealgorithm}{}
\begin{algorithm}[H]
\setstretch{0}
\caption{{\ABIPP}}\label{algo:A-BlockIPP}
\begin{algorithmic}[1]
\Statex \hskip-1.8em
\Require $(z,p,\lam, c)
\in {\cal H} \times A(\mathbb{R}^n) 
\times \R^B_{++}\times \R_{++} $
\Ensure $(z^+,v^+,\delta_+,\lam^+)\in {\cal H} \times \mathbb{R}^l 
\times \R_{++} \times \R^B_{++}$

  \For{$t=1,\ldots, B$}\label{algo:ABIPP.loop.inexact.sol}

    \State $\lambda_t^+ \gets \lambda_t$\label{lam+=lam}
    \State compute a $\left(1/8;z_t\right)$-relative  stationary solution $(z_t^+,r_t^+,\varepsilon_t^{+})$ of \begin{equation}\label{algo:BIPP.inexact.subproblem.2}
    \min_{u \in \R^{n_t}} \left  \{ \lam_t^+ \hat {\cal L}_{c}(z_{<t}^+,u,z_{>t};p)+\frac{1}{2}\|u-z_{t}\|^{2}+\lam_t^+ h_t(u) \right\} 
    \end{equation}
    \hspace{1.3em} with composite term $\lam_t^+ h_t(\cdot)$ (see Definition \ref{def:BB})  \label{output:step:2:ABIPP}
    %\vspace{1em}
\If{$z_t^+$ does \textbf{NOT} satisfy 
\begin{equation}\label{eq:test:adap:new}
{\cal L}_c(z_{<t}^+,z_{t}, z_{>t}; p)-{\cal L}_c(z_{<t}^+,z_t^+ ,z_{>t}; p)  \ge 
 \frac{1}{8\lam_t^+} \|z_t^+-z_t\|^2 + \frac{c}{4} \|A_t (z_t^+-z_t)\|^2
 \end{equation}\hskip1.3em}\label{algo:ABIPP.if.loop}
 \State $\lambda_t^+ \gets \lambda_t^+/2$  and go to line \ref{output:step:2:ABIPP}\label{eq:test:ABIPP}
\EndIf
\EndFor
\vspace{1em}
\State $z^+\gets (z_1^+,\ldots ,z_B^+)$ and
  $\lam^+\gets (\lam_1^+,\ldots ,\lam_B^+)$\label{algo:ABIPP.primal}

\For{$t=1,\ldots ,B$}
    \State $ v_t^+\gets \nabla_t f(z_{< t}^+,z_t^+, z_{>t}^+)- \nabla_t f(z_{< t}^+,z_t^+, z_{>t}) + \dfrac{r_t^+}{\lambda_t^+}+ cA_{t}^{*}\left[\sum_{s=t+1}^{ B}A_{s} (z^+_{s}-z_s)\right] -\dfrac{1}{\lambda_t^+}( z^+_{t}-z_t)$ \label{algo.stat.vit:ADMM}\;
    \EndFor
    \State $v^+\gets (v_1^+,\ldots,v_B^+)$ and $\delta_+\gets (\varepsilon_1^+/\lam_1^+) + \ldots +  (\varepsilon_B^+/\lam_B^+)$ \label{algo:IPP.v.delta:ADMM}
    %\vspace{1em}
    \State \Return $(z^+,v^+,\delta_+,\lam^+)$
   
\end{algorithmic}
\end{algorithm}
\renewcommand{\thealgorithm}{\arabic{algorithm}}
}

We now make two remarks about the way  {\ABIPP} computes the prox stepsize $\lam_t^+$ for every $t \in \{1,\ldots, B\}$. 
First, in contrast to {\SBIPP} which keeps $\lam_t$ constant,
{\ABIPP} adaptively searches for a suitable $\lam^+_t$ in the loop consisting of lines \ref{lam+=lam} to \ref{eq:test:ABIPP}, 
referred to as the $t$-th {\ABIPP} loop in our discussion below.

Each iteration of the $t$-th {\ABIPP} loop  halves $\lam_t^+$ and the loop
 terminates when a prox stepsize  $\lam_t^+$ satisfying \eqref{eq:test:adap:new} is generated.
 Second, if $(z^+,v^+,\delta_+, \lam^+)={\ABIPP}(z,p,\lam, c)$ and
 $\lam^+=\lam$, then
$(z^+,v^+,\delta_+) = {\SBIPP}(z,p,\lam,c)$; hence, if the initial and final prox stepsizes of {\ABIPP} are the same, then both {\SBIPP} and {\ABIPP} are equivalent.

The following result shows that {\SBIPP} is a special case of {\ABIPP}  when the prox stepsize input to the former is sufficiently small.

\begin{proposition}\label{prop:ABIPP.BIPP.same}
Assume that $1/\lam_t \in [2m_t,\infty)$ for every $t=1,\ldots, B$ and $(z^+,v^+,\delta_+)={\SBIPP}(z, p,\lambda,c)$ for some $(z,p,\lam,c)
\in {\cal H} \times A(\mathbb{R}^n) 
\times \R^B_{++}\times \R_{++}$. 
Then, the following statements hold:
\begin{itemize}
    \item[(a)] for each $t\in\{1,\ldots,B\}$, the smooth part of the objective function of the $t$-th block subproblem \eqref{algo:BIPP.inexact.subproblem.2} is $(1/2)$-strongly convex;

    \item[(b)] $(z^+,v^+,\delta_+, \lambda)={\ABIPP}(z, p,\lambda,c)$.
\end{itemize} 
\end{proposition}

Proposition~\ref{prop:ABIPP.BIPP.same}(b) shows that the reverse of the implication in the second remark preceding Proposition~\ref{prop:ABIPP.BIPP.same}  holds if $\lam$ input to {\SBIPP} is sufficiently small.

The following result,
which is a more general version of Proposition \ref{main:subroutine:1},
describes the main properties of the quadruple
$(z^+,v^+,\delta_+,\lam^+)$ output 
by {\ABIPP}.

\begin{proposition}\label{prop:ABIPP.output}
For a given $(z,p,\lam,c)
\in {\cal H} \times A(\mathbb{R}^n) 
\times \R^B_{++}\times \R_{++}$, the following statements about {\ABIPP} with input $(z,p,\lam, c)$ hold:

\begin{itemize}
\item[(a)] 
 for every $t \in \{1,\ldots,B\}$,  the $t$-th {\ABIPP} loop stops  in at most $1+\lceil \log_2(1+2m_t\lam_t) \rceil$ iterations, and as a consequence {\ABIPP} terminates; moreover, the prox stepsize $\lam^+$
output by {\ABIPP} satisfies
\begin{equation}\label{eq:ABIPP:output.lambda.bound.2}
\lam_t^+ \le \lam_t, \quad \max \left\{ \frac{1}{\lambda_t}, 4 m_t \right\} = 
 \max \left\{ \frac{1}{\lambda_t^+}, 4 m_t \right\}
  \quad \forall t=1,\ldots,B;
\end{equation}

\item[(b)] if $(z^+,v^+,\delta_+, \lam^+)={\ABIPP}(z,p,\lam, c)$, then  inclusion \eqref{prop.B-IPP.inclusion} holds and
\begin{gather}
\begin{gathered}\label{lemma:norm:residual:Ine.2}
\|v^+\|^2 +\delta_+  
\le \left[1+\frac{50}{\min(\lam^+)}+48 L^2\max(\lam^+) + c \zeta_2\right]
  \Big[ {\cal L}_c(z;p) - {\cal L}_c(z^+;p)\Big],
\end{gathered}
\end{gather}
where $\chitwo$ is as in  \eqref{eq:sigma1&2}
and $L$ is as in \eqref{eq:block_norm}.
%, and
% \begin{align}\label{def.lambda.min.max}
% \lam_{\min}^+ := \min_{1\leq t\leq B}\{ \lam_t^+\}, \quad \lam_{\max}^+ := \max_{1\leq t\leq B}\{ \lam_t^+\};
%\end{align}
\end{itemize}
\end{proposition}

% ?????? It follows from 60 the invariance...
% \begin{equation}
%  \max \left\{ \frac{1}{\lambda_t}, 4 m_t \right\} = 
%  \max \left\{ \frac{1}{\lambda_t^+}, 4 m_t \right\}
%   \quad \forall t=1,\ldots,B;
% \end{equation}

We now use Propositions \ref{prop:ABIPP.BIPP.same} and \ref{prop:ABIPP.output} 
to prove Proposition \ref{main:subroutine:1}.

\vspace{.4em}

\noindent
{\bf Proof of Proposition \ref{main:subroutine:1}:}
    Assume that $(z^+,v^+,\delta_+)={\SBIPP}(z, p,\lam,c)$ and $\lambda$ is chosen as in \eqref{choo:lamt}, and hence that  $1/\lambda_t \in[2m_t,\infty)$ for $t=1,\ldots,B$.
    It then follows from Proposition~\ref{prop:ABIPP.BIPP.same}(a) that
    Proposition \ref{main:subroutine:1}(a) holds.
    Moreover, Proposition~\ref{prop:ABIPP.BIPP.same}(b) implies that $(z^+,v^+,\delta_+,\lambda)={\ABIPP}(z, p,\lam,c)$, and hence that the assumption of Proposition \ref{prop:ABIPP.output} holds with $\lam^+=\lam$.
    The conclusion of Proposition~\ref{prop:ABIPP.output}(b) then implies that inclusion \eqref{prop.B-IPP.inclusion}, and inequality \eqref{lemma:norm:residual:Ine.2} with $\lambda^+$ replaced by $\lambda$ hold, i.e.,
\[
\|v^+\|^2 +\delta_+  
\le \left[1+\frac{50}{\min(\lam)}+48 L^2\max(\lam) + c \zeta_2\right]
  \Big[ {\cal L}_c(z;p) - {\cal L}_c(z^+;p)\Big].
\] 
Using the above inequality, the definition of $\chione$, and the fact that \eqref{choo:lamt} implies that
\begin{equation*} 
\frac{1}{\min(\lam)} = 
% \frac{1}{\min_{1\leq t\leq B} \lam_t} = 
\max_{1\leq t\leq B} \frac{1}{\lam_t}\overset{\eqref{choo:lamt}}= 2 \max \left\{1,\max_{1\leq t\leq B}m_t \right\}, \qquad
\max(\lam)  = \max_{1\leq t\leq B} \lam_t \overset{\eqref{choo:lamt}}\leq \frac{1}{2},
\end{equation*}
we then conclude that \eqref{lemma:norm:residual:Ine} also holds.
We have thus proved that  Proposition~\ref{main:subroutine:1}(b) holds.
\qedflush

\subsection{An adaptive prox stepsize {\DA}}
\label{Adaptive prox stepsize ADMM}

This subsection presents {\DAA}, an adaptive prox stepsize analog of {\DA}, and argues that the complexity results  for {\DA} derived in Section \ref{section:VP.ADMM} can be naturally extended to {\DAA},  using the observations made in the paragraph preceding Lemma~\ref{lemma:Ti.expression}.

We start by stating {\DAA}, the adaptive prox stepsize analog of {\DA}.
In contrast to {\DA}, which uses {\SA} as a subroutine, the extended description below incorporates all the details of the adaptive {\SA} analog into a single loop, i.e., the second one.

\setcounter{algorithm}{2}
\begin{algorithm}[H]
\caption{{\DAA} ($m=(m_1,\ldots,m_B)$ is {\bf NOT} required)}\label{alg:adaptiveFINAL}
\begin{algorithmic}[1]
\Statex \hskip-1.8em
\textbf{Universal Input:}  $(\rho,\eta) \in \R^2_{++}$, \rev{$\alpha> 0$}, \rev{$C \ge \rho$}
\Require $x^0
\in {\cal H}$ and $\gamma^0\in \R^B_{++}$
\Ensure $(\hat x, \hat p, \hat u, \hat \varepsilon, \hat c)$ that satisfies the conclusions of Theorem~\ref{the:dynamic.adaptive}.
\vspace{1em}
\State $p^0=(p^0_1,\ldots,p^0_B) \gets (0,\ldots,0)$, $c_0 \gets 1/[1+\|Ay^0-b\|]$ \label{c0:ADMM}
\vspace{1em}
\For{$\ell\gets 1,2,\ldots$}\label{begin:outer:loop}
\State $T_0\gets 0$, $k\gets 0$, 
$c\gets c_{\ell-1}$, $(y^0,q^0,\lam^0, c)\gets (x^{\ell-1}, p^{\ell-1}, \gamma^{\ell-1}, c_{\ell-1})$
\For{$i\gets 1,2,\ldots$}\label{begin:inner:loop} %\label{star:cycle}
        \State $(y^i,v^i,\delta_i,\lam^{i})= \ABIPP(y^{i-1},q^{i-1},\lam^{i-1},c)$    \label{call:ABIPP:}   
        \If{$\|v^i\|^2 +\delta_i  \le \rho^2$}
            \State $k\gets k+1$, \ $i_k\gets i$ 
            \State  $q^i \gets  q^{i-1} + c (A y^i - b) $ \label{algo.adapt.admm.update.multiplier}
            \State $(x^\ell,p^\ell,u^\ell,  \varepsilon_\ell,\gamma^{\ell})=(y^i,q^i,v^i,\delta_i,\lambda^i)$ and \textbf{go to line \ref{update:c:A:ADMM}}
        \EndIf
\vspace{1em}
    \State $T_i= {\cal L}_c( y^{i-1};q^{i-1}) - {\cal L}_c(y^i;q^{i-1}) + T_{i-1}$\label{def:T:i:ADMM}
    \If{$\|v^i\|^2 +\delta_i\le C^2$ and $\dfrac{\rho^2}{\alpha(k+1)} \ge \dfrac{T_i}{i}  $}\label{stat:begin.test.cond:ADMM}
        \State $k\gets k+1$, \ $i_k\gets i$
        \State  $q^{i} \gets  q^{i-1} + c (A y^i - b) $  \label{def:qi:chi:ADMM}
 %\label{stat:test.i.qi.ti.updates} 
    \Else
        \State $q^i\gets q^{i-1}$ \label{sa.qi.update:ADMM}
    \EndIf
\EndFor
\vspace{1em}
\State $c_{\ell} \gets 2c_{\ell-1}$\label{update:c:A:ADMM}
\If{$\|Ax^\ell-b\|\leq \eta$}
\State \Return $(\hat x, \hat p, \hat u, \hat \varepsilon, \hat c)=(x^\ell,p^\ell,u^\ell,\varepsilon_\ell, c_{\ell})$\label{end:outer:loop}
\EndIf
\EndFor
\end{algorithmic}
\end{algorithm}

We now make some comments about {\DAA}. First,  {\DAA} consists of two main loops. The outer one, indexed by $\ell$, starts in line \ref{begin:outer:loop} and ends in line \ref{end:outer:loop}. The inner one, indexed by $i$, starts in line \ref{begin:inner:loop}
and ends in line \ref{sa.qi.update:ADMM}. 
The iterations of an inner loop,
referred simply to as an \textit{inner iteration}, \textcolor{blue}{execute} an adaptive variant of {\SA}, with {\SBIPP} replaced by
its adaptive counterpart {\ABIPP} (see line \ref{call:ABIPP:} of {\DAA}).
Second, an iteration of the outer loop, referred to as an \textit{outer iteration}, can be viewed as an adaptive version of a {\DA} iteration. Third, in contrast to {\DA}, which uses the constant prox stepsize \eqref{choo:lamt}, {\DAA} chooses the sequence of inner and outer prox stepsizes $\{\lam^i\}$ and $\{\gamma^\ell\}$ adaptively due to its call to {\ABIPP}  in line \ref{call:ABIPP:}.
Thus, {\DAA} is an ADMM variant which is fully adaptive, i.e., adaptive relative to all instance parameters associated with problem~\eqref{initial.problem} and hence fulfills the list of attributes listed in the ``Contributions" part of the Introduction.

We next discuss how the iteration-complexity of {\DAA} can be established by following an argument close to, but slightly different than, the one used for {\SA}/{\DA}. We start with some remarks about the sequence of inner prox stepsizes $
\{\lam^i = (\lam_1^i,\ldots,\lam_B^i) \}$ generated within an arbitrary inner loop of  {\DAA}. First, it is straightforward to observe that line \ref{call:ABIPP:} of {\DAA} and Proposition~\ref{prop:ABIPP.output} with $(z,p,\lam,c)=(y^{i-1},q^{i-1},\lam^{i-1},c)$, imply that,
for every inner iteration $i$ of any inner loop, we have
\begin{equation}\label{cond:lam:sec:6.2}
\lam_t^i \le \lam_t^{i-1}, \quad \max\left\{ \frac{1}{\lam_t^{i}},4m_t\right\} = \max\left\{ \frac{1}{\lam_t^{i-1}},4m_t\right\} \quad \forall t \in \{1,\ldots,B\},
\end{equation}
\begin{equation}\label{lemma:inclusion:residual:Ine.2:sec:6.2}
v^i  \in \nabla f (y^i)+ \partial_{\delta_i} h(y^i) +
A^*[q^{i-1}+c(Ay^i-b)],
\end{equation}
\begin{align}\label{lemma:norm:residual:Ine.2:sec:6.2:2}
\|v^i\|^2 +\delta_i  
\le \left[1+\frac{50}{\min(\lam^i)}+48 L^2\max(\lam^i) + c \zeta_2\right]
  \Big[ {\cal L}_c(y^{i-1};q^{i-1}) - {\cal L}_c(y^i;q^{i-1})\Big],
\end{align}
where $\zeta_2$ is as in \eqref{eq:sigma1&2}.
Relation  \eqref{cond:lam:sec:6.2} immediately 
implies that
\begin{equation}\label{cond:gamma:lambda}
\lam_t^i \le \gamma_t^{0}, \quad  \max\left\{ \frac{1}{\lam_t^{i}},4m_t\right\} = \max\left\{ \frac{1}{\gamma^{0}_t},4m_t\right\} \quad \forall t \in \{1,\ldots,B\}.
\end{equation}

Now, \eqref{lemma:norm:residual:Ine.2:sec:6.2:2} and \eqref{cond:gamma:lambda} imply that every  inner iteration $i$ of any inner loop satisfies
\begin{align}\label{lemma:norm:residual:Ine.2:sec:6.2:3}
\|v^i\|^2 +\delta_i  
&\le \left[ \chione + c \zeta_2\right]
  \Big[ {\cal L}_c(y^{i-1};q^{i-1}) - {\cal L}_c(y^i;q^{i-1})\Big],
\end{align}
where the pair $(\zeta_1,\zeta_2)$ in this subsection is defined as
\begin{equation}\label{eq:new.zetas}
\chione = 1+50 \bar \chi +48 L^2\max(\gamma^{0}), \quad
\zeta_2 = 24(B-1)\|A\|^2_{\dagger},
\end{equation}
and
\[
\bar \chi :=\max_{1 \le t \le B} \left( \max\left\{ \frac{1}{\gamma^{0}_t},4m_t\right\} \right).
\]

Relations \eqref{lemma:inclusion:residual:Ine.2:sec:6.2} and \eqref{lemma:norm:residual:Ine.2:sec:6.2:3} then imply that condition {\bf C1} in Section \ref{subsubsec:subroutine} holds with $(\zeta_1,\zeta_2)$ as in \eqref{eq:new.zetas}.
Thus, as observed in the paragraph containing condition {\bf C1}, all the results derived in Sections \ref{subsubsec:subroutine} and \ref{section:VP.ADMM}, except for Lemma~\ref{lemma:SADMM.BIPP}, hold.
In particular, Theorem~\ref{the:dynamic} translated to the context of {\DAA}
becomes as follows:

\begin{theorem}[{\DAA} Complexity]\label{the:dynamic.adaptive}
The following statements about {\DAA} hold:
   \begin{itemize}
       \item[(a)]
       it obtains a $(\rho,\eta)$-stationary solution of \eqref{initial.problem}-\eqref{eq:block_structure} in no more than $
   \log_2 \left[1+4\Upsilon(C)/(c_0\eta) \right]+1$
   outer iterations;
   \item[(b)]
   the total number of inner iterations performed by it is bounded by
\begin{equation*}%\label{com.AADMM}
\frac{\zeta_2\bar \Gamma(x^0;C,\alpha)}{\rho^2}\left[ \frac{8\Upsilon(C)}{\eta} + c_0\right] +
\left[ 1+\frac{\zeta_1\bar \Gamma(x^0;C,\alpha)}{\rho^2}\right]\log_2\left( 2+ \frac{8\Upsilon(C)}{c_0\eta}\right),
\end{equation*}
\end{itemize} 
where $(\zeta_1,\zeta_2)$, $\Upsilon(C)$   and $\bar \Gamma(x^0;C,\alpha)$ are as in~\eqref{eq:new.zetas}, \eqref{eq:def.Upsilon} and \eqref{eq:def.Gamma.dynamic}, respectively, and $c_0$ is as in line \ref{c0:ADMM} of {\DAA}.
\end{theorem}

\rev{Under the mild assumptions that $\|Ax^0-b\| = {\cal O}(1)$, \rev{$\alpha=\Omega(1)$, and $C={\cal O}(1)$,} Theorem~\ref{the:dynamic.adaptive}(b) implies that the complexity of {\DAA}, in terms of the tolerances only, is $\tilde {\cal O} (\epsilon^{-3})$ when $B>1$, and $\tilde {\cal O} (\epsilon^{-2})$ when $B=1$,
where $\epsilon:=\min\{\rho, \eta\}$.}

\subsection{Further implementation issues of \textcolor{blue}{an} {\ABIPP} loop} \label{sub:6.4}

This subsection addresses some aspects
related to the computation of a $(1/8,z_t)$-relative  stationary solution of the $t$-th block subproblem \eqref{algo:BIPP.inexact.subproblem.2}
in line \ref{output:step:2:ABIPP} of {\ABIPP}.

We first recall that throughout our presentation in Section \ref{section:Adaptive.ADMM}, we have assumed that a $(1/8,z_t)$-relative  stationary solution of  \eqref{algo:BIPP.inexact.subproblem.2}
can be obtained every time line \ref{output:step:2:ABIPP} of {\ABIPP} is executed. 
Such an assumption is reasonable if an exact solution $z_t^+$ of \eqref{algo:BIPP.inexact.subproblem.2}
can be computed in closed form since then  
$(z_t^+,v_t^+,\varepsilon_t^+)=(z^+_t,0,0)$ is a $(1/8;z_t)$-relative  stationary solution of \eqref{algo:BIPP.inexact.subproblem.2}. 
% an optimal solution of the $t$-th block subproblem \eqref{algo:BIPP.inexact.subproblem} can be computed????.

We now discuss the issues of finding a
$(1/8,z_t)$-relative stationary solution of \eqref{algo:BIPP.inexact.subproblem.2}  using the {\ADAP} described in Appendix \ref{sec:acg}
with input $(\mu_0,M_0) = (1/2,\lam_tc\|A_t\|^2)$, and hence with the same input as in the discussion on the second last paragraph of Subsection~\ref{sub:B-IPP}.
Recall that in that paragraph,
as well as in here, we assume that $\nabla_{t} f (x_1,\ldots,x_B)$ is
$\tilde L_{t}$-Lipschitz continuous with respect to  $x_t$.
% method described in \cite[Appendix A]{sujanani2023adaptive}.
If $\lam_t^+ >1/(2m_t)$ in line \ref{output:step:2:ABIPP}, 
then 
{\ADAP}
% might fail, and hence 
may not be able to find the required relative  stationary solution.
This is because
the smooth part of the objective function of \eqref{algo:BIPP.inexact.subproblem.2} with the above input
is not necessarily $(1/2)$-strongly convex (see statement (a) of Lemma~\ref{Lemma:ABIPP} below), which can cause failure of ADAP-FISTA (see Proposition \ref{pro.inexact.sol}(c) with $\mu_0=1/2$).
% Hence, {\ADAP}
% might fail, and hence may not find the required near-stationary solution, but signals that t
Nevertheless, regardless of whether  {\ADAP} succeeds or fails, a similar reasoning as in the second last paragraph of Subsection~\ref{sub:B-IPP} shows that
it terminates in ${\cal O} ( [\lam_t^+(\tilde L_t + c \|A_t\|^2 )]^{1/2})$ iterations.
Moreover, failure of {\ADAP}
signals that the current prox stepsize $\lam_t^+$ is too large.
In this case,
$\lam_t^+$ should be halved regardless of whether \eqref{eq:test:adap:new} is satisfied or not. An argument close to the one used in the proof of Proposition~\ref{prop:ABIPP.output}(a), which not only uses Lemma \ref{Lemma:ABIPP}(b) but also Lemma \ref{Lemma:ABIPP}(a), shows that this slightly modified version of {\ABIPP}  terminates in at most $1+\lceil \log_2(1+4m_t\lam_t) \rceil$ loop iterations.

\subsection{Proofs of Propositions~\ref{prop:ABIPP.BIPP.same} and \ref{prop:ABIPP.output} }\label{sub:section:6.3}

This subsection contains the proof of Propositions~\ref{prop:ABIPP.BIPP.same} and \ref{prop:ABIPP.output}.
First, we present a technical result, which states that if the $t$-th prox stepsize input for {\ABIPP} is sufficiently small, 
then its $t$-th loop terminates.

The next result uses the fact that line~\ref{output:step:2:ABIPP} of {\ABIPP}, together with Definition~\ref{def:BB}, generates a triple $(z_t^+,r_t^+,\varepsilon_t^{+}) \in {\cal H} \times \mathbb{R}^l 
\times \R_{+}$ such that
\begin{gather}\label{eq:ref:lemma:A1}
\begin{gathered}
r_t^+\in \nabla\left[\lam_t^+ {\cal \hat L}_{c}(z_{<t}^+,\cdot ,z_{>t}; p)+\frac{1}{2}\|\cdot -z_{t}\|^2\right](z_t^+)+\partial_{\varepsilon_t^+}(\lam_t^+ h_t)(z_t^+), \\
\|r_t^+\|^2 + 2 \varepsilon_t^+ \leq \frac{1}{8}\| z_t^+-z_t\|^2,
\end{gathered}
\end{gather}
where ${\cal \hat L}_{c}(\cdot; p)$ is  as in \eqref{def:smooth:ALM}.

\begin{lemma}\label{Lemma:ABIPP} 
Let $t \in \{1,\ldots,B\}$ be given and assume that
the prox stepsize $\lam_t^+$ at a certain iteration of the $t$-th {\ABIPP} loop satisfies
$1/\lam_t^+ \in [2m_t,\infty)$.
Then, the following statements hold for this iteration of the $t$-th loop:
\begin{itemize}
    \item[(a)] the smooth part of the objective function of the $t$-th block subproblem \eqref{algo:BIPP.inexact.subproblem.2} is $(1/2)$-strongly convex;
 
   \item[(b)] the $t$-th loop ends at this iteration.
\end{itemize}

\end{lemma}

\begin{proof}
(a)  
The assumption that $1/\lam_t^+ \in [2m_t,\infty)$ implies that the matrix $B_t : = (1-\lam_t^+ m_t)I+\lam_t^+ cA_t^*A_t$ is clearly positive definite, and hence defines the norm $\|\cdot\|_{B_t}$ whose square is
\begin{equation}\label{def:norm:Bt:Ine:AB}
\|\cdot\|_{B_t}^2:=\langle\;\cdot\;, \, B_t(\cdot) \, \rangle 
\ge
\lam_t^+ c \|A_t (\cdot)\|^2 + \frac12 \|\cdot\|^2.
\end{equation}
Now, let $\Psi_t(\cdot)$ denote the smooth part of the objective function of the $t$-th block subproblem \eqref{algo:BIPP.inexact.subproblem.2}, i.e.,
\begin{equation}\label{def:Psi:t}
\Psi_t(\cdot):=\lam_t^+ \hat {\cal L}_{c}(z_{<t}^{+},\cdot,z_{>t};p)+\frac{1}{2}\|\cdot-z_{t}\|^{2}
\end{equation}
where ${\cal \hat L}_{c}(\cdot\,;p)$ is as in \eqref{def:smooth:ALM}.
Moreover, using assumption (A2), the above definitions of $B_t$ and the norm $\|\cdot\|_{B_t}$, the definitions of ${\cal \hat L}_{c}(\cdot\,;\cdot)$ and  $\Psi_t(\cdot)$ in \eqref{def:smooth:ALM}
 and \eqref{def:Psi:t}, respectively, we easily see that the function 
$\Psi_t(\cdot) - \frac12 \|\cdot\|^2_{B_t}$ is  convex, and hence $\Psi_t(\cdot)$ is $(1/2)$-strongly convex due to the inequality in \eqref{def:norm:Bt:Ine:AB}.
%}

(b) %The conclusion that the triple $(z_t^+,r_t^+,\varepsilon_t^{+})$
 %  obtained at the end of this iteration is
 %   a $(1/8,z_t)$-\rev{relative}  stationary solution  of \eqref{algo:BIPP.inexact.subproblem.2} follows from line \ref{output:step:2:ABIPP} of {\ABIPP}.
It suffices to show that  \eqref{eq:test:adap:new} holds at this loop iteration.
Due to \eqref{def:Psi:t} and the inclusion in \eqref{eq:ref:lemma:A1}, the quadruple $(z_t^+,r_t^+, \varepsilon_t^+,\lam_t^+)$ satisfies
\begin{align*}
r_t^+\overset{\eqref{eq:ref:lemma:A1}}\in  \nabla \Psi_t(z_t^+) 
         + \partial_{\varepsilon_t^+}(\lam_t^+ h_t) (z_t^+ )
         \subset \partial_{\varepsilon_t^+}
         \left[ \Psi_t + \lam_t^+ h_t\right](z_t^+)
        \end{align*}
where the set inclusion is due to
\cite[Thm. 3.1.1 of Ch.\ XI]{lemarechal1993}, and the fact that $\Psi_t(\cdot)$ and $\lam_t^+ h_t(\cdot )$ are convex functions.
Since $\Psi_t(\cdot) - \frac12 \|\cdot\|^2_{B_t}$ is  convex,
it follows from Lemma \ref{conv:result}  with $\psi=\Psi_t+\lambda_t^+h_t$
   , $(\xi,\tau, Q)=(1, 1, B_t)$, $(u,y,v) = (z_{t},z_t^+,r_t^+)$, and $\eta=\varepsilon_t^+$, that
\begin{align*}
    &\lam_t^+ {\cal L}_c(z_{<t}^+,z_{t}, z_{>t}; p)-\left[\lam_t^+ {\cal L}_c(z_{<t}^+,z_t^+, z_{>t}; p) +\frac{1}{2}\|z_t^+-z_t\|^2 \right] \\
    &= \left[\Psi_t(z_{<t}^+,z_t, z_{>t}; p)+\lambda_t^+ h_t(z_t)\right] - \left[ \Psi_t(z_{<t}^+,z_t^+, z_{>t}; p)+\lambda_t^+ h_t(z_t^+) \right]  \\
    &\ge  \frac{1}{4}\|z_t^+-z_t\|_{B_t}^2 -2\varepsilon_t^++\langle r_t^+, z_t-z_t^+\rangle \overset{\eqref{def:norm:Bt:Ine:AB}}\ge \frac{\lam_t^+ c}{4}\|A_t(z_t^+-z_t)\|^2 -2\varepsilon_t^+ +\langle r_t^+, z_t-z_t^+\rangle,
\end{align*}
where the first equality follows from
\eqref{DP:AL:F}, \eqref{def:smooth:ALM},  and \eqref{def:Psi:t},
the first inequality is due to Lemma~\ref{conv:result}, and
the last inequality is due to \eqref{def:norm:Bt:Ine:AB}.
Using the previous inequality, the inequality $a b \le (a^2 + b^2)/2$ with $(a,b)=(\sqrt{2}\|r_t^+\|, (1/\sqrt{2})\|z_t^+-z_t\|)$, and the inequality  in \eqref{eq:ref:lemma:A1}, we conclude that
\begin{align}
\label{tes:SA:line:search:Ine}
 {\cal L}_c(z_{<t}^+,z_{t}, z_{>t}; p)&-{\cal L}_c(z_{<t}^+,z_t^+, ,z_{>t}; p) \nonumber\\ 
&\ge  \frac{1}{2\lam_t^+}\|z_t^+-z_t\|^2 + \frac{c}{4} \|A_t (z_t^+-z_t)\|^2  -\frac{1}{\lam_t^+}\left(\| \sqrt{2} r_t^+\| \left\|\frac1{\sqrt{2}} (z_t^+-z_t)\right\|+ 2\varepsilon_t^+\right) \nonumber\\
&\ge \frac{1}{2\lam_t^+}\|z_t^+-z_t\|^2 + \frac{c}{4} \|A_t (z_t^+-z_t)\|^2  -\frac{1}{\lam_t^+} \left(  \|r_t^+\|^2 + \frac14 \|z_t^+-z_t\|^2+ 2\varepsilon_t^+ \right) \nonumber\\
&\overset{\eqref{eq:ref:lemma:A1}}\ge \frac{1}{2\lam_t^+}\|z_t^+-z_t\|^2 + \frac{c}{4} \|A_t (z_t^+-z_t)\|^2  -\frac{1}{\lam_t^+} \left(\frac14 +\frac18  \right) \|z_t^+-z_t\|^2 \nonumber\\
&= \frac{1}{8\lam_t^+} \|z_t^+-z_t\|^2 + \frac{c}{4} \|A_t(z_t^+-z_t)\|^2, \nonumber
\end{align}
and hence that \eqref{eq:test:adap:new} holds.  By the logic of {\ABIPP}, it then follows that the $t$-th {\ABIPP} loop terminates at the current loop iteration with $\lam_t^+$ being the final $t$-th stepsize output by {\ABIPP}. We have thus proved that (b) holds.
\end{proof}

\vspace{1em}

% Under the convention that $1/0=+\infty$, i
It follows from Lemma~\ref{Lemma:ABIPP}(b) that, if function $f$ restricted to its $t$-th block variable is convex, i.e., $m_t=0$, then the $t$-th {\ABIPP} loop terminates in one iteration with $\lam_t^+=\lam_t$. Hence, {\ABIPP} does not update $\lam_t$ when $m_t=0$.

Proposition~\ref{prop:ABIPP.BIPP.same} now follows immediately from Lemma~\ref{Lemma:ABIPP}. Indeed, both of its statements follow from its assumptions, line \ref{lam+=lam} of {\ABIPP}, and Lemma \ref{Lemma:ABIPP} with $\lambda_t^+=\lambda_t$ for every $t\in \{1,\ldots,B\}$.

We are now ready to prove Proposition~\ref{prop:ABIPP.output}

\noindent
{\bf Proof of Proposition \ref{prop:ABIPP.output}:}
(a) Let $t \in \{1,\ldots, B\}$ be given. 
Recall from Lemma~\ref{Lemma:ABIPP}(b) that if  $1/\lambda_t^+\in[2m_t,\infty)$ 
at some iteration of the $t$-th {\ABIPP} loop, then the loop must terminate at this iteration.
Using this observation and the fact that $1/\lam_t^+$ is doubled every time line \ref{eq:test:ABIPP} of {\ABIPP} is executed, we easily see that the $t$-th {\ABIPP} loop stops in at most $1+\lceil \log_2(1+2m_t\lam_t) \rceil$ iterations, the inequality in \eqref{eq:ABIPP:output.lambda.bound.2} holds, and \begin{equation*}
      \frac1{\lam_t^+ } \le \max \left\{ \frac1{\lam_t} , 4m_t\right\}.
\end{equation*}
The identity in \eqref{eq:ABIPP:output.lambda.bound.2} now follows immediately from 
the inequality $1/\lam_t \le 1/\lam_t^+$ and the one above.

% (a) First, note that if the inequality
% $\lam_t^+ \leq 2\lam_t/(1+4\lambda_tm_t)$
% is satisfied at the beginning of an iteration of the $t$-th loop, then the loop must terminate at this iteration, in view of Lemma~\ref{Lemma:ABIPP}(b) and the fact that the inequality implies that $1/\lambda_t^+\in[2m_t,\infty)$.
% Using this observation and the fact that $\lam_t^+$ is halved every time line \ref{eq:test:ABIPP} of {\ABIPP} is executed, it is now straightforward to see that
% the first two claims in (a) hold. Relations in \eqref{eq:ABIPP:output.lambda.bound.2} are an immediate consequence of \eqref{eq:ABIPP:output.lambda.bound}.

(b) We first prove the inclusion in \eqref{prop.B-IPP.inclusion}. 
% To simplify notation, let $p^+:=p+c(Az^+-b)$. 
Using the inclusion in
\eqref{eq:ref:lemma:A1}, the definition of ${\cal \hat L}_{c}(\cdot\,;p)$ in \eqref{def:smooth:ALM},
% the definition of $p^+$,
and 
relation \eqref{prop1:subd} with $(\varepsilon, \beta)=(\varepsilon_t^+, \lambda_t^+)$,
we easily see that  
\begin{align}\label{eq:COA1}
\frac{r_t^+}{\lambda_t^+}
&\overset{\eqref{eq:ref:lemma:A1}}\in \nabla_{t}f(z_{< t}^+,z_t^+, z_{>t})+A_{t}^{*}\left[ p+c[A(z_{< t}^+,z_t^+,z_{>t})-b]\right]+\frac{1}{\lambda_t^+}(z_t^+-z_t)+\partial_{(\varepsilon_t^+/\lambda_t^+)}h_t(z_t^+)\nonumber \\
& =\nabla_{t}f(z_{< t}^+,z_t^+, z_{>t})+A_{t}^{*}\left(p+c(Az^+-b) - c\sum_{s=t+1}^{ B}A_s (z_s^+-z_s)\right)+\frac{1}{\lambda_t^+}(z_t^+-z_t)+\partial_{(\varepsilon_t^+/\lambda_t^+)}h_t(z_t^+), \nonumber
\end{align}
for every $t\in \{1,\ldots,B\}$. 
Rearranging the above inclusion and using the definition of $v_t^+$ in line \ref{algo.stat.vit:ADMM} of {\ABIPP}, we have
% that for every $t\in \{1,\ldots,B\}$,
\[
v_t^+ \in \nabla_{t}f(z^+)+\partial_{(\varepsilon_t^+/\lambda_t^+)}h_t(z_t^+)+A_t^*\left[p+c(Az^+-b)\right] \quad \forall t\in \{1,\ldots,B\}.
\]
Moreover, we have
that $\partial_{(\varepsilon_1^+/\lambda_1^+)}h_1(z_1^+) \times \ldots \times
\partial_{(\varepsilon_B^+/\lambda_B^+)}h_B(z_B^+)  \subseteq \partial_{\delta_+}h(z^+)$,
due to \eqref{prop2:subd} and the definition of $\delta_+$ in line \ref{algo:IPP.v.delta:ADMM} of {\ABIPP}. 
The inclusion in \eqref{prop.B-IPP.inclusion} now follows from the last two conclusions and the definition of $v^+$ in line \ref{algo:IPP.v.delta:ADMM} of {\ABIPP}.

We now prove the inequality in \eqref{lemma:norm:residual:Ine.2}.
Lemma \ref{Lemma:ABIPP}(b) and the logic of {\ABIPP} imply that
the $t$-th {\ABIPP} loop 
terminates with 
% ( Lemma \ref{Lemma:ABIPP}(b)), it follows from the logic of {\ABIPP} that it returns 
a quadruple 
$(z_t^+,r_t^+,\varepsilon_t^{+},\lambda_t^+)$
satisfying  \eqref{eq:test:adap:new}.
Adding \eqref{eq:test:adap:new} from $t=1$ to $t=B$, we have 
\begin{equation}\label{eq:adapt.Delta.Lc.ineq}
\Delta {\cal L}_c := {\cal L}_c(z;p) - {\cal L}_c(z^+;p) \ge \frac{1}{8}\sum_{t =1}^B \frac{\|z_t^+-z_t\|^2}{\lam_t^+} + \frac{c}{4}\sum_{t =1}^B \|A_t (z_t^+-z_t)\|^2.
\end{equation}
% We now prove the inequality in \eqref{lemma:norm:residual:Ine.2}.
Now, using \eqref{eq:ref:lemma:A1}, \eqref{eq:adapt.Delta.Lc.ineq},  and that $1/\lam_t^+\leq 1/\min(\lambda^+)$ (see \eqref{def.lambda.min.max}), we have
\begin{align*}
\sum_{t=1}^B \left( 2\frac{\|r_t^+\|^2}{(\lam_t^+)^2}  + \frac{\varepsilon_t^+}{\lam_t^+} \right)
&\le \left[\frac{2}{\min(\lambda^+)}+1\right] \sum_{t=1}^B \left( \frac{\|r_t^+\|^2 + \varepsilon_t^+}{\lam_t^+}\right)\\
&\overset{\eqref{eq:ref:lemma:A1}}\le \left[\frac{2}{\min(\lambda^+)}+1\right]  \sum_{t=1}^B \left( \frac{\|z_t^+-z_t\|^2}{8\lam_t^+}  \right) \overset{\eqref{eq:adapt.Delta.Lc.ineq}}\le  \left[\frac{2}{\min(\lambda^+)}+1\right] \Delta {\cal L}_c.
\end{align*}
Defining $D_{t}:=\|v_t^+-r_t^+/\lam_t^+\|^2$, using the previous inequality, the definition of $\delta_+$ (see line \ref{algo:IPP.v.delta:ADMM} of {\ABIPP}), and that $\|a+b\|^2\leq 2\|a\|^2+2\|b\|^2$, for any $a,b\in \R^n$, we have
\begin{align}\label{eq:suff.decr.inter}
\|v^+\|^2 &+\delta_+  =
\sum_{t=1}^B \left( \|v_t^+\|^2 + \frac{\varepsilon_t^+}{\lam_t^+}  \right) 
\le 
\sum_{t=1}^B \left( 2 D_{t} + 2\frac{\|r_t^+\|^2}{(\lam_t^+)^2}  + \frac{\varepsilon_t^+}{\lam_t^+}
\right) \leq 2 \sum_{t=1}^B  D_{t} +  \left[\frac{2}{\min(\lambda^+)}+1\right] \Delta {\cal L}_c.
\end{align}
We will now bound $\sum_{t=1}^BD_{t}$. 
Using \eqref{eq:adapt.Delta.Lc.ineq} and the inequality $\lam_s^+ \leq \max(\lambda^+)$,
 we have
\begin{align}\label{eq:bound:vi:interm1}
\|z_{>t}^+-z_{>t}\|^2 
=
\sum_{s =t+1}^B \|z_s^+-z_s\|^2
\le
\left( 
\max(\lambda^+) \sum_{s =t+1}^B \frac{\|z_s^+-z_s\|^2}{\lam_s^+}
\right)   
\overset{\eqref{eq:adapt.Delta.Lc.ineq}}\le 8 \max(\lambda^+) \Delta {\cal L}_c.
\end{align}
Moreover, the definitions of $D_{t}$ given above and $v_t^+$ in line \ref{algo.stat.vit:ADMM} of {\ABIPP}, the Cauchy-Schwarz inequality, assumption (A3), and relations \eqref{eq:adapt.Delta.Lc.ineq} and \eqref{eq:bound:vi:interm1}, imply that
\begin{align}\label{eq:bound:vi:first:Ine}
D_{t} &=\left\|v_{t}^+-\frac{r_t^+}{\lam_t^+}\right\|^2  =\left\|\nabla_{t}f(z_{\leq t}^+, z_{>t}^+)-\nabla_{t}f(z_{\leq t}^+, z_{>t})+cA^*_{t}\left[\sum_{s=t+1}^{ B}A_{s}(z_s^+-z_s)\right]-\frac{(z_t^+-z_t)}{\lambda_t^+}\right\|^2\nonumber\\
&\leq 3\left\{ \|\nabla_{t}f(z_{\leq t}^+, z_{>t}^+)-\nabla_{t}f(z_{\leq t}^+, z_{>t})\|^2+ \left[ c\|A_t\| \sum_{s=t+1}^{B} \|A_{s} (z_s^+-z_s) \| \right]^2 +\frac{\|z_t^+-z_t\|^2}{(\lambda_t^+)^2}\right\}\nonumber\\
&\overset{\eqref{eq:lipschitz_x}}\leq 3\left\{  (L_{>t})^2 \|z_{>t}^+-z_{>t}\|^2+c^2\|A_t\|^2 (B-t)\sum_{s=t+1}^{B}\|A_{s}(z_s^+-z_s)\|^2+\frac{1}{\min(\lambda^+)} \frac{\|z_t^+-z_t\|^2}{\lambda_t^+} \right\} \nonumber\\
& \overset{\eqref{eq:adapt.Delta.Lc.ineq},\eqref{eq:bound:vi:interm1}}\le
3  \left\{  \Big[ 8 \max(\lambda^+) (L_{>t})^2 
+  4 c \|A_t\|^2 (B-1) \Big]  \Delta {\cal L}_c +  \frac{1}{\min(\lambda^+)} \frac{\|z_t^+-z_t\|^2}{\lambda_t^+} \right\} . \nonumber
\end{align}
The above inequality for $t=1,\ldots,B$, \eqref{eq:adapt.Delta.Lc.ineq}, and the definitions of $L$ and $\|A\|_\dagger$  in \eqref{eq:block_norm}, then imply that
\begin{align*}
\sum_{t=1}^B D_{t} & \overset{\eqref{eq:block_norm}} \leq  \left[ 12 c \|A\|^2_\dagger (B-1) + 24 \max(\lambda^+)L^2   \right] \Delta {\cal L}_c + \frac{3}{\min(\lambda^+)} \sum_{t=1}^B \frac{\|z_t^+-z_t\|^2}{\lambda_t^+} \\
&\overset{\eqref{eq:adapt.Delta.Lc.ineq}}\leq 12\left[  c \|A\|^2_\dagger (B-1)+ 2\max(\lambda^+) L^2+\frac{2}{\min(\lambda^+)} \right]\Delta {\cal L}_c. 
\end{align*}
This inequality, relation \eqref{eq:suff.decr.inter}, and the definition of 
$\chitwo$ in
\eqref{eq:sigma1&2}, show that 
 the inequality in \eqref{lemma:norm:residual:Ine.2} holds.\qedflush

\section{Numerical Experiments}\label{sec:numerical}

This section showcases the numerical performance of {\DAA} on three linearly constrained non-convex (weakly convex) programming problems. Subsections \ref{sub:numerics.Problem1.l1norm} and \ref{sub:numerics.Problem3} focus on a quadratic problem, while Subsection \ref{sub:numerics.Problem2} focuses on the distributed \textit{Cauchy loss} function~\cite{Cauchy-poh}. Subsections \ref{sub:numerics.Problem1.l1norm} and  \ref{sub:numerics.Problem2} employ fewer blocks, each with a wide dimensional range, whereas Subsection \ref{sub:numerics.Problem3} uses a large number of one-dimensional blocks. These three proof-of-concept experiments indicate that {\DAA} may not only substantially outperform the relevant benchmarking methods in practice, but also be relatively robust to the relationship between block counts and sizes.

To provide an adequate benchmark for our methods, we compare six algorithmic variants in the tables presented in the following subsections. 
All variants follow the same core {\DAA} framework, differing only in the use of adaptive or constant prox stepsizes and in the presence or absence of Lagrange multiplier updates:
\begin{itemize}
\item {\DAA} / {\DA}: the original method with adaptive or constant prox stepsizes.
\item {\ADAPTPE} / {\CONSTPE}: penalty-only variants with no multiplier updates.
\item {\ADAPTVA} / {\CONSTVA}: \textit{vanilla} ADMM variants with multiplier updates at each iteration.
\end{itemize}

It is worth mentioning that all nonadaptive variants use {\SBIPP} as their subroutine, while all adaptive variants use {\ABIPP}. The iteration counter reported in the tables corresponds to the total number of BIPP/A-BIPP calls, which corresponds to the total number of iterations mentioned in Theorem \ref{the:dynamic}  for the nonadaptive variant (and hence counts {\SBIPP} calls) and in Theorem \ref{the:dynamic.adaptive} for the adaptive variant (and hence counts {\ABIPP} calls).

For each of these variants, the total number of iterations (``Iter''), total runtime (``Time''), and the objective value at the solution (``$f+h$'') are included in the tables presented in the following subsections. 
The tables also contain a column labeled ``Mults'' for the variants {\DAA} and {\DA}, indicating the total number of Lagrange multiplier updates performed by the method. 
This column is omitted for {\PE} and {\VA} variants, since the number of multipliers they perform is naturally clear.

Notably, we attempted to implement the algorithm proposed in \cite{KongMonteiro2024} using different parameter choices $(\theta,\chi)$ that theoretically should ensure convergence; see \eqref{eq:p_update} and \eqref{eq:assumptio:B}. However, since none of these choices managed to find the desired point within the iteration limit, we omitted them from our benchmarks.

In all three subsections, we assume that the blocks for the generated instances of \eqref{initial.problem}-\eqref{eq:block_structure} have the same size $\bar n$, i.e.,
\[
\bar n = n_1 = \ldots = n_t,
\]
and hence that
$n=\bar n B$.
So, the sizes of instances are determined by a triple
$(B,\bar n,l)$ where
$l$ is the number of rows of the constraint matrix $A$.
Instances with $\bar n< l$ are usually harder to solve since they further ``deviate" from \textit{the last block condition} described in the
Introduction (see paragraph following~\eqref{eq:assumptio:B}).

The parameters
$(\omega,B,\bar n,l)$ are specified at the beginning of each table.
Each matrix $A_i\in\mathbb{R}^{l\times \bar n}$, corresponding to the linear constraint in each problem, is filled with i.i.d.\ standard-normal entries. The penalty parameter and the Lagrange multiplier are chosen in accordance with line~\ref{def:c0:p0} of {\DA}, i.e., $c_0 = 1/(1+\|A x^0 - b\|)$ and $p^0 = \mathbf 0$. Moreover, we fix the input $C$ as $C = 10^{3}\rho (1+\|\nabla f(x^0)\|)$. Finally, for consistency, all of the variants with adaptive stepsizes were initialized using the same value, setting $\lambda_i^0 = 10$ for each block.

For each of the problems in Subsections \ref{sub:numerics.Problem1.l1norm} and \ref{sub:numerics.Problem2}, we solve the corresponding subproblem \eqref{algo:BIPP.inexact.subproblem} using the {\ADAP} routine described in Appendix~\ref{sec:acg}, with $(M_0,\beta, \mu, \chi)=(1, 1.2, 0.5,  0.001)$.
The routine terminates once it produces a $(1/8,z_t)$-stationary point.  
If this criterion is not met, the current stepsize $\lambda_t$ is halved,
and {\ADAP} is restarted with the reduced stepsize. Since the subproblems in Subsection~\ref{sub:numerics.Problem3}
are all one-dimensional,
they are solved exactly without invoking  {\ADAP}.

All algorithms executed were run for a maximum of $100,000$ iterations. 
Any algorithm reaching this limit required at least $10$ milliseconds to complete.
A method is considered to outperform another when it achieves both a lower iteration count and a shorter total runtime. 

To ensure timely execution, each algorithm was terminated upon reaching the above iteration limit or upon finding an approximate stationary triple $(x^+, p^+, v^+)$ satisfying the relative error criterion
\[
v^+ \in \nabla f(x^+) + \partial h(x^+) + A^*p^+, \quad \frac{\|v^+\|}{1+\|\nabla f(x^0)\|}\leq \rho , \quad \frac{\|Ax^+-b\|}{1+\|Ax^0-b\|}\leq \eta,
\]
with $\rho=\eta=10^{-5}$.

All experiments were implemented and executed in  MATLAB 2024b and run on a macOS machine with an Apple M3 Max chip (14 Cores), and 96 GB of memory.
For the sake of brevity, our benchmark only considers randomly generated dense instances.
Its main goal is  to demonstrate that the ADMMs presented in this work are promising.

\subsection{Nonconvex Distributed Quadratic Programming (DQP) Problem}\label{sub:numerics.Problem1.l1norm}

This subsection studies the performance of the ADMM variants for finding stationary points of a nonconvex block distributed quadratic programming problem with $B$ blocks (DQP). 

For a given pair $(l,\bar n)\in \mathbb{N}_{++}^{2}$ with $l< \bar n B$, the $B$-block DQP is formulated as
\begin{align}\label{eq:num.problem1}
\min_{x=(x_{1},\ldots,x_{B})\in\mathbb{R}^{\bar n B}}\  & \left\{\sum_{i=1}^{B}\left[\frac{1}{2}\inner{x_{i}}{P_ix_{i}}+\left\langle x_{i},r_{i}\right\rangle \right] \; : \; \|x_i\|_{1}\leq \omega,\, \forall i=1,\ldots B, \text{ and } \sum_{i=1}^BA_ix_i=b  \right\},\nonumber
\end{align}
where $\omega > 0$, $b\in\mathbb{R}^l$, $P_i\in \mathbb{R}^{\bar n\times \bar n}$ is a symmetric indefinite matrix, $x_i\in \mathbb{R}^{\bar n}$, $r_i\in \mathbb{R}^{\bar n}$ and $A_i\in\mathbb{R}^{l\times \bar n}$, for all $i\in \{1,\ldots, B\}$. It is not difficult to check that the DQP problem fits within the template defined by \eqref{initial.problem}-\eqref{eq:block_structure}, where
\[
f(x) = \sum_{i=1}^{B}\left[\frac{1}{2}\inner{x_{i}}{P_ix_{i}}+\left\langle x_{i},r_{i}\right\rangle \right] \quad \text{and}\quad h_i(x_i) = \delta_{\{x\in \R^{\bar n}:\|x\|_1\leq \omega\}}(x_i), \ \forall i=1,\ldots,B.
\]

We now outline the experimental setup used for the DQP problem.
% how we conducted our DQP experiments. 
First, to define $b$, we sample $x^b$ uniformly at random satisfying $\|x^b_i\|_1\leq \omega$ for $i=1,\ldots,B$ and set  $b \;=\;\sum_{i=1}^{B} A_i x^b_i$. The initial iteration $x^0$ is drawn uniformly satisfying $\|x^0_i\|_1\leq \omega$ for $i=1,\ldots,B$. An orthonormal matrix $Q_i$ is generated using the standard normal distribution. Then, a diagonal matrix $D_i$ is constructed such that one-third of its diagonal entries are set to zero, while the remaining entries are drawn uniformly from the interval $[-10, 10]$, 
ensuring that at least one of them is negative. 
Next, the matrix $P_i$ is defined as $P_i = Q_i^\top D_i Q_i$.
It is straightforward to verify that if $m_i$ denotes the smallest eigenvalue of $D_i$, then $m_i < 0$, and the function $f(x_{<i}, \cdot , x_{>i})$ is 
$|m_i|$-weakly convex. 
Hence, all variants with constant prox stepsizes set $\lambda_i^0 = 1 / (2 \max\{1, |m_i|\})$ for $i\in\{1,\ldots,B\}$.
Finally, each vector $r_i$ is generated independently, with entries drawn from the standard normal distribution.

The results of the experiments are summarized in Table \ref{tab:P1-onenorm}.

\input{DQP-onenorm}

 This table does not include the methods {\ADAPTPE}, {\CONSTPE}, {\ADAPTVA}, and {\CONSTVA}, as none of them converged for any of the tested instances. We now make some remarks about the above numerical results. Both methods successfully converged on all instances, but {\DAA} outperformed {\DA} on  about $90$\% of the instances tested.  In summary, the above results show that {\DAA} is better than its constant stepsize counterpart {\DA}.

\subsection{Distributed \textit{Cauchy loss} function}\label{sub:numerics.Problem2}

This subsection studies the performance of the ADMM variants for finding stationary points of a  nonconvex block distributed \textit{Cauchy loss} function problem with $B$ blocks.

For a given pair $(l,\bar n)\in \mathbb{N}_{++}^{2}$, with $l < \bar n B$, the $B$-block \textit{Cauchy loss} function is formulated as
\begin{align*}
\min_{x=(x_{1},\ldots,x_{B})\in\mathbb{R}^{\bar n B}}\   \left\{ \sum_{i=1}^{B}\frac{\alpha_i^2}{2}\log \left[1 + \left(\frac{y_i-\inner{x_i}{z_i}}{\alpha_i}\right)^2\right] \; : \; x_i\in\omega\Delta_{\bar n},\; \forall i=1,\ldots,B, \text{ and } \;\sum_{i=1}^BA_ix_i=b\right\},
\end{align*}
where $\omega > 0$, $b\in\mathbb{R}^l$, $\alpha_i>0$,  $y_i\in\mathbb{R}$, 
$(z_i, x_i)\in \mathbb{R}^{\bar n}\times \mathbb{R}^{\bar n}$, and $A_i\in\mathbb{R}^{l\times \bar n}$ for all $i\in \{1,\ldots,B\}$. For $m\in\R_{++}$, let $\mathbf{1}_{m}\in \R^{m}$ be the vector of all ones. The standard $m$-dimensional simplex is defined as $\Delta_{m}=\{x\in\R^m_{++} \ : \ \rev{\mathbf{1}^\top_m} x=1 \}$, and its scaled version is $\omega\Delta_{m}=\{x\in\R^m_{++} \ : \ \rev{\mathbf{1}^\top_m} x=\omega \}$. It is not difficult to check that the distributed \textit{Cauchy loss} problem fits within the template defined by \eqref{initial.problem}-\eqref{eq:block_structure}, where
\[
f(x) = \sum_{i=1}^{B}\frac{\alpha_i^2}{2}\log \left[1 + \left(\frac{y_i-\inner{x_i}{z_i}}{\alpha_i}\right)^2\right] \quad \text{and}\quad h_i(x_i) = \delta_{\omega\Delta_{\bar n}}(x_i), \ \forall i=1,\ldots,B.
\]

We now outline the experimental setup used in the above problem.
For each $i\in \{1,\ldots,B\}$, the scalar $y_i\in\mathbb{R}$ and the vector $z_i\in\mathbb{R}^{\bar n}$ are generated with entries drawn from the standard normal distribution, and the parameter $\alpha_i$ is sampled uniformly at random from the interval $[50,100]$. To define $b$, we sample $x^b=(x^b_1,\ldots,x^b_B)$ uniformly at random satisfying $x^b_i\in \omega\Delta_{\bar n}$ for $i=1,\ldots,B$ and set  $b \;=\;\sum_{i=1}^{B} A_i x^b_i$. The initial iteration $x^0=(x^0_1,\ldots,x^0_B)$ is drawn uniformly satisfying $x^0_i\in\omega\Delta_{\bar n}$ for $i=1,\ldots,B$.

\input{R2CauchysimplexTable}

The results of this experiment are summarized in Table \ref{tab:cauchy}. The row labeled ``$f+h$’’ is omitted, as its values consistently ranged from $10^{-12}$ to $10^{-9}$ in every instance.

We now present some comments about the numerical results.
From these tables, we first observe that the adaptive variants outperform their constant  prox stepsize counterparts. Moreover, both {\DAA} and {\ADAPTPE} converged on all instances, whereas {\DA} and {\CONSTPE} converged on only $40\%$ of the instances. The standard ADMM performed poorly overall: {\ADAPTVA} converged on $36\%$ of the instances, while {\CONSTVA} converged in only $2\%$ of the cases (one instance). In summary, the tables above show that for the \textit{Cauchy loss} problem, {\DAA} and {\ADAPTPE} exhibit a similar behavior.

\subsection{Nonconvex QP with Box Constraints}\label{sub:numerics.Problem3}

This subsection studies the performance of ADMM variants for solving a class of nonconvex quadratic problems with box constraints (QP-BC).

Specifically, the QP-BC problem considered in this subsection is 
\begin{align*}
\min_{x=(x_1,\ldots,x_B)\in\mathbb{R}^{\bar n B}}\  & \left\{\frac{1}{2}\inner{x}{Px} + \inner{r}{x} \; : \; \|x\|_{\infty}\leq \omega \text{ and } Ax=b\right\},\nonumber
\end{align*}
where $P\in \R^{B\times B}$ is a symmetric indefinite matrix, $A\in \R^{l\times B}$, $(r,b)\in \R^B \times\R^l$, and $\omega>0$.
In this subsection, we view QP-BC as an extreme special case of  \eqref{initial.problem}-\eqref{eq:block_structure} where
each variable forms a block, and hence $B=n$ and $\bar n = 1$.
In this case, the variable blocks correspond to individual coordinates.
Consequently, each column of the matrix $A$ defines an $A_t$ matrix for $t\in\{1,\ldots,B\}$. It is not difficult to check that the QP-BC problem fits within the template defined by \eqref{initial.problem}-\eqref{eq:block_structure}, where
\[
f(x) =  \frac{1}{2}\inner{x}{Px} + \inner{r}{x}\quad \text{and}\quad h_i(x_i) = \delta_{\{x\in \R^{\bar n}:\|x\|_\infty\leq \omega\}}(x_i), \ \forall i=1,\ldots,B.
\]

We now describe how we orchestrated our QP-BC experiments.
First, an orthonormal matrix $Q$ is generated using the standard normal distribution.
Then, a diagonal matrix $D_i$ is constructed such that one-third of its diagonal entries are set to zero, while the remaining entries are drawn uniformly from the interval $[-10, 10]$, 
ensuring that at least one of them is negative.
Next, the matrix $P$ is defined as $P = Q^\top D Q$. 
The vector $r$ is generated using the standard normal distribution.

The results of this experiment are summarized in Table \ref{tab:problem3}.

\input{P3-Table-new}

We now make some remarks about the above numerical results.
We begin by comparing the performance of the two ADMM variants.
% {\DA}-Adapt against its fixed-step counterpart, {\DA}-Fixed. 
Both variants successfully converged on all instances, but 
{\DAA} outperformed {\DA} on  about $63$\% of the instances tested.

We now compare {\DAA} against {\ADAPTPE}. While {\DAA} converged successfully in all test instances, {\ADAPTPE} converged in only about 63\% of them. 
Moreover, {\DAA} outperformed {\ADAPTPE} on about 99\% of the test instances.

In summary, the above results confirm the findings of the previous subsections, i.e., {\DAA} is better than its constant stepsize counterpart {\DA}, and {\DAA} is more stable than {\ADAPTPE}.

%These results clearly demonstrate the superior robustness and efficiency of {\DA}-Adapt compared {\DA}-Fixed and {\PE}-Adapt.

\section{Concluding Remarks}\label{sec:concluding}

 We start by making some remarks about the analysis of this paper. Even though we have only considered proximal ADMMs, our analysis also applies to proximal penalty methods. \rev{If the input $C >0$ in {\SA} is chosen such that $C = \rho$, then 
{\SA} will only perform a single Lagrange multiplier update at its last iteration (see  line~\ref{algo.sa.lastLagrangeMultiplier} of {\SA}).
However, it can be easily seen from our convergence analysis that this last multiplier update is not essential and can be removed, thereby yielding a proximal penalty method that never performs Lagrange multiplier updates. Similarly, if $C = \rho$ in {\DAA}, then a Lagrange multiplier update is performed  at 
of the end of each $\ell$-th cycle, i.e.,
the iterations
for which $c=c_{\ell-1}$ (see line \ref{algo.adapt.admm.update.multiplier} of {\DAA}). Since these Lagrange multiplier updates are not essential for our convergence analysis, they can be removed from the description of {\DAA}, resulting in an adaptive proximal penalty method with convergence properties similar to those described in Theorem~\ref{the:dynamic.adaptive}.}

We now discuss some possible extensions of our analysis in this paper. First, recall that {\SA} performs the test in its line 9 to update the Lagrange multiplier, leading to infrequent multiplier updates. It would be interesting to develop proximal ADMM variants with alternative Lagrange multiplier update rules that are more computationally efficient than the one used by {\SA}. Second, it would be 
interesting to develop proximal ADMM variants for composite optimization problems with  block constraints given by $\sum_{t=1}^B g_t (x_t) \le 0$ where
the components of $g_t : \R^{n_t} \to \R^{l}$ are convex for each $t=1,\ldots,B$.
Third, this paper assumes that $\dom h$ is bounded (see assumption (A1)) and $h$ restricted to its domain is Lipschitz continuous (see assumption (A1)). 
It would be interesting to extend its analysis to  the case where these assumptions are removed. Finally, our analysis shows that $(\hat x, \hat p, \hat v, \hat \varepsilon, \hat c)$ 
output by {\DAA} satisfies  $\|A\hat x-b\|\le \eta$ whenever $\hat c={\cal O}(1/\eta)$.
It would be interesting to investigate ADMM variants that guarantees this same $\eta$-feasibility under  weaker conditions on $c$, e.g.,  $c={\cal O}(1)$. Some efforts along this direction have been made in \cite{zhang2020proximal,ErrorBoundJzhang-ZQLuo2020} under more restrictive conditions on problem \eqref{initial.problem} than the ones assumed in this paper.

% of ?????without requiring $c$ to grow indefinitely and under the same general conditions as in our paper.???????

% [Finally, although our method, {\DA}, does not perform Lagrange multiplier updates
% with 
%ADMM variants can asymptotically approach  feasibility  without requiring $c$ to grow indefinitely and under the same general conditions as in our paper
% $(\theta,\chi)=(0,1)$ (see the update formula \eqref{eq:p_update}) at every iteration (and numerical results suggest that this is not a good strategy -- see the results under the column titled {\VA} in tables \ref{tab:problem1} and \ref{tab:problem2}), we wonder if such a proximal ADMM can be developed.

% The proximal ADMM of \cite{KongMonteiro2024} performs Lagrange multiplier update at every iteration but chooses $(\theta,\chi)$ in a very conservative way, namely, satisfying
% \eqref{eq:assumptio:B} and, as our computational results show, does not perform well. 
% Moreover, the iteration-complexity bound
% obtained in \cite{KongMonteiro2024} for the latter method depends linearly on $\theta^{-1}$, and hence grows to infinity as $\theta$ approaches zero.]

\appendix

\section{Technical Results for Proof of Lagrange Multipliers}\label{Appendix:tech.lagra.mult}

This appendix provides some technical results to show
% about convexity 
that under certain conditions the sequence of Lagrange multipliers generated by {\SA} is bounded.

The next two results,  used to prove Lemma~\ref{lem:qbounds-2},  can be found in  \cite[Lemma B.3]{goncalves2017convergence} and 
\cite[Lemma 3.10]{kong2023iteration}, respectively. 

\begin{lemma}\label{lem:linalg} 
Let $A:\R^n \to \R^l$ be a  nonzero linear operator. Then,
\[
\nu^+_A\|u\|\leq \|A^*u\|,   \quad \forall u \in A(\R^n).
\]
\end{lemma}

\begin{lemma}\label{lem:bound_xiN}
Let $h$ be a function as in (A1).
Then, for every $\delta \ge 0$, $z\in {\mathcal H}$,   and $\xi \in \partial_{\delta} h(z)$, we have
\begin{equation*}\label{bound xi}
\|\xi\|{\rm dist}(u,\partial {\mathcal H}) \le \left[{\rm dist}(u,\partial {\mathcal H})+\|z-u\|\right]M_h + \Inner{\xi}{z-u}+\delta \quad \forall u \in {\mathcal H}
\end{equation*} 
where $\partial {\cal H}$ denotes the boundary of ${\cal H}$.
\end{lemma}

The following result, whose statement is in terms of
the $\delta$-subdifferential instead of the classical subdifferential,
is a slight generalization of \cite[Lemma B.3]{sujanani2023adaptive}.
For the sake of completeness, we also include its proof.  

\begin{lemma}\label{lem:qbounds-2}
Assume that
$b \in \R^{l}$,
linear operator $A:\mathbb{R}^n \to \mathbb{R}^l$
and function $h(\cdot)$,
satisfy assumptions (A4)
and (A1), respectively.  
If $(q^-,\varrho) \in A(\R^n) \times (0,\infty)$ and $(z,q,r,\delta) \in  \dom h \times \mathbb{R}^l \times \mathbb{R}^n\times \R_+$ satisfy
\begin{equation}\label{eq:cond:Lem:A1}
q=q^-+\varrho(Az-b)\quad\text{and}\quad  r \in \partial_\delta h(z)+A^{*}q 
, %\textrm{[choose another letter for $\chi$]}
\end{equation}
then we have 
% \begin{equation}\label{q-bound-2}
%  \|q\|\leq \max\left\{\|q^-\|,\frac{2D_h(M_h+\|r\|)+\delta}{\bar d \nu^{+}_A} \right\},
% \end{equation}
\begin{equation}\label{q-bound-2}
 \|q\|\leq \max\left\{\|q^-\|, \ \Xi\Big(\|z-\bar x\|\ , \ \|r\|+\delta \Big) \right\}
\end{equation}
where  $\mathrm{\bar x}$ is as in (A4),
% \[
% \bar \kappa(t_1,t_2) := \frac{2D_h(M_h+t_1)+t_2}{\bar d \nu^{+}_A} \quad \forall (t_1,t_2)
% \in \R^2_+.
% \]
\begin{equation}\label{eq:Technical.varphi.def}
\Xi(s,t) := 
%\upsilon(t) := 
\frac{t+ \left(s+\bar d \, \right)(M_h+t)}{\bar d \nu^{+}_A} \quad \forall (s,t)
\in \R_+^2,
\end{equation}
$M_h$ and $\bar d>0$ are as in (A1) and (A4), respectively, and $\nu^+_A$ is  the smallest positive singular value of $A$.
\end{lemma}	

\begin{proof}
 We first claim that
\begin{equation}\label{ineq:aux9001-2}
 \bar{d}\nu_A^{+}\|q\|
\leq (\|z-\bar x\|+\bar d)\left(M_h + \|r\| \right)  - \Inner{q}{Az-b}+\delta
\end{equation}
holds.
The assumption on $(z,q,r,\delta)$ implies that $r-A^{*}q \in \partial_{\delta} h(z)$. Hence, using the Cauchy-Schwarz inequality, the definitions of $\bar d$ and $\bar x$ in (A4), and
Lemma~\ref{lem:bound_xiN} with $\xi=r-A^{*}q$, and  $u=\bar x$, we have:
 \begin{align}\label{first inequality p-2}
   \bar d\|r-A^{*}q\|-\left[\bar d+\|z-\bar x\|\right]M_h &\overset{\eqref{bound xi}}{\leq}  \Inner{r-A^{*}q}{z-\bar x}+\delta\leq \|r\| \|z-\bar x\| - \Inner{q}{Az-b}+\delta.
 \end{align}
 Now, using the above inequality and 
 the triangle inequality, we conclude that:
 \begin{align}\label{second inequality p-2}
 \bar d \|A^*q\| + \Inner{q}{Az-b}
 &\overset{\eqref{first inequality p-2}}{\leq} \left[\bar d+\|z-\bar x\|\right]M_h + \|r\| \left(\|z-\bar x\| + \bar d\right)+\delta = (\|z-\bar x\|+\bar d)\left(M_h + \|r\| \right)+\delta.
 \end{align}
 We note that $q \in A(\R^n)$  follows immediately from the identity in (77), the hypothesis that $q^- \in A(\R^n)$, and the fact that $b\in \text{Im}(A)$ due to assumption (A4). Hence, inequality
 \eqref{ineq:aux9001-2} now follows
from the above inequality
and
Lemma~\ref{lem:linalg}.

We now prove \eqref{q-bound-2}. 
Relation \eqref{eq:cond:Lem:A1}  implies that $\Inner{q}{Az-b}=\|q\|^2/\varrho-\Inner{q^-}{q}/\varrho$, and hence that
\begin{equation}\label{q relation-2}
    \bar d \nu^{+}_A\|q\|+\frac{\|q\|^2}{\varrho}\leq (\|z-\bar x\|+\bar d)(M_h+\|r\|)+\frac{\Inner{q^-}{q}}{\varrho}+\delta\leq (\|z-\bar x\|+\bar d)(M_h+\|r\|)+\frac{\|q\| }{\varrho}\|q^-\|+\delta,
\end{equation}
where the last inequality is due to the Cauchy-Schwarz inequality.
Now, letting
$W$  denote the right hand side of \eqref{q-bound-2} and using \eqref{q relation-2},
we conclude that
\begin{align}\label{prelim q-bound-2}
\left(\bar d \nu^+_A+\frac{\|q\|}{\varrho}    \right)\|q\|&\overset{\eqref{q relation-2}}{\leq} \left(\frac{(\|z-\bar x\|+\bar d)(M_h+\|r\|)+\delta}{W}+\frac{\|q\| }{\varrho}\right)W\nonumber\\ 
&\leq \left(\frac{(\|z-\bar x\|+\bar d)M_h+(\|z-\bar x\|+\bar d+1)(\|r\|+\delta)}{W}+\frac{\|q\| }{\varrho}\right)W\leq \left(\bar d \nu^+_A+\frac{\|q\|}{\varrho}    \right) W,
\end{align}
and hence that \eqref{q-bound-2} holds.
\end{proof}

We conclude this section with a technical result of convexity which is used in the proof of Lemma \ref{Lemma:ABIPP}.
Its proof can be found in \cite[Lemma A1]{melo2023proximal} but for the sake of completeness
a more detailed proof is given here.

\begin{lemma}\label{conv:result}
		Assume that $\xi>0$, $\psi \in \bConv{n}$ and positive definite real-valued $n \times n$-matrix $Q$ are such that
		$\psi - (\xi/2) \|\cdot\|^2_Q$ is convex and let
		$(y,v,\eta) \in \R^n \times \R^n \times  \R_+$ be such that $v\in \partial_\eta \psi(y)$.
		Then, %for any $ u \in \dom \psi$,
		for any $\tau>0$,
			\begin{equation}\label{eq:auxlemA1}
		\psi(u) \ge \psi(y) + \Inner{v}{u-y} - (1+\tau^{-1})\eta+ \frac{(1+\tau)^{-1}\xi}{2} \|u-y\|_Q^2  \quad \forall u \in \R^n.
		\end{equation}
	\end{lemma}

\begin{proof}
		Let $\psi_v := \psi-\Inner{v}{\cdot}$.
		The assumptions imply that $\psi_v$ has a unique global minimum $\bar y $
		and that
		\begin{equation}\label{ineq:u}
		\psi_v(u)\ge \psi_v(\bar y)+\frac{\xi}{2}\|u-\bar y\|_Q^2  \ge \psi_v(y)-\eta  +\frac{\xi}{2}\|u-\bar y\|_Q^2
		\end{equation}
		for every $u \in \R^n$. 
The above inequalities with $u=y$
 imply that $(\xi/2) \|\bar y-y\|_Q^2 \le \eta$. 
On the other hand, for any $\tilde u,u'\in \R^n$ and $\tau>0$, it holds
\begin{align*}
\|\tilde u+u'\|^2 
&= \|\tilde u\|^2 + \|u'\|^2 + 2 \Inner{\frac{1}{\sqrt{\tau}} \tilde u}{\sqrt{\tau} u'} 
\le \|\tilde u\|^2 + \|u'\|^2+\frac{1}{\tau}\|\tilde u\|^2+\tau \|u'\|^2\\
&=(1+\tau) \|u'\|^2 + (1+\tau^{-1}) \|\tilde u\|^2    
\end{align*}
which implies in
\begin{align*}
(1+\tau)^{-1} \|\tilde u+ u'\|^2&\leq \|u'\|^2+(1+\tau)^{-1}(1+\tau^{-1}) \|\tilde u\|^2=\|u'\|^2+\tau^{-1} \|\tilde u\|^2.
\end{align*}
Hence, adding and subtracting the term $(\tau^{-1}\xi/2)\|\bar y-y\|_Q^2$ in the right hand side of \eqref{ineq:u}
 % \begin{align}\label{eq:aux:app}
	% 	\psi_v(u)
	% 	\ge  \psi_v(y)-\eta-\frac{\tau^{-1}\xi}{2}\|\bar y-y\|_Q^2 +\frac{\xi}{2}  \left(  \tau^{-1} \|y-\bar y\|_Q^2+\|u-\bar y\|_Q^2 \right).
 % \end{align}
and using the previous inequality with $\tilde u=u-\bar{y}$ and $u'=\bar{y}-y$, we obtain that
		\begin{align*}
		\psi_v(u)&
		\ge  \psi_v(y)-\eta-\frac{\tau^{-1}\xi}{2}\|\bar y-y\|_Q^2 +\frac{\xi}{2}  \left(  \tau^{-1} \|y-\bar y\|_Q^2+\|u-\bar y\|_Q^2 \right) \nonumber \\
		&\ge \psi_v(y)-(1+\tau^{-1})\eta+\frac{(1+\tau)^{-1}\xi}{2}\|u-y\|_Q^2 \end{align*}
 		for every $u \in \R^n$.		
Hence, \eqref{eq:auxlemA1}  follows  from the above conclusion and the  definition of $\psi_v$.	
\end{proof}

\section{ADAP-FISTA algorithm} \label{sec:acg}
This appendix section presents an adaptive variant of ACG, called ADAP-FISTA, for solving \eqref{ISO:problem} under the assumption that (B1), (B2), and
 $\nabla \psis(\cdot)$ is $\tilde M$-Lipschitz 
continuous, i.e., 
% that there exists $\bar L \geq 0$ such that
\begin{equation}\label{ineq:uppercurvature1}
\|\nabla \psis(z') -  \nabla \psis(z)\|\le \tilde M \|z'-z\| \quad \forall z,z' \in \R^n.
\end{equation}
% tool in the development of the AS-PAL method. 
% We first introduce the assumptions on the problem  it solves. 
% ADAP-FISTA considers the following problem
% \begin{equation}\label{OptProb1}
% \min\{ \psi(x):= \psi_s(x) + \psi_n(x) : x \in \R^n\}
% \end{equation}
% where $\psi_s$ and $\psi_n$ are assumed to satisfy the following assumptions:
% \begin{itemize}
% \item[\textbf{(I)}] $\psi_n:\R^n \to \R\cup\{+\infty\}$ is a possibly nonsmooth convex function;
% \item[\textbf{(II)}]
% $\psi_s: \R^n\to \R$ is a differentiable function and
% there exists $\bar L \geq 0$ such that
% \begin{equation}\label{ineq:uppercurvature1}
% \|\nabla \psi_s(z') -  \nabla \psi_s(z)\|\le \bar L \|z'-z\| \quad \forall z,z' \in \R^n.
% \end{equation}
% \end{itemize}
We would like to emphasize that the notations introduced in this appendix, related to the ADAP-FISTA, are local to this section and should not be confused with those used in previous sections. 
These choices are made to remain consistent with the original presentation of the algorithm in \cite[Appendix A]{sujanani2023adaptive}, and they do not carry the same interpretation as in the rest of the paper.

Before formally stating ADAP-FISTA, we give some comments.
% In addition to the pair $(\sqrt{\sigma};z^0)$ as in Definition~\ref{def:BB}, 
{\ADAP} requires
% which requires a pair of positive scalars
% $(L_0,\tilde \mu) $ to be given as input. The first 
as input an arbitrary  positive estimate  $M_0$ 
for the unknown parameter $\tilde M$.
Moreover, {\ADAP} is
a variant of SFISTA \cite{beck2009fast,Amirbeck2009fast,nesterov1983method}, which in turn
 is an
ACG variant that solves instances of \eqref{ISO:problem} with $\psis$ strongly convex and that requires the availability of a strong convex parameter for $\psis$.
 Since
 {\ADAP} is an enhanced version of SFISTA,
 it also uses as input a good guess  $ \mu_0$ for what is believed to be a  strong convex parameter of $\psis$  (even though such parameter may  not exist as $\psis$ is not assumed to be strongly convex).
In other words, {\ADAP} is used 
under the belief that $\psis$ is 
$\mu_0$-strongly convex. If a key test inequality within {\ADAP} fails to be satisfied then it stops without  finding a $(\sqrt{\sigma};x_0)$-relative  stationary solution of \eqref{ISO:problem}, but reaches the important conclusion  that $\psis$ is not $\mu_0$-strongly convex.

% We now describe the type of approximate solution that ADAP-FISTA aims to find. 

%\vgap

% \noindent{\bf Problem A:} 

% We now recall the type of approximate solution that ADAP-FISTA aims to find.
% Given  $\psi$ as in \eqref{ISO:problem} and satisfying assumptions (B1) and (B2),
%  a point $x_0 \in \dom \psi^{(\text{n})}$, a parameter $\sigma\in(0,\infty)$,
% the problem is to find a pair $(y,u) \in \dom \psi^{(\text{n})} \times \R^n$ such that
% \begin{gather*}
%      \|u\|\leq \sigma \|y-x_0\|, \quad u \in \nabla \psis(y)+\partial \psi^{(\text{n})}(y).\label{acg problem}
%  \end{gather*}

We are now ready to present the ADAP-FISTA algorithm below.

\noindent\begin{minipage}[t]{1\columnwidth}%
\rule[0.5ex]{1\columnwidth}{1pt}

\noindent \textbf{ADAP-FISTA Method}

\noindent \rule[0.5ex]{1\columnwidth}{1pt}%
\end{minipage}

\begin{itemize}[labelsep=0.5em, left=0pt, align=left, labelwidth=4em]
\item[{\bf Input:}] $(x_0, M_0, \mu_0, \sigma)\in \dom \psin\times \R_{++}\times \R_{++}\times \R_{++}$ such that $M_0>\mu_0$.

% \item[{\bf 0.}] Let initial point $ x_0 \in \dom \psi^{(\text{n})}$ and scalars  $M_0>\mu_0>0$, 
\item[{\bf 0.}] Let $\chi \in (0,1)$ and $\beta>1$ be given, and set $ y_0=x_0 $, $ A_0=0 $, $\tau_0=1$, and $ j=0$.

\item[{\bf 1.}] Set $M_{j+1}=M_j$.
\item[{\bf 2.}]	Compute
		\begin{equation*}\label{def:ak-sfista1}
		a_j=\frac{\tau_{j}+\sqrt{\tau_{j}^2+4\tau_{j} A_{j}(M_{j+1}-\mu_0)}}{2(M_{j+1}-\mu_0)}, \quad \tx_{j}=\frac{A_{j}y_{j}+a_j x_j}{A_j+a_j},
		\end{equation*}
		\begin{equation}
		y_{j+1}:=\underset{v\in \dom \psin}\argmin\left\lbrace q (v;\tx_{j},M_{j+1}) 
		:= \psis(\tilde x_j)+\langle \nabla \psis(\tilde x_j), v-\tilde x_j\rangle + \psin(v) + \frac{M_{j+1}}{2}\|v-\tx_{j}\|^2\right\rbrace.
		\label{eq:ynext-sfista1}
		\end{equation}
		If  the inequality
        \begin{equation}\label{ineq check}
		\psis(\tilde x_j)+\langle \nabla \psis(\tilde x_j), y_{j+1}-\tilde x_j\rangle+\frac{(1-\chi) M_{j+1}}{2}\|y_{j+1}-\tilde x_{j}\|^2\geq \psis(y_{j+1})
		\end{equation}
		holds, then go to step~3; else, set $M_{j+1} \leftarrow \beta M_{j+1} $ and repeat step~2.
\item[{\bf 3.}] Compute
\begin{align*}
A_{j+1}&=A_j+a_j, \quad \tau_{j+1}= \tau_j + a_j\mu_0,  \\
s_{j+1}&=(M_{j+1}-\mu_0)(\tilde x_j-y_{j+1}),\\
\quad x_{j+1}&= \frac{1}{\tau_{j+1}} \left[\mu_0 a_j y_{j+1} + \tau_j x_j-a_js_{j+1} \right].
\end{align*}
	
\item[{\bf 4.}]  If the inequality
\begin{equation}\label{ineq: ineq 1}
\|y_{j+1}-x_{0}\|^{2} \geq \chi A_{j+1}M_{j+1} \|y_{j+1}-\tilde x_{j}\|^2, 
\end{equation}
holds, then  go to step 5; otherwise, stop with {\bf failure}.

\item[{\bf 5.}] Compute 
\begin{equation*}\label{def:uk}
u_{j+1}=\nabla \psis(y_{j+1})-\nabla \psis(\tilde x_{j})+M_{j+1}(\tilde x_j-y_{j+1}).   
\end{equation*} If the inequality
		\begin{equation}\label{u sigma criteria}
		\|u_{j+1}\| \leq \sqrt{\sigma} \|y_{j+1}-x_0\|
		\end{equation}
		holds, then stop with {\bf success} and output $(y,u):=(y_{j+1},u_{j+1})$; otherwise,
		 $ j \leftarrow j+1 $ and go to step~1.
\end{itemize}
\noindent \rule[0.5ex]{1\columnwidth}{1pt}

We now make some remarks about ADAP-FISTA. 
First, 
% SFISTA for solving the strongly convex version of \eqref{ISO:problem} consist of repeatedly invoking only 
steps 2 and 3 of ADAP-FISTA appear in the usual SFISTA for solving strongly convex version of \eqref{ISO:problem}, either
with a static Lipschitz constant
(i.e., $M_{j+1}=\tilde M$ for all $j \ge0$),
or with adaptive line search for $M_{j+1}$ (e.g., as in step 2 of ADAP-FISTA).
Second, the pair $(y_{j+1},u_{j+1})$ always satisfies the inclusion in \eqref{ISO:Cond:1&2} (see \cite[Lemma A.3]{sujanani2023adaptive}); hence, if ADAP-FISTA stops successfully in step 5, then
 the triple $(y_{j+1},u_{j+1},0)$ is a $(\sigma, x_0)$-relative  stationary solution of \eqref{ISO:problem},
 due to \eqref{u sigma criteria}.
Finally, if condition \eqref{ineq: ineq 1} in step 4 is never violated, then ADAP-FISTA must stop successfully in step 5 (see Proposition~\ref{pro.inexact.sol} below).

The following result describes the main properties of {\ADAP}.

\begin{prop}\label{pro.inexact.sol}
Assume that (B1) and (B2) hold and that
$\nabla \psis(\cdot)$ is $\tilde M$-Lipschitz 
continuous.
% Let $(L_0,\mu_0) \in \R^2_{+}$ be the input for ADAP-FISTA,  $(\tau;z^0)\in\R_{++}\times \dom \psin$ be given.
Then, the following statements about the {\ADAP} method 
% of \cite[Appendix A]{sujanani2023adaptive}
with arbitrary input $(x_0, M_0, \mu_0, \sigma)\in \dom \psin\times \R_{++}\times \R_{++}\times \R_{++}$  
% $(\sigma, z^0,M_0,\mu_0) \in \R_{++}\times \dom \psin \times \R_{++}\times \R_{++} $
hold:
\begin{itemize}
    \item[(a)]  it
always stops (with either success or failure) in at most 
\begin{equation}\label{complex:Arnesh:Monteiro}
 {\cal O}\left( \sqrt{\frac{\tilde M +M_0}{ \mu_0}}\max\left\{\log_2(\sigma^{-1/2}\tilde M), 1\right\}%\log^+_1 (\sigma^{-1/2}\tilde M)
 \right) 
 \end{equation}
iterations/resolvent evaluations;
    \item[(b)] if it stops successfully with output $(y,u)$, then
    the triple $(y,u,0)$ is a $(\sqrt{\sigma};x_0)$-relative stationary solution of \eqref{ISO:problem} (see  Definition \ref{def:BB});
    % i.e., the triple $(\bar z,\bar r,\bar\varepsilon)=(y,u,0)$;
    \item[(c)] if $\psis(\cdot)$ is $ \mu_0$-strongly convex, then {\ADAP} always terminates successfully, and therefore
    with a $(\sqrt{\sigma};x_0)$-relative  stationary solution of \eqref{ISO:problem}.
\end{itemize}
\end{prop}

We now make some remarks about Proposition~\ref{pro.inexact.sol}.
First, if {\ADAP} fails then it follows from Proposition~\ref{pro.inexact.sol}(c) that $\psis$ is not $\tilde \mu$-strongly convex.
Hence, failure of the method sends the message that $\psis$ is not ``desirable", i.e., is far from being $\tilde \mu$-strongly convex.
Second, if {\ADAP} successfully terminates (which can happen even if $\psis$ is not $\tilde \mu$-strongly convex), then Proposition~\ref{pro.inexact.sol}(b) guarantees that it finds the desired relative  stationary solution.
Third, if
    $\sigma^{-1/2} = {\cal O}(1)$ and
$M_0 = {\cal O}(\tilde M)$, then \eqref{complex:Arnesh:Monteiro} reduces to
$ {\cal O}((\tilde M/\mu_0)^{1/2})$.

\section{Inexact Solution Concept}\label{Appendix:inexact.sol}

This section shows how  near-stationary solutions, which are  absolute analogues of the ones considered in Definition~\ref{def:BB}, yield points with nearly nonnegative 
directional derivatives along all unit directions.
For a given function $\phi$ and $y \in \dom \phi$ such that the directional derivative $\phi'(y;u-y)$ is well-defined for every $u \in \R^n$, define
\begin{align}\label{def:inf.derivative}
\Theta(y;\phi) := - \inf_{u\in\R^n}\{ \phi'(y;u-y) \, : \, \|u-y\|\le 1\}.
\end{align}

Clearly, $\Theta(y;\phi) \ge 0$ and equality holds if and only if
$\phi'(y;u-y) \ge 0$ for every $u \in \R^n$. If $\phi \in \bConv{n}$, the latter condition is equivalent to $y$ being an optimal solution of $\phi$, or equivalently, the inclusion
$0 \in \partial \phi(y)$. More generally, a point $y$ such that $\Theta(y;\phi) = 0$ is referred to as a stationary point of $\phi$.

More generally, $y$ is called a directional near-stationary point when $\Theta(y;\phi)$ is near zero.
This section discusses how an absolute analogue of the stationary condition in Definition~\ref{def:BB} yields directional near-stationary points
and related ones expressed in terms of subdifferentials.

The main result of this section is Proposition~\ref{Appendix.Prop.C.main}, stated in the setting of a nonconvex composite optimization problem.
Its proof requires two preliminary technical lemmas.

\begin{lemma}
\label{prop:eps_to_stationarity}
Let $\lam>0$, $\varepsilon \ge 0$, 
function
$\phi\in \bConv{n}$,
and $x\in \dom \phi$ such that $0\in \partial_\varepsilon\phi(x)$ be given,
and let  $y$ denote the unique optimal solution of the strongly convex optimization problem
\begin{align}\label{eq.lemma.C1.optm}
y := \arg\min_{u \in \mathbb{R}^n}
\left\{ \phi_\lam (u) := \phi(u) + \frac{1}{2\lambda}\|u-x\|^2 \right\}.
\end{align}
Then,  $y \in \mathbb{R}^n$ satisfies $\|y-x\| \le \sqrt{\varepsilon\lam}$.
\end{lemma}

\begin{proof}
Let $\lambda>0,$ $\varepsilon\ge 0$, and $x\in \dom \phi$ such that $0\in \partial_\varepsilon\phi(x)$ be given. Using the definition of $y$ and the fact that $ \phi_\lam$ is $(1/\lam)$-strongly convex, we have
 $\phi_\lam(u) \ge \phi_\lam(y) + \|u-y\|^2/(2\lambda)$ for every $u \in \R^n$, or equivalently,

\[
\phi(u) + \frac{1}{2\lambda}\|u-x\|^2
\ge
\phi(y) + \frac{1}{2\lambda}\|y-x\|^2
+ \frac{1}{2\lambda}\|u-y\|^2 \quad \forall u\in\R^n.
\]
The above inequality with $u=x$  implies that
$\|y-x\|^2/\lam
\le \phi(x) - \phi(y)
\le \varepsilon$,
where the last inequality is due to the assumption that $0\in \partial_\varepsilon\phi(x)$ and the definition of the $\varepsilon$-subdifferential in \eqref{e-subd}. 
Thus, the conclusion of the lemma holds.
\end{proof}

The proof of the next well-known result can be found
for example in \cite[Lemma F.1.2]{kong2021thesis}
(see also Chapter 8 of \cite{RockafellarWets2009} and \cite{BurkeMore1988}).
\begin{lemma}\label{lemma:C:2}
Let $\psin\in\bConv{n}$, (possibly nonconvex) differentiable function $\psis$ on $\dom \psin$,
and $(y,w) \in \dom \psin\times \R^n$ such that $w\in \nabla \psis(y) + \partial \psin(y)$ be given. Then,
\begin{align*}
\Theta(y;f+h) = \dist\left(0, \nabla \psis(y) + \partial \psin(y) \right)\leq \|w\|,
\end{align*}
where $\Theta(y;f+h)$ is as in~\eqref{def:inf.derivative}.
\end{lemma}

We are now ready to state the main result of this subsection.

\begin{proposition}\label{Appendix.Prop.C.main}
Let $\psin\in\bConv{n}$, (possibly nonconvex) differentiable function $\psis$ on $\dom \psin$,
and $(x,r,\varepsilon) \in \dom \psin\times \R^n\times \R_+$ such that
\begin{align}\label{eq.lemma.composite.proof1}
r \in \nabla \psis(x) + \partial_\varepsilon \psin(x),  
\end{align}
be given, and define
\begin{align}
y=y(x,r) &:= \arg\min_{u \in \mathbb{R}^n}
\left\{ \inner{\nabla \psis(x)-r}{u } + \psin(u) + \frac{1}{2}\|u-x\|^2 \right\}, \label{eq.def.C.y}\\
w=w(x,r) &:= x-y + r + [ \nabla \psis(y)-\nabla \psis(x)]. \label{eq.def.C.w}
\end{align}
Then, the following statements hold:
\begin{itemize}
\item[a)] 
$(y,w) =(y(x,r),w(x,r)) \in \dom \psin \times \R^n$ satisfies
\begin{align}
\label{eq.lemma.composite.proof20}
w\in \nabla \psis(y) + \partial \psin(y), \qquad \|y-x\|\leq \sqrt{\varepsilon},
\end{align}
and the inequality
\begin{align}
\label{eq.lemma.composite.proof2}
\|w\| \leq  \|\nabla \psis(y)-\nabla \psis(x)\| + \sqrt{\varepsilon} + \|r\|.
\end{align}
\item[b)]
if, in addition,  $\dom \psin$ is compact and $\nabla \psis$ is continuous on $\dom \psin$, then
for any $\eta>0$, there exists $\delta>0$ satisfying the following property:
the pair  $(y,w)=(y(x,r),w(x,r))$ associated with
any $(x,r,\varepsilon)$
such that inclusion \eqref{eq.lemma.composite.proof1}
and the inequality $\|r\|^2 + 2 \varepsilon \le \delta^2/2$ holds, 
  satisfies
\begin{align}\label{eq.lemma.Append.C.b}
\|y-x\| \le \eta , \qquad   \|w\| \leq  \eta;
\end{align}
as a consequence,
\begin{align}\label{eq.lemma.Append.C.b.2}
\Theta(y;\psis+\psin)
= \dist\left(0, \nabla \psis(y) + \partial \psin(y) \right) \leq  \eta.
\end{align}
\end{itemize}
\end{proposition}

\begin{proof}
(a) The inclusion in~\eqref{eq.lemma.composite.proof20} follows from the fact that $y$ satisfies the optimality condition for \eqref{eq.def.C.y} and
the definition of $w$ in \eqref{eq.def.C.w}.
Now, define
\[
\phi(\cdot)
:=  \inner{\nabla \psis(x)-r}{\, \cdot } + \psin(\cdot)
\]
and note that inclusion in~\eqref{eq.lemma.composite.proof1} implies that $0 \in \partial_\varepsilon \phi(x)$. Proposition~\ref{prop:eps_to_stationarity} with $\lambda=1$ and the definition of $y$ in~\eqref{eq.def.C.y} then imply that the inequality in \eqref{eq.lemma.composite.proof20} holds.
Now, the inclusion in \eqref{eq.lemma.composite.proof20} and Lemma \ref{lemma:C:2} imply that the first inequality in \eqref{eq.lemma.composite.proof2} holds. 
Finally, the definition of $w$, the triangular inequality, and the inequality in \eqref{eq.lemma.composite.proof20}, imply that the second inequality in \eqref{eq.lemma.composite.proof2} also holds.
We have proved that (a) holds.

(b) Let $\eta>0$ be given. The additional assumptions made on this statement imply that $\nabla \psis$ is uniformly continuous on $\dom \psin$. Thus, there exists $\rho>0$ such that
\begin{align}\label{eq.AppenC.HC}
\|z-x\| \le \rho \Rightarrow \|\nabla \psis(x)-\nabla \psis(z)\| \le \frac{\eta}2.
\end{align}
We now show that the scalar $\delta := \min\{2\rho,\eta/2\}$ fulfills the conclusion of this statement. Indeed, let triple $(x,r,\varepsilon)$ satisfying inclusion \eqref{eq.lemma.composite.proof1}
and the inequality $\|r\|^2 + 2 \varepsilon \le \delta^2/2$ be given. Clearly, the last inequality, together with the definition of $\delta$ and the inequality in~\eqref{eq.lemma.composite.proof20}, implies that
\[
\|r\| + \sqrt{\varepsilon} \le \left( 2 \|r\|^2 + 2 \varepsilon \right)^{1/2} \le  \delta \le  \frac{\eta}2,
\qquad
\|y-x\| \le \sqrt{\varepsilon} \le \sqrt{\frac{\delta^2}4} = \frac{\delta}2 \le \min\{\rho,\eta\}
\]
% $\|r\| + \sqrt{\varepsilon} \le \eta/2$ and $\|y-x\| \le \delta/2$.
These two inequalities,
the inequalities in \eqref{eq.lemma.composite.proof2},
and implication \eqref{eq.AppenC.HC} with $z=y$, then show that \eqref{eq.lemma.Append.C.b} holds. Results in~\eqref{eq.lemma.Append.C.b.2} follow by Lemma~\ref{lemma:C:2} and the last inequality in~\eqref{eq.lemma.Append.C.b}.
\end{proof}
\vspace{.5em}

Note that if the function $\psis$ in Proposition~\ref{Appendix.Prop.C.main} is also $L$-smooth on $\dom \psin$, then it follows from~\eqref{eq.lemma.composite.proof2} that
\begin{align*}
 \|w\| 
&\leq \sqrt{\varepsilon}(L+1) + \|r\|.
\end{align*}
Hence, without the compactness assumption on $\dom \psin$, it can be easily seen that the  conclusion of statement (b) holds if
$\delta$ is chosen as $\delta := \eta/[2(L+1)]$.

Finally, the usual subdifferential for  convex functions has been generalized to
functions of the form $\psis+\psin$ where $(\psis,\psin)$ are as in Lemma~\ref{lemma:C:2} or Proposition~\ref{Appendix.Prop.C.main} (e.g., see \cite[Ch. 10]{rockafellar1998variational},
\cite{mordukhovich2006variational}). 
This more general subdifferential, denoted by $\tilde \partial (\psis+\psin)$, has the property that
\[
\tilde \partial (\psis+\psin)(y)= \nabla \psis(y) + \partial \psin(y) \quad \forall y \in \dom \psin.
\]
Hence, the above inequalities involving the quantity
$\dist\left(0, \nabla \psis(y) + \partial \psin(y) \right)$ can equivalently be rewritten in terms of
$\dist (0, \tilde \partial(\psis + \psin)(y))$.

\bibliography{relax_admm_ref}

\end{document}

%% file: def.tex
\usepackage{amsmath}
\usepackage{amsfonts}
\usepackage{latexsym}
\usepackage{amssymb}
\usepackage{framed}

\newtheorem{thm}{Theorem}[section]
\newtheorem{theorem}[thm]{Theorem}

\newtheorem{lemma}[thm]{Lemma}
\newtheorem{prop}[thm]{Proposition}
\newtheorem{proposition}[thm]{Proposition}

\newtheorem{definition}[thm]{Definition}

\newtheorem{remark}[thm]{Remark}

\newcommand{\beq}{\begin{equation}}
\newcommand{\eeq}{\end{equation}}
\newcommand{\beqa}{\begin{eqnarray}}
\newcommand{\eeqa}{\end{eqnarray}}
\newcommand{\beqas}{\begin{eqnarray*}}
\newcommand{\eeqas}{\end{eqnarray*}}
\newcommand{\bi}{\begin{itemize}}
\newcommand{\ei}{\end{itemize}}

\newcommand{\vgap}{\vspace{.1in}}

\newcommand{\Ik}{{\mathcal{I}}_k}

\newcommand{\R}{\mathbb{R}}

\newcommand{\lam}{{\lambda}}

\newcommand{\chione}{\zeta_1}
\newcommand{\chitwo}{\zeta_2}
\newcommand{\chisum}{\zeta_1+c\zeta_2}

\newcommand{\inner}[2]{\langle #1,#2\rangle}
\newcommand{\Inner}[2]{\left \langle #1\,,#2 \right \rangle}

\newcommand{\argmin}{\mathrm{argmin}\,}

\newcommand{\interior}{\mathrm{int}\,}

\newcommand{\dom}{\mathrm{dom}\,}

\newcommand{\bConv}[1]{\overline{\mbox{\rm Conv}}\,(\R^{#1})}

\newcommand{\tx}{\tilde x}
\newcommand{\ty}{\tilde y}
\newcommand{\tT}{\tilde T}

\newcommand{\tu}{\tilde u}
\newcommand{\tq}{\tilde q}
\newcommand{\tv}{\tilde v}
\newcommand{\tp}{\tilde p}

%% file: DQP-onenorm.tex
\begin{table}[ht]
\centering
\scriptsize
\begin{adjustbox}{max width=\textwidth}
\begin{tabular}{c c c | r r r | r r r}
\hline
\multicolumn{3}{c|}{Instance} &
\multicolumn{3}{c|}{{\DAA}} &
\multicolumn{3}{c}{{\DA}} \\
$\omega$ & $B$ & $(n,m)$ &
Iters / Mult & Time & $f+h$ &
Iters / Mult & Time & $f+h$ \\
\cmidrule(lr){1-3}\cmidrule(lr){4-6}\cmidrule(lr){7-9}

% ======================= omega = 100 =======================
100 & 2 & (25,25) &
\textbf{4830}/19 & \textbf{2.046} & -6.058e+03 &
+43\%/16 & +352\% & -6.058e+03 \\
    &   & (50,50) &
\textbf{7474}/19 & \textbf{4.076} & -7.936e+02 &
+51\%/22 & +1038\% & -7.936e+02 \\
&  & (50,75) &
\textbf{12558}/17 & \textbf{7.788} & -2.899e+02 &
+129\%/22 & +2172\% & -3.815e+02 \\ \\

    & 5 & (25,10) &
\textbf{2598}/30 & \textbf{1.249} & -2.205e+04 &
+215\%/25 & +671\% & -2.102e+04 \\
    &   & (25,25) &
\textbf{5241}/21 & \textbf{2.744} & -2.126e+04 &
+48\%/22 & +280\% & -2.108e+04 \\
    &   & (50,50) &
\textbf{7260}/17 & \textbf{7.098} & -9.451e+03 &
+98\%/27 & +1169\% & -9.408e+03 \\
    &  & (50,75) &
\textbf{15907}/17 & \textbf{15.287} & -8.303e+03 &
+102\%/15 & +1966\% & -8.055e+03 \\ \\

    & 10 & (25,10) &
\textbf{1051}/18 & \textbf{0.620} & -4.913e+04 &
+444\%/73 & +1649\% & -4.960e+04 \\
    &    & (25,25) &
\textbf{2061}/13 & \textbf{1.623} & -5.257e+04 &
+749\%/14 & +3019\% & -5.051e+04 \\
    &    & (50,50) &
\textbf{7078}/12 & \textbf{10.488} & -2.678e+04 &
+229\%/12 & +2766\% & -2.715e+04 \\
&  & (50,75) &
\textbf{7818}/79 & \textbf{12.764} & -2.651e+04 &
+142\%/47 & +1225\% & -2.406e+04 \\

\hline
% ======================= omega = 1000 =======================
1000 & 2 & (25,25) &
\textbf{8086}/16 & \textbf{3.634} & -2.985e+05 &
+58\%/16 & +327\% & -2.985e+05 \\
     &   & (50,50) &
\textbf{6948}/14 & \textbf{4.145} & -1.670e+05 &
+110\%/18 & +1420\% & -1.717e+05 \\ \\

     & 5 & (25,25) &
\textbf{12919}/18 & \textbf{8.411} & -1.993e+06 &
+55\%/38 & +308\% & -1.972e+06 \\
     &   & (25,50) &
\textbf{8796}/18 & \textbf{6.113} & -8.917e+05 &
+26\%/18 & +344\% & -1.094e+06 \\
     &   & (50,50) &
\textbf{6407}/20 & \textbf{5.883} & -1.127e+06 &
+159\%/18 & +1508\% & -1.257e+06 \\
 &  & (50,75) &
+14\%/102 & \textbf{26.804} & -7.187e+05 &
\textbf{20902}/62 & +657\% & -7.265e+05 \\ \\
     & 10 & (25,10) &
\textbf{1508}/31 & \textbf{1.009} & -6.565e+06 &
+412\%/15 & +724\% & -6.762e+06 \\
     &    & (25,25) &
\textbf{3270}/33 & \textbf{2.990} & -4.490e+06 &
+324\%/18 & +750\% & -4.275e+06 \\
     &    & (25,50) &
\textbf{6802}/23 & \textbf{7.481} & -4.101e+06 &
+181\%/45 & +993\% & -3.531e+06 \\
     &    & (50,50) &
\textbf{10545}/50 & \textbf{19.049} & -3.197e+06 &
+174\%/16 & +1379\% & -3.356e+06 \\
     &  & (50,75) &
+229\%/20 & \textbf{191.972} & -3.999e+06 &
\textbf{27809}/25 & +126\% & -4.039e+06 \\
\hline
\end{tabular}
\end{adjustbox}
\caption{Performance of {\DAA} and {\DA} for the DQP problem.}
\label{tab:P1-onenorm}
\end{table}

%% file: R2CauchysimplexTable.tex
    \begin{table}[ht]
  \centering
  \begin{adjustbox}{max width=\textwidth}
    \begin{tabular}{cccrrrrrrrrrrrr}
    \hline
    \multicolumn{3}{c}{Instance} & \multicolumn{2}{c}{\DAA} & \multicolumn{2}{c}{\DA} & \multicolumn{2}{c}{\ADAPTPE} & \multicolumn{2}{c}{\CONSTPE} & \multicolumn{2}{c}{\ADAPTVA} & \multicolumn{2}{c}{\CONSTVA}  
 \\    $\omega$ & $B$ & $(n,m)$ & Iters/Mults & Time & Iters/Mults & Time & Iters       & Time  & Iters       & Time \\
    \cmidrule(lr){1-3} \cmidrule(lr){4-5} \cmidrule(lr){6-7} \cmidrule(lr){8-9} \cmidrule(lr){10-11} \cmidrule(lr){12-13} \cmidrule(l){14-15}
    100 & 2 & (10,10) & \textbf{158} / 1&+18\% &+24059\% / 1&+107\% / 1&+3\% & \textbf{0.635} &+24541\% &+75\% &+101\% &+41\% & * & * \\
     &  & (25,10) & \textbf{130} / 1& \textbf{0.093} &+55649\% / 1&+2911\% / 1&+3\% &+5\% &+56682\% &+2652\% &+178\% &+159\% & * & * \\
     &  & (25,25) & \textbf{712} / 5&+17\% & * & * &$<$1\% & \textbf{1.639} & * & * &+139\% &+123\% & * & * \\
     &  & (50,50) & \textbf{755} / 4&+6\% & * & * &+3\% & \textbf{2.075} & * & * &+2718\% &+7427\% & * & * \\
     &  & (50,75) & \textbf{7194} / 4& \textbf{42.344} & * & * &+4\% &+6\% & * & * & * & * & * & * \\
     &  & (100,100) & \textbf{1801} / 5&$<$1\% & * & * &$<$1\% & \textbf{5.148} & * & * &+2351\% &+1074\% & * & * \\
     &  & (100,150) & \textbf{5449} / 4&+5\% & * & * &$<$1\% & \textbf{258.139} & * & * & * & * & * & * \\
\\     & 5 & (10,10) & \textbf{146} / 1&+1642\% &+14593\% / 3&+6\% / 3&+3\% &+1494\% &+14873\% & \textbf{1.570} &+9\% &+62\% & * & * \\
     &  & (25,10) & \textbf{55} / 1& \textbf{0.167} &+101656\% / 1&+2935\% / 1&+4\% &+3\% &+102953\% &+2760\% &+47\% &+41\% & * & * \\
     &  & (25,25) & \textbf{306} / 1& \textbf{0.439} &+22699\% / 1&+1475\% / 1&+2\% &$<$1\% &+23135\% &+1409\% &+276\% &+1810\% & * & * \\
     &  & (50,50) & \textbf{597} / 1& \textbf{1.401} & * & * &+2\% &+2\% & * & * &+1380\% &+1220\% & * & * \\
     &  & (50,75) & \textbf{1180} / 1&+1\% & * & * &+1\% & \textbf{59.888} & * & * &+293\% &+12\% & * & * \\
     &  & (100,100) & \textbf{575} / 1&+2\% & * & * &+2\% & \textbf{13.309} & * & * &+1444\% &+146\% & * & * \\
     &  & (100,150) & \textbf{551} / 1& \textbf{12.601} & * & * &+2\% &$<$1\% & * & * &+6456\% &+1513\% & * & * \\
\\     & 10 & (10,10) & \textbf{114} / 1&+77\% &+16081\% / 1&+4\% / 1&+3\% &+78\% &+16340\% & \textbf{2.863} &+88\% &+156\% &+25539\% &+82\% \\
     &  & (25,10) &+15\% / 1&+783\% &+22886\% / 1&+769\% / 1&+17\% &+817\% &+23265\% &+736\% & \textbf{162} & \textbf{0.843} & * & * \\
     &  & (25,25) & \textbf{184} / 1& \textbf{3.571} &+30884\% / 1&+243\% / 1&+3\% &$<$1\% &+31374\% &+231\% &+386\% &+328\% & * & * \\
     &  & (50,50) & \textbf{431} / 1&+7\% &+20886\% / 1&+401\% / 1&+2\% & \textbf{6.153} &+21174\% &+385\% &+690\% &+171\% & * & * \\
     &  & (50,75) & \textbf{455} / 1&+1\% & * & * &+2\% & \textbf{13.017} & * & * &+366\% &+18\% & * & * \\
     &  & (100,100) & \textbf{427} / 1&+12\% & * & * &+2\% & \textbf{21.062} & * & * &+1453\% &+1825\% & * & * \\
     &  & (100,150) & \textbf{507} / 1&+1\% & * & * &+2\% & \textbf{32.618} & * & * &+2473\% &+237\% & * & * \\
    \hline
    1000 & 2 & (10,10) &+2\% / 5&+2\% & * & * & \textbf{3377} & \textbf{29.445} & * & * &+1478\% &+602\% & * & * \\
     &  & (25,10) & \textbf{1339} / 1& \textbf{1.224} & * & * &$<$1\% &+3\% & * & * &+3310\% &+4016\% & * & * \\
     &  & (25,25) &$<$1\% / 1&$<$1\% & * & * & \textbf{4468} & \textbf{2.403} & * & * & * & * & * & * \\
     &  & (50,50) &$<$1\% / 1&+2\% & * & * & \textbf{3080} & \textbf{6.456} & * & * & * & * & * & * \\
     %&  & (50,75) & * & * & * & * & * & * & * & * & * & * & * & * \\
     &  & (100,100) &$<$1\% / 1&$<$1\% & * & * & \textbf{25730} & \textbf{70.317} & * & * & * & * & * & * \\
     &  & (100,150) & \textbf{79902} / 1& \textbf{513.481} & * & * &+2\% &$<$1\% & * & * & * & * & * & * \\
\\     & 5 & (10,10) & \textbf{405} / 1&+2\% & * & * &+2\% & \textbf{6.591} & * & * &+3527\% &+398\% & * & * \\
     &  & (25,10) & \textbf{851} / 1& \textbf{2.328} & * & * &$<$1\% &+4\% & * & * &+2668\% &+1325\% & * & * \\
     &  & (25,25) &+2\% / 1&+2\% & * & * & \textbf{476} & \textbf{1.164} & * & * & * & * & * & * \\
     &  & (50,50) &$<$1\% / 1& \textbf{4.209} & * & * & \textbf{1126} &$<$1\% & * & * & * & * & * & * \\
     &  & (50,75) & \textbf{1990} / 1& \textbf{9.221} & * & * &+2\% &+3\% & * & * & * & * & * & * \\
     &  & (100,100) &$<$1\% / 1& \textbf{20.906} & * & * & \textbf{1685} &+2\% & * & * & * & * & * & * \\
     &  & (100,150) &+2\% / 1&+2\% & * & * & \textbf{3647} & \textbf{75.915} & * & * & * & * & * & * \\
\\     & 10 & (10,10) & \textbf{584} / 1&+4\% & * & * &$<$1\% & \textbf{8.175} & * & * &+727\% &+359\% & * & * \\
     &  & (25,10) &$<$1\% / 1&+5\% & * & * & \textbf{447} & \textbf{2.401} & * & * &+4366\% &+1524\% & * & * \\
     &  & (25,25) &$<$1\% / 1& \textbf{4.535} & * & * & \textbf{965} &+11\% & * & * &+3532\% &+1866\% & * & * \\
     &  & (50,50) &$<$1\% / 1&+5\% & * & * & \textbf{442} & \textbf{5.887} & * & * & * & * & * & * \\
     &  & (50,75) &$<$1\% / 1&+4\% & * & * & \textbf{1125} & \textbf{9.475} & * & * & * & * & * & * \\
     &  & (100,100) & \textbf{4831} / 1&$<$1\% & * & * &+1\% & \textbf{104.390} & * & * & * & * & * & * \\
     &  & (100,150) & \textbf{2143} / 1& \textbf{65.109} & * & * &$<$1\% &+1\% & * & * & * & * & * & * \\
    \hline\multicolumn{13}{l}{\footnotesize\textit{Bolded values equal to the best algorithm according to iteration count or time. Column ``Time'' is measured in seconds.}}\\
\multicolumn{13}{l}{\footnotesize\textit{* indicates the algorithm failed to find a stationary point meeting the tolerances by the 100,000th iteration.}}\\
\bottomrule
    \end{tabular}
  \end{adjustbox}
  \caption{Performance of {\DAA}, {\DA}, {\ADAPTPE}, {\CONSTPE}, {\CONSTVA}, and {\ADAPTVA} for the Cauchy loss problem.}
  \label{tab:cauchy}
\end{table}

%% file: P3-Table-new.tex
\begin{table}[htb!]
  \centering
  \begin{adjustbox}{max width=\textwidth}
    \begin{tabular}{cc|rrrr|rrrr|rrr}
    \hline
    \multicolumn{2}{c|}{Instance} & \multicolumn{4}{c|}{{\DAA}} & \multicolumn{4}{c|}{{\DA}} & \multicolumn{3}{c}{{\ADAPTPE}} \\
    $\omega$ & $(B,l)$ & Iters & Time & Mults & $f+h$ & Iters & Time & Mults & $f+h$ & Iters & Time & $f+h$ \\
    \cmidrule(lr){1-2} \cmidrule(lr){3-6} \cmidrule(lr){7-10} \cmidrule(lr){11-13} 
1 & (50,20)
 & +74\% & +75\% & 18 & -7.280e+01
 & \textbf{2935} & \textbf{1.422} & 24 & -7.309e+01
 & +1479\% & +1468\% & -7.309e+01 \\
 & (50,40)
 & \textbf{24942} & \textbf{15.872} & 34 & -4.231e+01
 & +2\% & +2\% & 34 & -4.231e+01
 & +12\% & +11\% & -4.263e+01 \\
\\
 & (100,10)
 & +197\% & +283\% & 37 & -1.418e+02
 & \textbf{641} & \textbf{1.407} & 17 & -1.510e+02
 & +379\% & +604\% & -1.431e+02 \\
 & (100,25)
 & \textbf{2739} & \textbf{8.347} & 19 & -2.070e+02
 & +11\% & +4\% & 19 & -2.070e+02
 & * & * & * \\
 & (100,50)
 & +243\% & +245\% & 101 & -1.593e+02
 & \textbf{9113} & \textbf{39.417} & 12 & -1.799e+02
 & +498\% & +476\% & -1.522e+02 \\
 & (100,75)
 & \textbf{69340} & +1\% & 20 & -5.108e+01
 & +1\% & \textbf{373.452} & 20 & -5.108e+01
 & * & * & * \\
\cmidrule(lr){1-2} \cmidrule(lr){3-6} \cmidrule(lr){7-10} \cmidrule(lr){11-13}
10 & (50,20)
 & \textbf{2674} & \textbf{1.310} & 12 & -6.293e+03
 & +668\% & +664\% & 20 & -6.528e+03
 & +89\% & +87\% & -6.379e+03 \\
 & (50,40)
 & \textbf{27770} & \textbf{17.865} & 13 & -5.477e+02
 & +537\% & +535\% & 23 & -6.229e+02
 & +68\% & +67\% & -4.488e+02 \\
\\
 & (100,10)
 & \textbf{446} & \textbf{1.110} & 17 & -1.504e+04
 & +178\% & +152\% & 24 & -1.519e+04
 & +3392\% & +2975\% & -1.463e+04 \\
 & (100,25)
 & +7\% & +7\% & 15 & -4.293e+03
 & \textbf{5735} & \textbf{16.410} & 10 & -4.326e+03
 & * & * & * \\
 & (100,50)
 & \textbf{9170} & \textbf{36.270} & 21 & -6.599e+03
 & +121\% & +121\% & 16 & -6.495e+03
 & * & * & * \\
 & (100,75)
 & \textbf{42857} & \textbf{217.284} & 41 & -5.032e+03
 & +187\% & +187\% & 15 & -3.907e+03
 & * & * & * \\
\cmidrule(lr){1-2} \cmidrule(lr){3-6} \cmidrule(lr){7-10} \cmidrule(lr){11-13}
100 & (50,20)
 & \textbf{2530} & \textbf{1.236} & 18 & -8.944e+05
 & +1\% & +1\% & 17 & -9.056e+05
 & +144\% & +141\% & -8.944e+05 \\
 & (50,40)
 & \textbf{14982} & \textbf{9.710} & 32 & -3.017e+05
 & +65\% & +66\% & 26 & -3.546e+05
 & +128\% & +125\% & -3.580e+05 \\
\\
 & (100,10)
 & \textbf{2148} & \textbf{4.723} & 17 & -8.557e+05
 & +123\% & +123\% & 18 & -7.527e+05
 & +12838\% & +12789\% & -7.892e+05 \\
 & (100,25)
 & \textbf{3474} & \textbf{9.905} & 22 & -9.911e+05
 & +415\% & +416\% & 48 & -9.818e+05
 & +2242\% & +2238\% & -9.911e+05 \\
 & (100,50)
 & +21\% & +21\% & 45 & -8.377e+05
 & \textbf{24711} & \textbf{97.683} & 23 & -9.188e+05
 & * & * & * \\
 & (100,75)
 & \textbf{19436} & \textbf{98.258} & 22 & -6.714e+05
 & +168\% & +168\% & 94 & -6.770e+05
 & +739\% & +739\% & -4.715e+05 \\
\cmidrule(lr){1-2} \cmidrule(lr){3-6} \cmidrule(lr){7-10} \cmidrule(lr){11-13}
1000 & (50,20)
 & \textbf{3143} & \textbf{1.665} & 19 & -2.403e+07
 & +370\% & +330\% & 36 & -2.194e+07
 & * & * & * \\
 & (50,40)
 & $<$1\% & $<$1\% & 73 & -1.556e+07
 & \textbf{85465} & \textbf{54.416} & 75 & -1.556e+07
 & * & * & * \\
\\
 & (100,10)
 & +95\% & +102\% & 22 & -7.742e+07
 & \textbf{1413} & \textbf{3.096} & 17 & -7.958e+07
 & +25263\% & +25244\% & -7.742e+07 \\
 & (100,25)
 & \textbf{4987} & \textbf{14.265} & 45 & -2.093e+08
 & +159\% & +159\% & 23 & -2.088e+08
 & * & * & * \\
 & (100,50)
 & +52\% & +53\% & 42 & -3.153e+07
 & +122\% & +122\% & 20 & -3.135e+07
 & \textbf{18548} & \textbf{73.107} & -3.266e+07 \\
 & (100,75)
 & \textbf{44632} & \textbf{226.221} & 22 & -2.793e+07
 & +120\% & +120\% & 27 & -2.793e+07
 & * & * & * \\
\hline
\multicolumn{13}{l}{\footnotesize\textit{Bolded values equal to the best algorithm according to iteration count or time. Column ``Time'' is measured in seconds.}}\\
\multicolumn{13}{l}{\footnotesize\textit{* indicates the algorithm failed to find a stationary point meeting the tolerances by the 100,000th iteration.}}\\
\bottomrule
    \end{tabular}
  \end{adjustbox}
  \caption{Performance of {\DAA}, {\DA} and {\ADAPTPE} for the nonconvex QP-BC}
  \label{tab:problem3}
\end{table}